%% file: arxiv_revised.tex
\documentclass[11pt]{article}
\usepackage{color}
\usepackage[colorlinks=true,citecolor=blue]{hyperref}
\usepackage{amssymb,amsmath,amsfonts,amsthm}
\usepackage{mathtools}
\usepackage[english]{babel}
\usepackage[utf8]{inputenc}
\usepackage[margin=24mm]{geometry}
\usepackage{graphicx}
\usepackage{times,fourier}
\usepackage[dvipsnames]{xcolor}

\makeatletter
\renewcommand\section{\@startsection {section}{1}{\z@}%
  {-2ex \@plus -1ex \@minus -.2ex}%
  {1ex \@plus.1ex}%
  {\normalfont\bf\sffamily}}
\renewcommand\subsection{\@startsection{subsection}{2}{\z@}%
  {-1.75ex\@plus -0.4ex \@minus -.2ex}%
  {0.6ex \@plus .1ex}%
  {\normalfont\small\bf\sffamily}}
\renewcommand\subsubsection{\@startsection{subsubsection}{3}{\z@}%
  {-0.6ex\@plus -0.2ex \@minus -.2ex}%
  {0.4ex \@plus .1ex}%
  {\normalfont\normalsize\it}}
\renewcommand\paragraph{\@startsection{paragraph}{4}{\z@}%
  {0.2ex \@plus0.2ex \@minus0.1ex}{-0.5em}%
  {\normalfont\normalsize\bfseries}}
\def\ps@headings{%
  \let\@oddfoot\@empty
  \let\@evenfoot\@empty
  \def\@evenhead{\small\sffamily\thepage\hfil\slshape\leftmark}%
  \def\@oddhead{\small\sffamily{\slshape\rightmark}\hfil\thepage}%
  \let\@mkboth\markboth
  \def\chaptermark##1{\markboth{{\ifnum \c@secnumdepth >\m@ne
		\if@mainmatter \@chapapp\ \thechapter. \ \fi \fi ##1}}{}}%
  \def\sectionmark##1{\markright {{\ifnum \c@secnumdepth >\z@
		\thesection. \ \fi ##1}}}}
\def\fbf#1{\setbox0=\hbox{$#1$}\kern-0.10\wd0
  \lower0.02em\copy0\kern-\wd0 \lower0.02em\hbox{\kern+0.04em\copy0}\kern-\wd0
  \raise0.00em\copy0\kern-\wd0 \raise0.00em\hbox{\kern-0.04em\box0}}
\def\overl@ss#1#2{\vcenter{\offinterlineskip
        \ialign{$\m@th#1\hfil##\hfil$\crcr#2\crcr<\crcr } }}
\def\gl{\mathrel{\mathpalette\overl@ss>}}
\makeatother

\advance\textwidth -12mm
\advance\textheight -16mm
\advance\hoffset 6mm
\advance\voffset 6mm

\advance\parskip 4pt
\numberwithin{equation}{section}
1

\newtheorem{theorem}{Theorem}[section]
\newtheorem{definition}[theorem]{Definition}
\newtheorem{corollary}[theorem]{Corollary}
\newtheorem{lemma}[theorem]{Lemma}
\newtheorem{remark}[theorem]{Remark}
\newtheorem{proposition}[theorem]{Proposition}

\def\maketitle{\par\noindent{\LARGE\bf\sffamily\thetitle}\\[1.4ex]
{\large\theauthor}\\[0.6ex]
\textit{\thetextinfo}\\[0.2ex]
{\small\today}\par\vglue1.4\bigskipamount}
\def\title#1{\def\thetitle{#1}}
\def\author#1{\def\theauthor{#1}}
\def\textinfo#1{\def\thetextinfo{#1}}

\def\be{\begin{equation}}
\def\ee{\end{equation}}
\def\bse{\begin{subequations}}
\def\ese{\end{subequations}}
\newtheorem{RHP}{RHP}[section]

\definecolor{deeppurple}{rgb}{0.5, 0, 0.7}

\def\half{{\textstyle\frac12}}

\def\diag{\mathop{\rm diag}\nolimits}

\def\Natural{\mathbb{N}}
\def\Real{\mathbb{R}}
\def\R{\mathbb{R}}
\def\Complex{\mathbb{C}}
\def\I{\mathbb{I}}
\def\Integer{\mathbb{Z}}
\def\Z{\mathbb{Z}}
\def\i{\text{i}}

\def\Re{\mathop{\rm Re}\nolimits}
\def\Im{\mathop{\rm Im}\nolimits}
\def\Res{\mathop{\rm Res}\limits}
\def\tr{\mathop{\rm tr}\nolimits}
\def\re{\mathrm{re}}
\def\im{\mathrm{im}}
\def\d{\mathrm{d}}
\def\sgn{\mathop{\rm sgn}\nolimits}
\def\e{\mathop{\rm e}\nolimits}
\def\@#1{{\mathbf{#1}}}
\def\_#1{{\mathsf{#1}}}
\let\~=\tilde
\let\==\bar
\let\^=\hat
\let\<=\langle
\let\>=\rangle

\def\L{\mathcal{L}}

\def\min{\mathop{\rm min}\nolimits}
\def\max{\mathop{\rm max}\nolimits}

\def\Dir{{\mathrm{Dir}}}
\def\note[#1]{\marginpar{\color{blue}[#1]}}
\def\partialderiv#1#2{{\frac{\partial #1}{\partial #2}}}
\def\txtfrac#1#2{{\textstyle\frac{#1}{#2}}}

\def\C{{\mathbb C}}
\def\R{{\mathbb R}}

\def\Z{{\mathbb Z}}

\def\D{\Delta}
\def\1{{\bf 1}}

\def\z{\zeta}
\def\l{\lambda}

\def\A{\mathcal{A}}
\def\e{\mathrm{e}}
\def\g{g}
\makeatother

\let\trueparagraph=\paragraph
\def\paragraph#1{\par\smallskip\trueparagraph{\rm\textbf{#1}}}
\advance\textheight 12mm
\advance\voffset -6mm
\allowdisplaybreaks

\def\bm{\begingroup\color{magenta}}
\def\em{\endgroup}

\def\bm{\begingroup}

\begin{document}
\pagestyle{plain}

\title{Spectral theory for periodic self-adjoint Dirac\\[0.1ex]
operators and inverse scattering transform for the\\[0.1ex]
defocusing nonlinear Schr\"odinger equation with\\[0.1ex]
periodic boundary conditions}
\author{Gino Biondini and Zechuan Zhang}
\textinfo{Department of Mathematics, State University of New York at Buffalo, Buffalo, New York 14260}
\maketitle

\begin{abstract}
\noindent
The inverse spectral theory for a self-adjoint one-dimensional Dirac operator associated
periodic potentials is formulated via a Riemann-Hilbert problem approach.
The resulting formalism is also used to solve the initial value problem for the nonlinear Schr\"odinger (NLS) equation.
A uniqueness theorem for the solutions of the Riemann-Hilbert problem is established, 
which provides a new method for obtaining the potential from the spectral data.
Two additional, scalar Riemann-Hilbert problems are also formulated that provide conditions for the periodicity in space and time of the solution generated by arbitrary sets of spectral data.
The formalism applies for both finite-genus and infinite-genus potentials.
The formalism also shows that only a single set of Dirichlet eigenvalues is needed in order to uniquely reconstruct 
the potential of the Dirac operator and the corresponding solution of the defocusing NLS equation,
in contrast with the representation of the solution of the NLS equation via the  trace formulae in the finite-genus formalism,
in which two different sets of Dirichlet eigenvalues are used.
\end{abstract}

\medskip
\tableofcontents

\section{Introduction and outline} 
\label{s:intro}

In this work we formulate the direct and inverse spectral theory for the self-adjoint Zakharov-Shabat (ZS) operator 
with periodic potentials,
and we use the results to solve the initial value problem for the defocusing nonlinear Schr\"odinger (NLS) equation
with periodic boundary conditions.
More precisely, we study the eigenvalue problem 
\vspace*{-1ex}
\be
\L\,v = z\,v\,,
\label{e:Diraceigenvalueproblem}
\ee
where the matrix-valued differential operator
\vspace*{-1ex}
\be
\L:= \i\sigma_{3}(\partial_{x} - Q)\,
\label{e:Diracoperator}
\ee
with $v = v(x,z) = (v_1, v_2)^{T} $, the superscript $T$ denoting matrix transpose, 
where $\sigma_3$ is the third Pauli matrix (cf.\ Appendix~\ref{a:notation}), 
and $Q(x)$ is the matrix-valued function
\be
Q(x) = \begin{pmatrix} 0 & q(x) \\ q^*(x) & 0 \end{pmatrix}\,,
\label{e:Q}
\ee
the asterisk denoting complex conjugation,
and the ``potential'' $q: \R \to \C$,
is a function 
in $L^1_\mathrm{loc}(\R)$ 
with minimal period~$L$,
i.e.,
\be
q(x+L) = q(x)\,,\quad \forall x\in\Real\,.
\label{e:Qperiodic}
\ee
The operator $\L$ in~\eqref{e:Diracoperator} 
is a one-dimensional Dirac operator acting in $L^2(\Real,\C^2)$ with dense domain $H^{1}(\Real,\C^2)$.
The Lax spectrum $\Sigma(\L)$ of $\L$ is defined as the set 
\be
\Sigma(\L) := \big\{z\in\C: \mathcal L\phi=z\phi,\,\, 0<\|\phi\|_{L^{\infty}(\R;\C^2)} < \infty \big\}\,,
\label{e:laxspec}
\ee
i.e., the set of complex numbers $z$ such that \eqref{e:Diraceigenvalueproblem} has at least one bounded nonzero solution. 
It can be proved that for $q$ locally integrable the Lax spectrum defined above equals the spectrum of the maximal operator associated with $\mathcal L$ in $L^2(\R;\C^2)$, the space of square-integrable two-component vector-valued functions,
namely the set $\{z\in\Complex : z\notin\rho(\mathcal L)\}$, where $\rho(\mathcal L)$ is the resolvent set of $\mathcal L$ \cite{rofebek}. 
Throughout this work, we will use the term ``spectrum'' as a synonym for the Lax spectrum. 
Also, to avoid technical complications, we will require $q\in C^2([0,L])$ for simplicity,
unless explicitly stated otherwise.


The eigenvalue problem~\eqref{e:Diraceigenvalueproblem} is of fundamental interest in the theory of completely integrable 
nonlinear evolution equations, 
since it comprises the first half of the Lax pair for the defocusing NLS equation, namely the partial differential equation (PDE)
\be
\i q_t + q_{xx} - 2 |q|^2 q = 0\,,
\label{e:nls}
\ee 
with $q = q(x,t)$, 
and where subscripts $x$ and $t$ denoting partial differentiation.
Indeed, Zakharov and Shabat \cite{ZS1972,ZS1973} showed that \eqref{e:nls} is the compatibility condition of the following
overdetermined linear system of ordinary differential equations (ODEs)
\vspace*{-1ex}
\bse%
\label{e:NLSLP}
\begin{gather}
v_x = (-\i z\sigma_3 + Q(x,t))\,v\,,
\label{e:zs}
\\
v_t = ( -2\i z^2\sigma_3 + H(x,t,z))\,v\,,
\label{e:NLSLP2}
\end{gather}
\ese
where $v = v(x,t,z)$,
with 
$H(x,t,z) = 2zQ - \i\sigma_3(Q^2-Q_x)$ and 
$Q(x,t)$ has the same form as in~\eqref{e:Q}.
For this reason, \eqref{e:zs} is referred to as the Zakharov-Shabat scattering problem,
$z\in\C$ as the scattering (or spectral) parameter,
and $v$ as the scattering eigenfunction.
It is trivial to see~\eqref{e:zs} is equivalent to the eigenvalue problem~\eqref{e:Diraceigenvalueproblem}.
Together, \eqref{e:zs} and~\eqref{e:NLSLP2} comprise the Lax pair of the NLS equation~\eqref{e:nls}.

In 1967, Gardner, Greene, Kruskal and Miura pioneered the use of direct and inverse spectral methods to solve the 
initial value problem (IVP) for the Korteweg-de\,Vries equation \cite{GGKM},
whose scattering problem is the time-independent Schr\"odinger equation.
Then, using a similar approach, in \cite{ZS1972,ZS1973} Zakharov and Shabat 
showed how the Lax pair~\eqref{e:NLSLP} can be used to solve the IVP for the NLS equation~\eqref{e:nls}
via a technique which is now referred to as the inverse scattering transform (IST),
in which the direct and inverse spectral theory of the Zakharov-Shabat problem~\eqref{e:zs} play a crucial role.
Shortly afterwards, similar ideas were shown by Ablowitz, Kaup, Newell and Segur to also apply to a rather large class of integrable nonlinear PDEs \cite{AKNS1974}.
Various significant further developments of the IST then followed 
(e.g., see \cite{AC1991,APT2004,AS1981,BDT1988,DeiftTrubowitz,FT1987,Konopelchenko,NMPZ1984}),
a key one among them being the reformulation of the inverse problem
(i.e., the reconstruction of the potential from the scattering data) via a matrix Riemann-Hilbert problem (RHP) \cite{AC1991,BealsCoifman,Deift1998,Its2003,NMPZ1984,Zhou1989},
which in turn made it possible to formulate powerful asymptotic techniques such as the Deift-Zhou method \cite{DeiftZhou1991},
which has been used with great success in a variety of different settings.

Following the development of the direct and inverse spectral theory and IST for localized potentials 
(i.e., potentials decaying rapidly as $x\to\pm\infty$, 
the precise conditions dependending on the particular PDE under consideration), 
a natural next step was the study of IVP with periodic boundary conditions.
The direct and inverse spectral theory for Hill's equation
(i.e., the time-independent Schr\"odinger equation with periodic potentials)
and its application to the KdV equations with periodic boundary conditions
has of course a long and distinguished history, going back to the early 1970s, 
related to the so-called finite-genus formalism, 
e.g., see 
\cite{BBEIM,Dubrovin1975,ItsMatveev1,ItsMatveev2,ItsMatveev3,KappelerPoschel2010,Lax1975,MW1966,Marchenko1974,MarchenkoOstrovsky1975,Matveev2008,McKeanVanMoerbeke,Novikov1974,NMPZ1984}.
The theory was also extended to the case of infinite genus in \cite{KappelerPoschel2010,McKeanTrubowitz1,McKeanTrubowitz2},
including the construction of global coordinate transformations which map the KdV flow to a periodic flow on an infinite-dimensional torus.
The direct and inverse spectral theory of focusing and defocusing Zakharov-Shabat operators
and in particular their finite-gap formalism
also have a long history
\cite{GesztesyHolden,gesztesyweikard_acta1998,gesztesyweikard_bams1998,Harutyunyan,ItsKotlyarov1976,Kotlyarov1976,Levitan1975,Levitan1991,MaAblowitz,McKean1981,McLaughlinOverman}.
In particular, it is well known that, for the self-adjoint Zakharov-Shabat spectral problem~\eqref{e:Diraceigenvalueproblem} with periodic potentials,
$\Sigma(\mathcal L)$ is purely continuous \cite{danilov,Eastham,hislop,rofebek}
(that is, essential without any eigenvalues and empty residual spectrum),
and it is comprised of an at most countable collection of segments 
referred to as spectral bands, also often referred to as stability bands.
All such bands lie along the real $z$-axis, separated by spectral gaps (also often referred to as instability bands),
with the endpoints of the gaps lying at interlaced periodic and antiperiodic real eigenvalues.


Despite its undeniable success, the finite-genus approach also has certain drawbacks.
For infinite gap potentials, computation of the map from the spectral variables to the infinite-genus torus requires solving 
an infinite-genus Jacobi inversion problem,
and the infinite-dimensional Riemann period matrix must be explicitly computed in order to use the Its-Matveev formula.
Anothed unsettled question is how effective is the finite-gap approach for the solution of the IVP for general periodic initial data. 
Finally, even in the case of finite genus, the fact that the finite-gap integration method is so different from the IST on the line
is somewhat puzzling.

Recently, McLaughlin and Nabelek developed a Riemann-Hilbert approach to solve the inverse spectral problem for 
Hill's operator in \cite{McLaughlinNabelek}.
Besides avoiding the above-mentioned technical difficulties arising from the finite-genus theory, 
the novel formulation of the inverse problem makes it possible for the first time to employ 
all the machinery available for Riemann-Hilbert problems
in order to study a variety of physically interesting problems.
\bm
The direct spectral theory for periodic self-adjoint operators, including first order systems, is classical, and is naturally formulated in terms of Floquet theory, the monodromy matrix, and the resulting band-gap decomposition of the spectrum, see \cite{Levitan1975,Levitan1991}. Building on this framework, 
in this work we generalize the approach of~\cite{McLaughlinNabelek} 
to formulate the inverse spectral theory for the self-adjoint Dirac operator~\eqref{e:Diracoperator} with periodic potentials 
via a Riemann-Hilbert problem formalism \em
and we use the resulting formalism to solve the initial value problem for the defocusing NLS equation~\eqref{e:nls} with periodic BC.
In the process, we will show that a single set of Dirichlet eigenvalues is sufficient to completely determine the potential 
from the inverse problem.
This result is in sharp contrast to the reconstruction formula for the potential from the finite-genus formalism, 
in which two different sets of Dirichlet eigenvalues appear (e.g., see~\cite{ForestLee,MaAblowitz,McLaughlinOverman}).

Specifically, the outline of this work is as follows.
In section~\ref{s:spectrum} we begin by recalling some well-known results from Bloch-Floquet thoery,
and we also introduce a similarity transformation of the fundamental solution of the scattering problem
that leads to the definition of the Dirichlet spectrum.
We then explicitly define the spectral data that will allow the reconstruction of the potential in the inverse problem.
In section~\ref{s:asymptotics} 
we define modified Bloch-Floquet solutions,
we compute the asymptotics of all relevant quantities as $z\to\infty$,
and we write down infinite product expansions of $\~y_{12}(L,z)$, $\D(z)-1$ and $\D(z)+1$.
All these results are used to formulate the inverse problem as a suitable matrix Riemann-Hilbert problem in section~\ref{s:inverse},
where we also show how the potential can be recovered from the solution of the RHP. 
A key step in the process is the introduction of a matrix function $B(z)$, uniquely determined by the spectral data, 
which serves two purposes: 
first, to ensure that the solution of the RHP is unimodular, 
and, second, to eliminate the singularities of the RHP.  
Using $B(z)$ and the modified Bloch-Floquet solutions,
we then define a matrix-valued function $\Phi(x,z)$ that satisfies RHP~\ref{RHP1}.
We then also prove this RHP to have a unique solution.
In section~\ref{s:time} we use the results of sections~\ref{s:spectrum}--\ref{s:inverse} to establish a Riemann–Hilbert problem characterization 
for the solution of the initial value problem for the defocusing NLS equation with periodic, smooth, infinite-gap initial conditions. 
Namely, we compute the time evolution of the spectral data and study the evolution of the normalized 
Bloch–Floquet solutions to the Dirac equation. 
In section~\ref{s:periodicity}, we use the results of section~\ref{s:inverse} and section~\ref{s:time} to obtain conditions for the spatial and temporal periodicity of the solutions to the defocusing NLS equation.
In section~\ref{s:finitegap} we connect the uniqueness results for the solutions to RHP \ref{RHP1} and RHP \ref{RHP2} to a uniqueness results for the corresponding spatially periodic infinite gap Baker–Akhiezer functions, 
and we discuss the relation of the infinite gap theory discussed in this paper to the well-known finite gap theory.
Finally, in section~\ref{s:conclusions} we end this work with some final remarks and a discussion of some open questions motivated
by the results of this work.
Some technical results, various proofs, an explicit example 
and a few other considerations are relegated to the appendix.

\section{Direct spectral theory for the periodic self-adjoint Zakharov-Shabat problem}
\label{s:spectrum}

In this section we begin formulating the direct spectral theory for the defocusing ZS spectral problem.
Since the time dependence of the potential does not play any role in the direct and inverse spectral theory,
in this section and the following ones we will temporarily omit the time dependence, 
which will then be restored when discussing the IVP for the NLS equation in section~\ref{s:time}.

We begin by briefly recalling some well-known results from Bloch-Floquet theory.
These will be used to then define various necessary quantities as well as the spectral data that will allow us in section~\ref{s:inverse} 
to uniquely reconstruct the potential.
Recall that, for $A\in L^1_{{\rm loc}}(\R)$ an $n \times n$ matrix-valued function with $A(x+L) = A(x)$,
Floquet's theorem~\cite{Floquet_int,Eastham,Floquet} states that
any fundamental matrix solution $Y(x)$ of \eqref{e:yprime} of the system of linear homogeneous ODEs 
\be
y_x = A(x)\,y
\label{e:yprime} 
\ee
can be written in the Floquet normal form
\be
Y(x) = W(x)\,\e^{Rx}\,,
\label{e:canonical}
\ee
where $W(x)$ is a nonsingular matrix with $W(x+L) = W(x)$, and $R$ is a constant matrix. 
Thus, all bounded solutions of the ZS system~\eqref{e:zs} have the form
\be
v(x,z) = \e^{\i\nu x} w(x,z)\,,
\label{e:bounded}
\ee
where 
$w(x+L,z) = w(x,z)$,
and $\nu\in\Real$ is the Floquet \bm exponent, \em or quasi-momentum. 
One also defines the so-called Bloch-Floquet solutions, or normal solutions, 
as the solutions of \eqref{e:zs} such that 
\be
\psi(x+L,z) = \rho\,\psi(x,z)\,,
\label{e:bdsol}
\ee 
where $\rho$ is the Floquet multiplier. 
Thus, 
a solution of~\eqref{e:zs} is bounded if and only if $|\rho|=1$, 
in which case $\rho= \e^{\i\nu L}$ with $\nu\in\Real$.
Moreover, 
the Floquet multipliers are the eigenvalues of the monodromy matrix $M(z)$,  
which is defined by
\be
Y(x+L,z)=Y(x,z)M(z)\,, 
\label{e:monodromy}
\ee
where $Y(x,z)$ is any fundamental matrix solution of~\eqref{e:zs}.
Hereafter, we choose $Y(x,z)$ as the principal matrix solution of~\eqref{e:zs}
that is,
the matrix solution of~\eqref{e:zs} normalized so that $Y(0,z)\equiv\mathbb{I}$,
where $\I$ is the $2\times 2$ identity matrix.
We then have
\be
M(z) = Y(L,z)\,.
\label{e:monodromy2}
\ee
Standard techniques allow one to shows that, under the above assumptions, 
$Y(x,z)$ can be expressed as the following Volterra integral equation
\be
Y(x,z)=\e^{-\i z\sigma_3x}+\int_0^x\e^{-\i z\sigma_3(x-s)}Q(s)Y(s,z)\,\d s\,, 
\label{phi}
\ee
which also allows one to show that, for all $x\in\Real$, $Y(x,z)$ is an entire function of $z$.
Since the RHS of ~\eqref{e:zs} is traceless, 
Abel's formula implies $\det M(z) \equiv 1$.  
Hence the the Floquet multipliers, i.e., the eigenvalues of $M(z)$, are given by roots of the quadratic equation
\vspace*{-0.6ex}
\be
\rho^2 - 2\D(z)\rho+ 1 = 0,
\label{tracepoly}
\ee
where $\D(z)$ is the Floquet discriminant
\be
\D:=\D(z)=\half\tr M(z)\,,
\label{e:DeltaM}
\ee
which is also is an entire function of $z$ \cite{ForestLee,MaAblowitz, McLaughlinOverman}.
Moreover, the Schwarz symmetry of the ZS problem implies that $\D(z)$ satisfies a Schwarz reflection principle:
$\D(z^*)=\D^*(z)$.
As a result, $\D(z)$ is real-valued along the real $z$-axis.
Accordingly, for $z\in\Real$ the following possible cases arise:
(i)~if $\Delta^2(z)>1$, the Floquet multipliers are real and the two Bloch eigenfunctions are unstable
(i.e., they diverge either as $x\to\infty$ or as $x\to-\infty$).
(ii)~if $\Delta^2(z)<1$, the Floquet multipliers are complex conjugate, have unit magnitude and the two Bloch eigenfunctions are stable (i.e., bounded for all $x\in\Real$).
(iii)~if $\Delta^2(z)=1$, the Floquet multipliers $\pm 1$ and at least one of the Bloch eigenfunctions is periodic or antiperiodic.
Thus, an equivalent representation of the Lax spectrum is:
\be
\Sigma(\mathcal L) = \{ z \in \Complex : \D(z) \in [-1,1]\}.
\label{e:laxspec2}
\ee

For the development of the direct and inverse spectral theory, 
it is convenient to have an explicit representation of the Floquet multipliers for all $z\in\Complex$ as follows.
For $z\in\Complex$, denote the roots of~\eqref{tracepoly} as
\vspace*{-0.4ex}
\be
\rho_{1,2}(z) = \D(z) \mp \sqrt{\D^2(z)-1},
\label{rho}
\ee
for some appropriate choice of the complex square root, 
to be defined next.
Obviously $\rho_{1,2}(z)$ satisfy the relation $\rho_1(z) = 1/\rho_2(z)$.
For $z\in\Sigma(\L)$, $\D(z)\in[-1,1]$ implies $|\rho_{1,2}(z)|=1$, whereas 
for $z\in\Complex\setminus\Sigma(L)$, $\D(z)\notin[-1,1]$ implies $|\rho_{1,2}(z)|\ne1$. 

\begin{remark}
Let $\rho(z)$ be the root of~\eqref{tracepoly} that is holomorphic for $z\notin\Sigma(\L)$ and such that 
$|\rho(z)|<1$ for $z\in\C\setminus\Sigma(\L)$.
Setting $\rho(z) = \rho_1(z)$ $\forall z\in\Complex$ uniquely defines the complex square root $\sqrt{\D^2(z)-1}$ so that:%
\begingroup
\def\item[#1]{#1}
\item[(i)]
its branch cut coincides with $\Sigma(\L)$,
\item[(ii)] 
$\sqrt{\D^2(z)-1}>0$ for all $z\in\Real\setminus\Sigma(\L)$.
\item[(iii)]
$\sqrt{\D^2(z)-1}$ is continuous from above for all $z\in\Sigma(\L)$.
\endgroup
\end{remark}

The \textit{main spectrum} $\{\z_j\}_{j\in\Integer}$ of the ZS problem~\eqref{e:zs} is defined as the set
$\{\z\in\C: \Delta^2(\z) = 1\}$.
That is, the main spectrum comprises the eigenvalues $\z_j$ for which at least one eigenfunction is periodic 
($\D(\z) =1$, such that $\rho(\z)=1$) or antiperiodic
($\D(\z) = - 1$, such that $\rho(\z)= -1$).   
Therefore we can represent the Lax spectrum as $\Sigma(\L)=\cup_{j\in\Integer}[\z_{2j},\z_{2j+1}]$.

It is well known that, for the ZS problem with periodic potentials,
knowledge of the main spectrum is not sufficient to uniquely recover the potential, 
and as a result one must introduce auxiliary spectral data in the form of Dirichlet eigenvalues.
(The same is also true for the direct and inverse spectral theory for Hill's operator.)
Following \cite{MaAblowitz,McLaughlinOverman},
we therefore define the \textit{Dirichlet spectrum} associated with \eqref{e:zs} as follows:
\be
\Sigma_{\Dir}(x_o)
:=\{\mu \in \C : \exists~v \not\equiv 0 \in H^1([x_o,x_o+L],\C^2)~ ~\text{s.t.}~ ~ \L\,v = \mu\,v ~ \wedge~ v\in {\rm BC}_{{\rm Dir},x_o}
\}\,,
\label{e:dirichlet}
\ee
where ``BC$_{{\rm Dir},x_o}$'' denote the following Dirichlet boundary conditions (BC) with base point $x_o$:
\be
\label{e:Dirbcs}
v_1(x_o)+v_2(x_o)=v_1(x_o+L)+v_2(x_o+L)=0\,,
\ee
where $v = (v_1,v_2)^T$.
Any point $\mu\in\Sigma_{\Dir}(x_o)$ will be referred to as a Dirichlet eigenvalue of~\eqref{e:Diraceigenvalueproblem}.
Similarly to the Floquet spectrum,
one can identify $\Sigma_{\Dir}(x_o)$ with the zero set of a suitable entire function.
For the ZS problem~\eqref{e:zs}, however,
additional complications arise compared to Hill's operator.
For this reason, 
we introduce the following similarity transformation, which will be instrumental not only for characterizing the Dirichlet spectrum,
but also for carrying out the inverse spectral theory:
\be
\~Y(x,z)=U\,Y(x,z)\,U^{-1},
\label{e:Ytildedef}
\ee
where
\be
U = \bm\frac1{\sqrt2}\em\begin{pmatrix} 1 & 1 \\ -\i & \i \end{pmatrix}.
\label{e:Udef}
\ee
Then $\~Y(x,z)$ is the fundamental solution of the following modified scattering problem:
\be
\~y_x = U(-\i z\sigma_3 + Q)U^{-1}\,\~y\,.
\label{e:modZS}
\ee
For convenience, also let 
$\tilde{M}(z)=\~Y(L,z)$,
which is also an entire function of~$z$ like $M(z)$.
Moreover, since the trace and determinant are invariant under the transformation~\eqref{e:Ytildedef},
the Floquet eigenvalues of the modified scattering problem~\eqref{e:modZS},
i.e., the eigenvalues of $\~M(z)$, coincide with $\rho(z)$.
On the other hand, we now have

\begin{proposition}
\label{p:dirzeros}
The Dirichlet spectrum $\{\mu_j\}$ with base point $x_o=0$ coincides with the set
\be
\Sigma_{\Dir}(0)= \left\{\mu \in \C : \~y_{12}(L,\mu) = 0\right\}.
\label{e:dirzeros}
\ee
\end{proposition}

\begin{proof}
Any vector solution of~\eqref{e:Diraceigenvalueproblem} can be expressed as
$v(x,z)= c_1 y_1(x,z) + c_2 y_2(x,z)$ for suitable constants $c_{1,2}$. 
Evaluating this expression at $x=0$ and imposing $v_1(0)+v_2(0)=0$, we obtain $c_2=-c_1$.  
Then, evaluating it at $x=L$ and imposing $v_1(L)+v_2(L)=0$, we obtain 
$c_2(M_{11}+M_{21}-M_{12}-M_{22})=0$. 
Since $c_2$ cannot be zero, we obtain that the Dirichlet boundary conditions \eqref{e:Dirbcs}
are equivalent to the condition
$M_{11}+M_{21}-M_{12}-M_{22}=0$, i.e., $\~y_{12}(L,z)=0$ by~\eqref{e:Ytildedef}.
\end{proof}

\begin{figure}[b!]
\bigskip
\centerline{\includegraphics[width=0.95\textwidth]{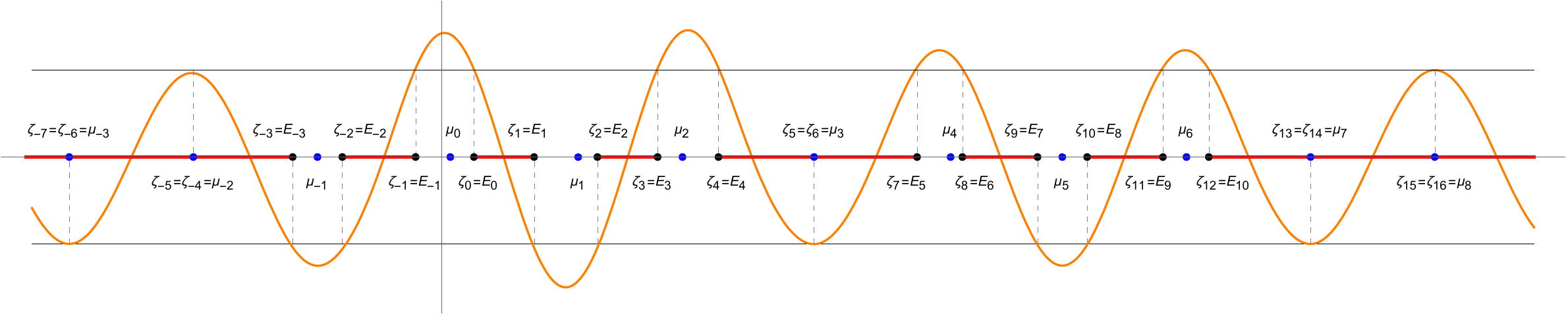}}
\caption{Schematic representation of the behavior of the Floquet discriminant $\D(z)$ (vertical axis) as a function of $z$
(horizontal axis), 
together with the periodic and antiperiodic spectrum, the spectral bands (in red) and gaps and the Dirichlet eigenvalues (blue dots).
See text for further details.}
\label{f:1}
\end{figure}

The formulation of the inverse spectral theory in section~\ref{s:inverse}
will require the use of some known results on forward spectral theory, 
which are summarized for convenience in the following theorem
\cite{djakovmityagin,grebertkappelermityagin,kappeler,McLaughlinOverman,MaAblowitz}.
(A proof of some of these results is also given in Appendix~\ref{a:Dirichlet}.)
\begin{theorem}
\label{spectraprop}
The main spectrum and Dirichlet spectrum of the Dirac operator $\L$~\eqref{e:Diracoperator} with $L$-periodic potential 
$q\in L^2(\R,\C)$ satisfy the following properties:
\begin{enumerate}
\item 
The main spectrum, $\{\z_j\}_{j\in\Integer}$ is real and can be divided into nondegenerate band edges $E_j$ 
and degenerate band edges $\hat{\z}_j$, at which one has 
$\Delta'(E_j)\ne0$ and $\Delta'(\hat{\z}_j)=0$, respectively.
Each $\z_j$ is either a simple or a double root of of $\Delta^2(z)-1=0$, but not of higher order.
\item 
The Dirichlet spectrum, $\{\mu_j\}_{j\in\Integer}$, is real, and each $\mu_j$ lies inside an unstable band or on a band edge. Moreover, there is exactly one Dirichlet eigenvalue $\mu_j$ in each unstable band.
All the $\mu_j$ are simple roots of $\~y_{12}(L,z)$. 
\item 
$\Delta'(z)\neq0$ for all $z\in\Real$ such that $-1<\Delta(z)<1$.
\item 
The periodic, antiperiodic and Dirichlet eigenvalues $\z_j$ and $\mu_j$  have the following asymptotic behavior as $|j|\to\infty$:
\be
\z_{2j}, \z_{2j-1} = \frac{\pi j}{L}+O(1/j),\qquad
\mu_j =\frac{\pi j}{L}+O(1/j).
\label{asymmuz}
\ee
\end{enumerate}
\end{theorem}

For simplicity of presentation, 
we will assume that $\z_0$ is the first point of the main spectrum that is greater than or equal to~$0$,
with $z=0$ lying in a degenerate or non-degenerate spectral gap, 
and that $\z_0$ and $\z_{-1}$ are the zeros of $\Delta-1$, so that we can order the $\z_j$ as
\be
\cdots <\z_{-2j}\leq\z_{-2j+1}<\cdots<{\z}_{-1}\leq{\z}_0<{\z}_1\leq\cdots<\z_{2j}\leq\z_{2j+1}<\cdots
\ee
(cf.\ Fig.~\ref{f:1}).
We also number the Dirichlet eigenvalues so that $\z_{2j-1}\le \mu_{j} \le\z_{2j}$. 
%
If $z=0$ lies a in spectral band instead of a spectral gap, 
we set can simply set $\zeta_{-1}$ to be the first point of the main spectrum that is greater to or equal than zero. 
The Lax spectrum can still be expressed as above, and the whole formalism remains unchanged.
(Moreover, if $\z_{-1}>0$, one need not worry about the possibility that $\mu_0=0$ in section~\ref{s:inverse}.)
Similarly, if $z=0$ 
lies in a spectral gap but $\zeta_1$ and $\zeta_0$ are zeros of $\D(z)+1$ instead of $\Delta(z)-1$, 
one just needs to switch the expressions for the infinite products in~\eqref{e:d-1} and \eqref{e:d+1} in section~\ref{s:inverse}.

\begin{definition}
For each Dirichlet eigenvalue $\mu_k$, we define $\nu_k=-\sgn(\log|\~y_{22}(L,\mu_k)|)$.
\end{definition}

\begin{lemma}
\label{y2}
When $z = \mu_j$, 
one has 
$\~y_{22}(L,\mu_j) = 1/\~y_{11}(L,\mu_j) = \rho^{\nu_j}(\mu_j)$, and 
the vector-valued fundamental solution $\~y_2(x,\mu_j)$ is a Bloch-Floquet solution of~\eqref{e:modZS}
with Floquet multiplier  $\~y_{22}(L,\mu_j)$.
\end{lemma}

\begin{proof}
The Dirichlet eigenfunction  $\~y_2(x,\mu_j)$ solves the modified ZS problem with spectral parameter $z=\mu_j$, 
and so does $\~y_2(x+L,\mu_j)$.  
So we have $\~y_2(x+L,\mu_j)=a\~y_1(x,\mu_j)+b\~y_2(x,\mu_j)$ for some constants $a$ and $b$. 
Evaluating this expression at $x=0$, we conclude $a=0$ and $b=\~y_{22}(L,\mu_j)$. 
Therefore, by Floquet's theorem, $\~y_2(x,\mu_j)$ is a Bloch-Floquet solution with multiplier $\~y_{22}(L,\mu_j)=\rho(\mu_j)$  or $\~y_{22}(L,\mu_j)=\rho(\mu_j)^{-1}$. Since $\det \~Y(L,\mu_j)=1$, we also have $\~y_{22}(L,\mu_j)= 1/\~y_{11}(L,\mu_j)$.
\end{proof}

We are now ready to define the spectral data that will eventually determine uniquely the potential~$q(x)$.
For simplicity we exclude the trivial potential $q(x)\equiv0$ from consideration.
\begin{definition}
(Non-degenerate gaps)
Fix $\g_-$ and $\g_+$ with $\g_- < \g_+$ so that 
$\{E_{2\g_--1},\dots,E_{2\g_+}\}$ is the subsequence of periodic/antiperiodic eigenvalues $\{{\z}_j\}_{j\in\Integer}$ 
for which the spectrum 
$\Sigma(\L)$ is the disjoint union 
\vspace*{-1ex}
\be
\Sigma(\L) = (-\infty, E_{2\g_--1}]\cup[E_{2\g_+},\infty)
  \bigcup_{j=\g_-}^{\g_+-1}[E_{2j},E_{2j+1}],
\ee
where each of the numbers $\g_\pm$ is either finite or countably infinite.
The finite-length intervals\break $(E_{2\g_--1},E_{2\g_-}),\dots,(E_{2\g_+-1},E_{2\g_+})$
comprise the spectral gaps of the potential.
\end{definition}

\begin{remark}
When $\g_-$ and $\g_+$ are finite, their value denotes the smallest and largest values of $j$ for which the corresponding
spectral gaps are non-empty.  
In other words, 
$\{E_{2\g_--1},\dots,E_{2\g_+}\}$ is the subset of the periodic/antiperiodic eigenvalues $\{\z_j\}_{j\in\Integer}$
that are associated to open spectral gaps.
Note that the set of open spectral gaps might not necessarily be consecutive.
\end{remark}

If both $\g_\pm$ are finite, 
the total number of gaps is $\g_+ - \g_- + 1$, and 
the genus of the potential (and the associated solution of the defocusing NLS equation) is $\g = \g_+ - \g_- $.
(That is, the number of spectral gaps is $g+1$.)
The case $\g=0$ (i.e., $\g_+=\g_-$, implying that there is only one gap and no finite spectral band) 
yields a constant or plane-wave potential,
while the case $\g=1$ (i.e., two gaps and one finite band) yields the elliptic potentials of the NLS equation.
Note that the above definition only identifies the indices $\g_\pm$ up to an identical arbitrary additive constant.
We used the freedom to choose this constant when definining $\z_0$ above.

\begin{definition}(Fixed and movable Dirichlet eigenvalues)\label{defgamma}
We denote by $\{\gamma_{\g_-},\dots,\gamma_{\g_+}\}$ 
the subsequence of Dirichlet eigenvalues that lie in the non-degenerate spectral gaps, numbered so that 
${E}_{2k-1} \le \gamma_k \le {E}_{2k}$.
These Dirichlet eigenvalues are movable, i.e., they depend on the base point $x_o$ 
(here chosen to be zero)
and also on time (cf.\ section~\ref{s:time}).
The remaining eigenvalues are independent of the base point, and are called the fixed Dirichlet eigenvalues.
\end{definition}

\begin{definition}
(Spectral data)
\label{specdata}
We define the spectral data associated to the potential $q$ as the set
\be
S(q):=\{{E}_{2k-1},{E}_{2k},\gamma_k,\nu_k\}_{k=\g_-,\dots,\g_+}.
\label{e:spectraldata}
\ee
\end{definition}

\begin{remark}
Lemma~\ref{y2} implies that, for all $k\in\Integer$:
(i)
If $\nu_k=\pm1$, $\~y_2(x,\mu_k)$ is a Bloch-Floquet solution associated with the Floquet multiplier
$\rho^{\pm1}(\mu_k)$, respectively.
(ii)
If $\nu_k=0$, $\mu_k$ lies on one of the edges of a (degenerate or non-degenerate) spectral gap.
Thus, $\nu_k=0$ for all fixed Dirichlet eigenvalues,
whereas $\nu_k$ can equal $\pm1$ or 0 for the movable Dirichlet eigenvalues.
\end{remark}

\begin{remark}
In section~\ref{s:inverse} we will show that the spectral data~\eqref{e:spectraldata} defined by the main spectrum and 
the single set of Dirichlet eigenvalues of~\eqref{e:zs} 
enables the unique reconstruction of the potential~$q$.
The result is in contrast with the trace formulae obtained from the finite-genus formalism
(e.g., see~\cite{MaAblowitz,McLaughlinOverman}),
for which two sets of Dirichlet eigenvalues are used.
(See Appendix~\ref{a:alternative} for further discussion.)
\end{remark}

\section{Modified Bloch-Floquet solutions}
\label{s:asymptotics}

Some additional results are needed before we can begin to formulate the inverse spectral theory.
Specifically, as usual one needs knowledge of the asymptotic behavior of various relevant quantities as $z\to\infty$.
It is straightforward to show the following (see Appendix~\ref{a:Ytildeasymp} for details): 
\begin{lemma}
\label{l:asymY}
If $q\in C^2([0,L],\Complex)$, then for all $z\in\Complex$ the fundamental matrix $Y(x,z)$ has
the following asymptotic behavior as $z\to\infty$: 
\bm\be
\label{asymY}
Y(x,z)= \left[I-\frac{1}{2\i z}\sigma_3 K[q](x)+\frac{1}{2\i z}
\begin{pmatrix}
   0  & q(x)-q(0)\e^{-2\i zx}\\
    q^*(0)\e^{2\i z x}-q^*(x) & 0
\end{pmatrix}
\right]\e^{-\i zx\sigma_3}
\,,
\ee
where
\be
K[q](x)=\int_0^x|q(s)|^2\d s\,.
\ee\em
\end{lemma}

\begin{remark}
The off-diagonal terms of $Y(x,z)$ at $O(1/z)$ contain the value of the potential $q(x)$ at $x=0$ (as well as its complex conjugate).
As a result, $Y(x,z)$ is not sufficient for the formulation of an inverse spectral theory.
As with Hill's equation \cite{McLaughlinNabelek}, 
one uses the fundamental solution to define suitable Bloch-Floquet eigenfunctions instead.
In the case of the Zakharov-Shabat problem~\eqref{e:zs}, however, a further step needed is the similarity transformation 
from 
$Y(x,z)$ to $\~Y(x,z)$
introduced in~\eqref{e:Ytildedef}.
\end{remark}

Hereafter, for simplicity we will always assume that $q\in C^2([0,L],\Complex)$ unless explicitly stated otherwise.
It is straightforward to show that the asymptotic behavior of $Y(x,z)$ in Lemma~\ref{l:asymY} implies the following:
\begin{proposition}
\label{asymD}
The Floquet discriminant $\D(z)$ has the following asymptotic behavior as $z\to\infty$:
\vspace*{-0.6ex}
\begin{gather}
\Delta(z)=\half\e^{\mp\i zL}\left(1\mp\frac{1}{2\i z}\int_0^L|q(x)|^2\d x+O\left(\frac{1}{z^2}\right)\right),
\qquad 
z\in\C^\pm\,.
\end{gather}
\end{proposition}

Next we introduce a particular family of Bloch-Floquet solutions of the modified ZS problem~\eqref{e:modZS},
which will be instrumental in the formulation of the inverse problem:
\begin{definition}
Let $\psi^\pm(x,z)$ be the Bloch-Floquet solutions of the modified ZS problem~\eqref{e:modZS} 
uniquely determined by the conditions 
\be
\psi^{\pm}(x+L,z)=\rho^{\pm1}(z)\psi^{\pm}(x,z)\,,\qquad
\psi_1^{\pm}(0,z)=1.
\label{e:psipmdef}
\ee
Also, let $\Psi(x,z)$ be the matrix Bloch-Floquet solution defined as
\vspace*{-1ex}
\be
\label{e:defPsi}
\Psi(x,z) = \big( \,\psi^\mp\,, \,\psi^\pm\, \big)\,,\qquad z\in\Complex^\pm\setminus\Real\,.
\ee
\end{definition}

\noindent 
In Appendix~\ref{a:Dirichlet} and Appendix~\ref{asymofpsipm}
we prove the following:
\begin{proposition}
\label{p:BFyexpansion}
The Bloch-Floquet solutions~ $\psi^\pm(x,z)$ defined by~\eqref{e:psipmdef} are given by
\be
\psi^{\pm}(x,z)=\~y_1(x,z)+\frac{\rho^{\pm1}(z)-\~y_{11}(L,z)}{\~y_{12}(L,z)}\~y_2(x,z),
\label{BF}
\ee
where $\~Y(x,z) = (\~y_1,\~y_2)$.
\end{proposition}

\begin{proposition}
\label{p:asympsi}
For all $x>0$, $\psi^-(x,z)$ has the following asymptotic behavior as $z\to\infty$
\bse
\label{asympsi-}
\begin{align}
    &\psi^-(x,z)=\e^{-\i zx}\left(\begin{array}{c}
    1-\frac{1}{2\i z}\int_0^x|q(s)|^2\d s-\frac1{2\i z}(q^*(x)-q^*(0))+O(1/z^2) \\ -\i+\frac{1}{2z}\int_0^x|q(s)|^2\d s-\frac1{2z}(q^*(x)+q^*(0))+O(1/z^2)\end{array}\right)\,,\quad 
    z\in\C^+\,,
    \\
    &\psi^-(x,z)=\e^{\i zx}\left(\begin{array}{c}1+\frac{1}{2\i z}\int_0^x|q(s)|^2\d s+\frac1{2\i z}(q(x)-q(0))+O(1/z^2)\\\i+\frac{1}{2z}\int_0^x|q(s)|^2\d s-\frac1{2z}({q(x)+q(0)})+O(1/z^2)\end{array}\right), \quad 
    z\in\C^-\,.	
\end{align}
\ese
Similarly, $\psi^+(x,z)$ has the following asymptotic behavior as $z\to\infty$:
\bse
\label{asympsi+}
\begin{align}
    \psi^+(x,z)&=\e^{\i zx}\left(\begin{array}{c}1+\frac{1}{2\i z}\int_0^x|q(s)|^2\d s+\frac1{2\i z}(q(x)-q(0))+O(1/z^2)\\\i+\frac{1}{2z}\int_0^x|q(s)|^2\d s-\frac1{2z}(q(x)+q(0)+O(1/z^2)\end{array}\right),\quad z\in\C^+\,,
	\\
	\psi^+(x,z)&=\e^{-\i zx}\left(\begin{array}{c}1-\frac{1}{2\i z}\int_0^x|q(s)|^2\d s-\frac1{2\i z}(q^*(x)-q^*(0))+O(1/z^2)\\-\i+\frac{1}{2z}\int_0^x|q(s)|^2\d s-\frac1{2z}(q^*(x)+q^*(0))+O(1/z^2)\end{array}\right)\,,\quad z\in\C^-\,.
\end{align}
\ese
\end{proposition}

\begin{remark}
The introduction of the transformation~\eqref{e:Ytildedef} 
is needed to ensure that 
the quantities $q(0)$ and $q^*(0)$ can be eliminated from the 
$O(1/z)$ terms in~\eqref{asympsi-} and \eqref{asympsi+}.
This will make it possible to use the Bloch-Floquet eigenfunctions~$\psi^\pm$ to define a suitable Riemann-Hilbert problem in section~\ref{s:inverse}
(see RHP~\ref{RHP1}),
and would have not be possible if one defined $\psi^\pm$ using $Y(x,z)$ instead of $\~Y(x,z)$
(see Appendix~\ref{a:alternative} for details).
The second components of $\psi^\pm$ are the Weyl-Titchmarsh $m$-functions \cite{Teschl,titchmarsh}.
Note that $\psi_\pm(x,z)$ are undefined at points $z$ for which $\~y_{12}(L,z) = 0$,
due to the normalization $\psi_\pm(0,z) = 1$, which cannot be satisfied if the
first component of the eigenvector of the monodromy matrix vanishes.
It is straightforward to see the asymptotic behavior in Proposition~\ref{p:asympsi}, 
together with the definitions~\eqref{e:Ytildedef} and~\eqref{BF}, yields:
\end{remark}

\begin{corollary}
The matrix potential $Q(x)$ of the Dirac operator~\eqref{e:Diracoperator} can be recovered as 
\be
Q(x) = \lim_{z\to\infty} \bm \sqrt2 \em iz[\sigma_3,U^{-1}\Psi(x,z)\e^{\i zx\sigma_3}]\,.
\label{e:reconstruction}
\ee
\label{c:reconstruction}
\end{corollary}


\kern-\bigskipamount
Proposition~\ref{p:asympsi} immediately gives the asymptotic behavior of $\psi_2^\pm(0,z)$ as $z\to\infty$
up to $O(1/z)$. 
When considering the time dependence of the potential in section~\ref{s:time}, however,
the next term in the asymptotic expansion of the Bloch-Floquet solutions at $z=0$
will also be needed
(cf.\ Appendix~\ref{asymofpsipm}):
\begin{proposition}
\label{asympsi0}
If $q\in C^2([0,L],\Complex)$, in addition to~\eqref{asympsi-} and~\eqref{asympsi+} we also have,
as $z\to\infty$: 
\bse
\begin{align}
	&\psi_2^-(0,z)=\begin{cases}\displaystyle
		-\i-\frac{q^*(0)}{z}-\frac{q^2(0)+q_x(0)}{2\i z^2}+O(1/z^3)\,, &z\in\C^+\,,
		\\[1ex]\displaystyle
		\i-\frac{q(0)}{z}+\frac{q^2(0)+q_x(0)}{2\i z^2}+O(1/z^3)\,, & z\in\C^-\,,\end{cases}\label{asymppsi-20}\\
	&\psi_2^+(0,z)=\begin{cases}\displaystyle
		\i-\frac{q(0)}{z}+\frac{q^2(0)+q_x(0)}{2\i z^2}+O(1/z^3)\,, &z\in\C^+\,,
		\\[1ex]\displaystyle
		-\i-\frac{q^*(0)}{z}-\frac{q^2(0)+q_x(0)}{2\i z^2}+O(1/z^3)\,, & z\in\C^-\,.\end{cases}\label{asymppsi+20}
\end{align}
\ese
\end{proposition}

\begin{remark}
Using the asymptotics in~\eqref{asymY}
and the corresponding expressions for the entries of $\~Y(x,z)$ (cf.\ Appendix~\ref{a:Ytildeasymp}), 
one can prove that $\Delta(z)$, and $\~y_{12}(L,z)$ are entire functions with order~1
(cf.\ Appendix~\ref{ordereigenf}). 
Thus, Hadamard's factorization theorem (e.g., see \cite{Ahlfors,Krantz}), 
allows us to write $\Delta(z)\pm1$ and $\~y_{12}(L,z)$ as the following infinite products:
if $\mu_0\ne0$ and $\z_{0}\ne0$,
\bse
\label{e:Hadamardproducts}
\begin{gather}
\~y_{12}(L,z)=\e^{A_1z+B_1}\prod_{j\in\mathbb{Z}}\left(1-\frac{z}{\mu_j}\right)\e^{\frac{z}{\mu_j}}\,,
\label{y12}
\\
\Delta(z)-1=\e^{A_2z+B_2}\prod_{j\in\Z\atop j\text{ even}}\left(1-\frac{z}{\z_{2j-1}}\right)\left(1-\frac{z}{\z_{2j}}\right)\e^{z\big(\frac{1}{\z_{2j-1}}+\frac{1}{\z_{2j}}\big)}\,,
\label{e:d-1}
\\
\Delta(z)+1=\e^{A_3z+B_3}\prod_{j\in\Z\atop j\text{ odd}}\left(1-\frac{z}{\z_{2j-1}}\right)\left(1-\frac{z}{\z_{2j}}\right)\e^{z\big(\frac{1}{\z_{2j-1}}+\frac{1}{\z_{2j}}\big)}\,.
\label{e:d+1}
\end{gather}
\ese
If $\mu_0=0$, \eqref{y12} is replaced by
\bse
\be
\~y_{12}(L,z)=z\e^{A_1z+B_1}\prod_{j\neq 0}\bigg(1-\frac{z}{\mu_j}\bigg)\,\e^{\frac{z}{\mu_j}}.
\ee
Similarly, if $\zeta_{0}=0$, \eqref{e:d-1} is replaced respectively by 
\begin{gather} 
\Delta(z)-1 = z\bigg(1-\frac{z}{\z_{-1}}\bigg)\,\e^{A_2z+B_2+\frac{z}{\z_{-1}}}\prod_{j\neq 0\atop j\text{ even}}\left(1-\frac{z}{\z_{2j-1}}\right)\left(1-\frac{z}{\z_{2j}}\right)\e^{z\big(\frac{1}{\z_{2j-1}}+\frac{1}{\z_{2j}}\big)}\,,
\end{gather}
\ese
while the factorization of $\Delta(z)+1$ remains unchanged.
The presence of the exponential terms in~\eqref{e:Hadamardproducts} is a significant difference between the spectral theory
for the ZS operator compared to that for Hill's equation,
and it is a consequence of the fact that the solutions of the ZS problem~\eqref{e:zs} have order of growth 1 as $z\to\infty$,
as opposed to those of Hill's equation, whose order of growth is~1/2.
A corresponding difference will appear in the definition of the jump matrix for the RHP in section~\ref{s:inverse}.
\bm
Note also that, by virtue of \cite{kappeler2} (see Lemma~5.4 and Appendix~C), we can also express
$\Delta^2(z)-1$ and $\~y_{12}(L,z)$ via the following infinite products:
\bse
\label{e:modifiedinfiniteproducts}
\begin{gather}
\tilde{y}_{12}(L,z)=-\prod_{j\in\Integer}\frac{\mu_j-z}{\pi_j},\label{y12simpler}\\
\Delta^2(z)-1=-\prod_{j\in\Integer}\frac{(\zeta_{2j}-z)(\zeta_{2j-1}-z)}{\pi_j^2},
\end{gather}
\ese
where $\pi_j$ is defined as
\be
\pi_j=\begin{cases}
	j\pi,& j\neq 0\\
	1, & j=0.
	\end{cases}
\ee
\em
\end{remark}

While the fundamental matrices $Y(x,z)$ and $\~Y(x,z)$ are entire, 
the Bloch-Floquet eigenfunctions $\psi^\pm$ are singular at $z=\mu_j$, as is easily seen from~\eqref{BF}.
For completeness, we next characterize this singular behavior.

\begin{proposition}
If $\nu_j=0$, both $\psi^+(x,z)$ and $\psi^-(x,z)$ remain finite at $z=\mu_j$.
\end{proposition}

\proof
When $\nu_j=0$, we have $\rho(\mu_j)=1/\rho(\mu_j)=1$. Thanks to Lemma \ref{y2}, $\~y_{11}(L,\mu_j)=1/\rho(\mu_j)=1$, indicating that $\mu_j$ is the zero of  both $\~y_{12}(L,z)$ and $\rho^{\pm1}(z)-\~y_{11}(L,z)$. According to the definition of $\psi^{\pm}$ \eqref{BF}, they are finite at $\mu_j$. 
\endproof

\begin{proposition}
If $\nu_j=1$, then $\psi^-(x,\mu_j)$ is finite, while $\psi^+(x,z)$ has a simple pole at $\mu_j$ with residue
\bse
\be
\Res_{z=\mu_j}\psi^+=\frac{\rho(\mu_j)-\rho^{-1}(\mu_j)}{\~y'_{12}(L,\mu_j)}\~y_2(x,\mu_j).
\ee
Conversely, if $\nu_j=-1$, then $\psi^+(x,\mu_j)$ is finite, while $\psi^-(x,z)$ has a simple pole at $\mu_j$ with residue
\be
\Res_{z=\mu_j}\psi^-=\frac{\rho^{-1}(\mu_j)-\rho(\mu_j)}{\~y'_{12}(L,\mu_j)}\~y_2(x,\mu_j).
\ee
\ese
\end{proposition}

\proof
For the purposes of the present calculation, it is convenient to define $\~\Psi(x,z) = (\psi^-,\psi^+)$ for all $z\in\Complex$.
[Obviously one has $\~\Psi = \Psi$ for all $z\in\Complex^+$ and $\~\Psi = \Psi\,\sigma_1$ for all $z\in\Complex^-$.]
We can write 
\be
\~\Psi(x,z) = \~Y(x,z)\,C(z),,\qquad
C(z) = \begin{pmatrix} 1 & 1 \\ c_-(z) & c_+(z) \end{pmatrix}\,,
\ee
with $c_\pm(z) = (\rho^{\pm1}(z) - \~y_{11}(L,z))/\~y_{12}(L,z)$. 
\bm 
The result then follows because, by Lemma~\ref{y2}, one has
$\~y_{11}(L,\mu_j) = \rho^{-\nu_j}(\mu_j)$.
Incidentally, note that 
\be
C^{-1}(z)=  \frac1{\rho(z) - \rho^{-1}(z)} \begin{pmatrix} \rho(z)-\~y_{11}(L,z) & - \~y_{12}(L,z) \\
	- \rho^{-1}(z)-\~y_{11}(L,z) & \~y_{12}(L,z)
		\end{pmatrix}
\ee
for all $z$ such that $\rho(z)\ne1$.
\em
\endproof

\begin{proposition}
The fundamental matrix solutions $Y$ and $\~Y$ satisfy the following symmetries:
	\be
    Y^*(z^*)=\sigma_1 Y(z)\sigma_1\,,\quad
	\~Y^*(z^*)= \~Y(z)\,,\qquad z\in\C\,.\label{e:symY~}
	\ee
\end{proposition}
\begin{proof}
The symmetry of $Y(z)$ follows as usual from its definition as a solution of the ZS problem and the initial condition $Y(0,z) = I$.
Using the definition~\eqref{e:Ytildedef} of $\~Y$ and the fact that $U^*\sigma_1 = U$, 
we then have
$\sigma_1 Y^*(z^*)\sigma_1 = \sigma_1 U^{-1*}\~Y^*(z^*)U^*\sigma_1 = Y(z)=U^{-1}\~Y(z)U$,
	which directly yields
$\~Y(z)=U\sigma_1 U^{-1*}\~Y^*(z^*)U^*\sigma_1U^{-1}=\~Y^*(z^*)$.
\end{proof}
\begin{proposition}
The Bloch-Floquet solutions $\psi^\pm$ and the matrix Bloch-Floquet solution satisfy
the symmetries 
\be
\psi^{\pm*}(z^*)=\psi^\pm(z)\,,\quad
\Psi^*(z^*)=\Psi(z)\,\sigma_1\,,\qquad
\forall z\in\C\setminus\R\,.
\ee
\end{proposition}
\begin{proof}
	Using \eqref{BF}, \eqref{e:symY~}, and the fact that $\rho^*(z^*)=\rho(z)$,  we have
	\be
	(\psi^+)^*(z^*)= \~y_1^*(z^*)+\frac{\rho^{*}(z^*)-\~y^*_{11}(L,z^*)}{\~y^*_{12}(L,z^*)}\~y^*_2(x,z^*)=\~y_1(x,z)+\frac{\rho(z)-\~y_{11}(L,z)}{\~y_{12}(L,z)}\~y_2(x,z)=\psi^+(z).
	\ee
The result for $\psi^-$ is proved in a similar way.
The result for $\Psi$ then follows immediately from the definition~\eqref{e:defPsi}.\end{proof}

\bm
\begin{lemma}
If the potential $q(x)$ is purely imaginary and even,
the Dirichlet eigenvalues are located at points of the main spectrum.
\label{l:Dirichletbase0}
\end{lemma}

\begin{proof}
Recall first that, for any purely imaginary and even potential, the symmetries of the ZS problem allow the monodromy matrix $M(z)$ to be written as 
\be
M(z) = \Delta(z) I + c(z) \sigma_3 + \i s(z)\sigma_2\,,
\label{e:MDeltaIcs}
\ee
where $I$ and $\sigma_2$ are the identity matrix and the second Pauli matrix.
In terms of $M(z)$, the Dirichlet eigenvalues are the zeros of 
$\frac \i2 [ M_{11} - M_{22} +   M_{21} - M_{12} ] = \i(c(z) - s(z))$.
Thus, the Dirichlet eigenvalues are the points for which $c(z)-s(z)=0$.
Next, note that, since $\det M(z) =1$, \eqref{e:MDeltaIcs} yields
\be
\Delta^2(z)- 1 = c^2(z) - s^2(z)\,.
\ee
The main spectrum is therefore located at the points for which $c^2(z)=s^2(z)$.  
We then see that each Dirichlet eigenvalue must coincide with a main eigenvalue.
\end{proof}

\em

\section{Inverse spectral theory for the periodic self-adjoint Zakharov-Shabat problem}
\label{s:inverse}

Using the results of section~\ref{s:asymptotics}, 
we now construct a RHP for the ZS eigenvalue problem~\eqref{e:zs} with periodic potentials 
that also applies in the case of infinite genus.
We then establish a unique solution of this RHP and we obtain a reconstruction formula for the potential 
from the solution of this RHP.

\begin{definition}
\label{d:f0f+f-}
Let the $2\times2$ matrix-valued function $B(z)$ be defined as
\bm
\bse
    \be
	B(z) = b(z) \begin{cases} 
            \i\,\diag(f^-,f^+),&z\in\Complex^+\setminus\Sigma(\L)\,,\\[1ex]
            \diag(f^+,f^-),&z\in\Complex^-\setminus\Sigma(\L)\,,
	\end{cases}
	\label{B}
    \ee
    with
    \be
        b(z) = \frac{(f^0(z))^{1/2}}{(\Delta^2(z)-1)^{1/4}}\,,
    \ee
and where $f^0(z)$ and $f^\pm(z)$ are defined as 
    \be
	f^0(z)= - \prod_{\nu_j=0}\frac{\mu_j-z}{\pi_j}\,,\quad
	f^+(z)=\prod_{\nu_j=1}\frac{\mu_j-z}{\pi_j}\,,\quad
	f^-(z)=\prod_{\nu_j=-1}\frac{\mu_j-z}{\pi_j}\,.
	\label{e:f+-0def}
    \ee
\ese\em
\end{definition}
%
Note that the functions $f^0(z)$, $f^\pm(z)$ defined by the infinite products in \eqref{e:f+-0def} are entire. 
\bm 
Moreover, one has $f^0(z) f^+(z) f^-(z) = \~y_{12}(L,z)$.
\em
As usual, if any of the products in~\eqref{e:f+-0def} contain no terms, the corresponding functions are defined to be equal to~1.
The function $\sqrt{f^0(z)}$ is defined to be holomorphic everywhere except where $f^0(z)\le0$,
and the branch for the square root is chosen accordingly (cf. Fig.~\ref{f:2}).
The function $\sqrt[4]{\Delta^2-1}$ is defined so that it is holomorphic for 
$z\in\C\setminus\big(\cup_{n=-\infty}^{\infty}[\zeta_{4n-4},\zeta_{4n-1}]\big)$
(also see Fig.~\ref{f:2}).

\begin{figure}[b!]
\centerline{\includegraphics[width=0.6\textwidth]{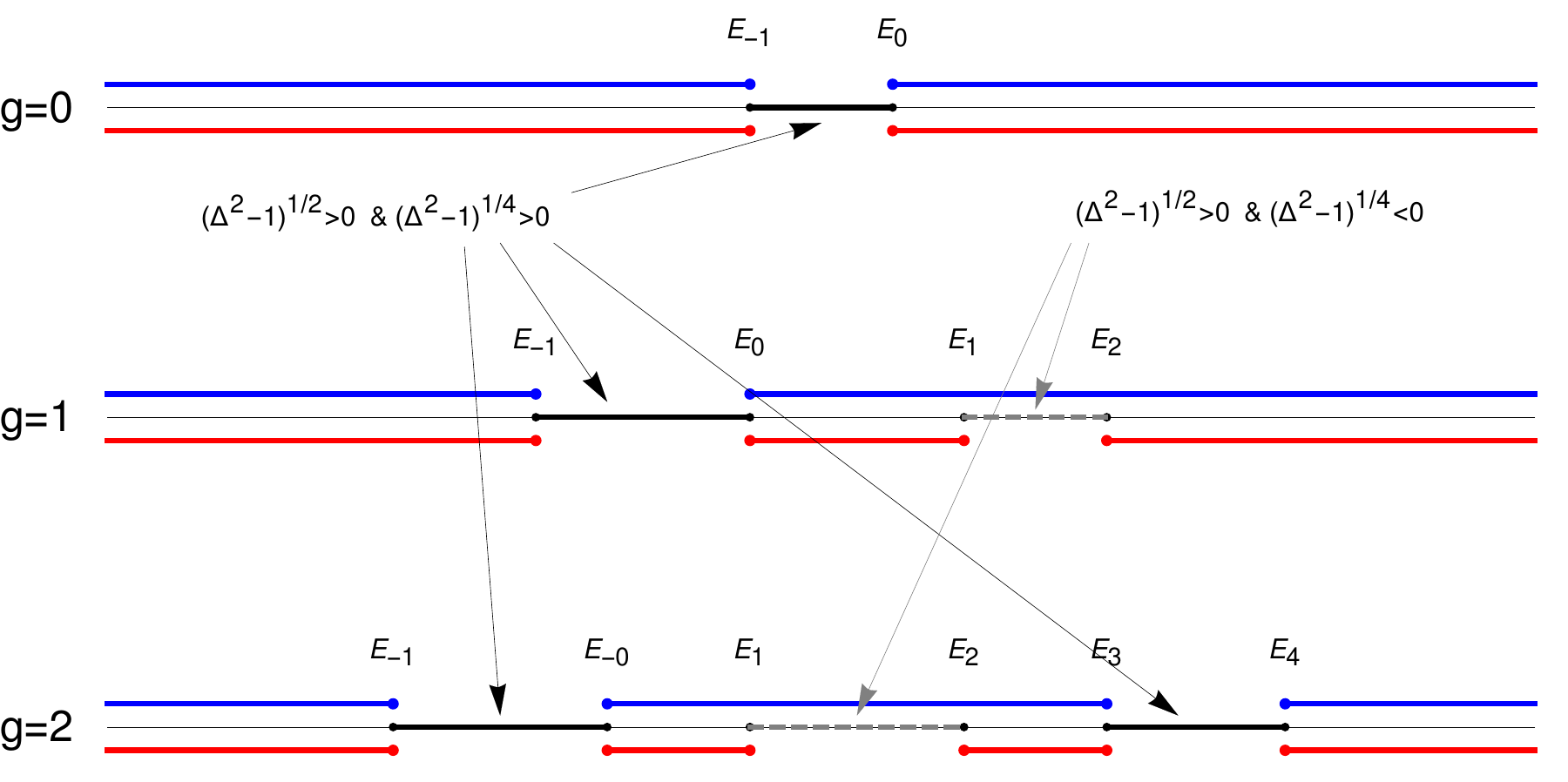}}
\caption{The branch cuts along the real $z$-axis for the functions $\sqrt{\D^2-1}$ (in red) and for $\sqrt[4]{\D^2-1}$ (in blue),
as well as the segments (thick black) where $\sqrt{\D^2-1}$ and $\sqrt[4]{\D^2-1}$ are both positive.
Top: $g=0$. Middle: $g=1$. Bottom: $g=2$.}
\label{f:2}
\end{figure}

\bm
\begin{proposition}
\label{p:spectralproducts}
The matrix $B(z)$ is completely determined by the spectral data $S(q)$ defined in \eqref{e:spectraldata}.
Explicitly,  
we have 
\vspace*{-0.6ex}
\bse
\begin{gather}
    b^4(z) =  - 
	 \prod\limits_{\nu_j=0} \frac{(\gamma_j-z)^2}{(E_{2j}-z)(E_{2j}-z)} \prod\limits_{\nu_j\neq0}\frac{\pi_j^2}{(E_{2j}-z)(E_{2j}-z)}\,,
		\label{e:rootratio}
		\\
    f^+(z)=\prod_{j=\g^-,\atop\nu_j=1 }^{\g^+}\frac{\gamma_j-z}{\pi_j}\,,
		\qquad
		f^-(z)=\prod_{j=\g^-\atop\nu_j=-1 }^{\g^+}\frac{\gamma_j-z}{\pi_j}\,.
		\label{e:fpmnew}
\end{gather}
\ese 
\end{proposition}
\begin{proof}
Taking the fourth power of the coefficient of $B(z)$, we have
\vspace*{-1ex}
\begin{multline}
    b^4(z) =  - \frac{
	\prod\limits_{\nu_j=0}(\mu_j-z)^2}{\prod\limits_{j\in\mathbb{Z}}(\z_{2j}-z)(\z_{2j-1}-z)}
    = \frac{
	\prod\limits_{\nu_j=0}(\mu_j-z)^2}{\prod\limits_{\nu_j=0}(\z_{2j}-z)(\z_{2j-1}-z)}\frac{\pi_j^2}{\prod\limits_{\nu_j\neq0}(E_{2j}-z)(E_{2j-1}-z)}
\\[-1ex]
    = \prod\limits_{\nu_j=0}\frac{(\gamma_j-z)^2}{(E_{2j}-z)(E_{2j-1}-z)}
	\prod\limits_{\nu_j\neq0}\frac{\pi_j^2}{(E_{2j}-z)(E_{2j-1}-z)},
\end{multline}
where in the last equality we used the fact that the fixed Dirichlet eigenvalues coincide with the double main eigenvalues $\zeta_j$, allowing them to cancel each other in the first fraction. 
\end{proof}
\em

\begin{definition}
Let the sectionally meromorphic $2\times2$ matrix $\Phi(x,z)$ be defined as
\be
\Phi(x,z)=\Psi(x,z)B(z)\e^{\i zx\sigma_3}\,,
\quad z\in\Complex\,.
\label{Phi}
\ee 
Also, let 
\be
V(x,z) = \begin{cases}
(-1)^{n + m(z)} \begin{pmatrix}{f^-(z)}/{f^+(z)} & 0 \\ 0 & {f^+(z)}/{f^-(z)}\end{pmatrix}\,, &z\in[E_{2n-2},E_{2n-1}]\,,\\
(-1)^{n + m(z)}\i\e^{-2\i zx\sigma_3}\sigma_1\,, &z\in[E_{2n-1},E_{2n}]\,.
\end{cases}
\label{e:Vsimple}
\ee
The counting function $m(z)$ in~\eqref{e:Vsimple} is defined as follows:
If there are no $\gamma_k$ with $\nu_k=0$, let $m(z) = 0$.
Alternatively, if there is at least one $\gamma_k$ with $\nu_k=0$, 
let $\gamma_*$ be the value of $\gamma_k$ closest to zero with $\nu_k=0$,
and let 
\be
m(z) = \begin{cases} 
  \big|\{\gamma_k:\gamma_*\leq\gamma_k\leq z,\nu_k=0\} \big|\,, &z\geq\gamma_*,
\\[0.4ex]
  \big|\{\gamma_k:z<\gamma_k< \gamma_*,\nu_k=0\} \big|\,, &z < \gamma_*\,.
\end{cases}
\label{e:mdef}
\ee
In other words,
for $z\ge\gamma_*$, $m(z)$ counts the number of Dirichlet eigenvalues $\gamma_k$ smaller than or equal to $z$ with $\nu_k=0$, 
whereas  
for $z<\gamma_*$, $m(z)$ counts the number of Dirichlet eigenvalues $\gamma_k$ larger than or equal to $z$ with $\nu_k=0$.
\end{definition}

\begin{definition}
\label{d:disks}
Let $R=\max_{\g^-\leq k\leq \g^+}\left\{E_{2k}-E_{2k-1}\right\}$. Define $D_j$ to be the open disc of radius $R$ centered at 
$c_j=\frac12(\z_{2j}+\z_{2j-1})$, and 
\be
\mathcal{D}=\mathbb{C}\setminus\left(\mathbb{R}\cup\bigcup_{k=\g^-}^{\g^+}\bar{D}_{k}\right).
\ee
\end{definition}

\begin{theorem}
The $2\times2$ matrix-valued function $\Phi(x,z)$ defined by~\eqref{Phi} solves the following Riemann-Hilbert problem:
\label{t:direct}
\begin{RHP}
\label{RHP1}
Find a $2\times2$ matrix-valued function $\Phi(x,z)$ such that
\vspace*{-1ex}
\begin{enumerate}
\advance\itemsep-2pt
\item 
$\Phi(x,z)$ is a holomorphic function for of $z$ for $z\in\C\setminus\R$.
\item 
$\Phi^\pm(x,z)$ are continuous functions of $z$ for $z\in\R\setminus\{E_k\}$, and have at worst quartic root singularities on $\{E_k\}$.
\item 
$\Phi^\pm(x,z)$ satisfies the jump relation 
\be
\Phi^+(x,z)=\Phi^-(x,z)V(x,z),\qquad z\in\Real\,,
\ee
where $V(x,z)$ is given by~\eqref{e:Vsimple}.
\item 
As $z\to\infty$ with $z\in\Complex$, $\Phi(x,z)$ has the following asymptotic behavior 
\be
\Phi(x,z) = U\,(I+O(1/z))B(z),
\label{asymPhi}
\ee
with $U$ as in~\eqref{e:Udef} and $B(z)$ as in~\eqref{B}.
\item 
There exist positive constants $c$ and $M$ such that $|\phi_{ij}(x,z)|\leq M\e^{c|z|^2}$ for all $z\in \mathcal{D}$.
\end{enumerate}
\end{RHP}
\end{theorem}
\begin{proof} 
We prove each condition separately.

\textbf{Condition 1}. 
It follows from \eqref{BF} that $\psi^\pm$ are meromorphic functions of $z$ for $z\in\mathbb{C}\setminus\Sigma(\L)$, so $\Psi(x,z)$ is meromorphic in $\mathbb{C}\setminus\Sigma(\L)$. From \eqref{B}, we can derive that $\Phi(x,z)$ could only be singular on $\{\mu_j\}\subset\mathbb{R}$ and $\{\z_j\}\subset\mathbb{R}$. Then $\Phi(x,z)$ is a holomorphic function of $z$ for $z\in\mathbb{C}\setminus\mathbb{R}$.
	
\textbf{Condition 2}.
From the definition of the Bloch-Floquet solutions~\eqref{BF}, we can derive directly that $\Psi_\pm(x,z)$ can only be singular at $\mu_j$. 
From \eqref{B}, we can also get that $\Phi_\pm(x,z)$ can only be singular at $\mu_j$ and $\z_j$. Next we will prove that $\Phi_\pm(x,z)$ cannot be singular at $\mu_j$ unless $\mu_j=E_k$ for some $k$. We will discuss three cases: $\nu_j=\pm1$ and $\z_{2j}=\z_{2j-1}=\mu_j$.

Suppose that $\nu_j=1$, which occurs only when $j=k$ for some $k$. 
From the definition of $\nu_j$  and Lemma~\ref{y2}, we have that $\~y_{11}(L,\mu_j)=\rho^{-1}(\mu_j)$,
and $\rho^{-1}(z)-\~y_{11}(L,z)$ is holomorphic in a neighborhood of $\mu_j$.
For $z\in\C^+$, we have 
\be
\Phi_+ = i\bm b(z)\em \big(\,\psi^-f^-\,,\,\psi^+f^+\,\big)\,\e^{\i zx\sigma_3}.
\ee
For the 1st column, the zeros of $\rho^{-1}(z)-\~y_{11}(L,z)$ at $\mu_j$ can cancel the zeros of $\~y_{12}(L,z)$ at $\mu_j$. For the~2nd column, the zeros of $f^+$ cancel the singularities of $\psi^+$. Therefore, $\Phi_+$ is nonsingular at $\mu_j$.
We can get the same conclusion for $z\in\C^-$ using the same method.

Suppose that $\nu_j=-1$,  which occurs only when $j=k$ for some $k$. We also have $\rho(\mu_j)=\~y_{11}(L,\mu_j)$. So $\rho(z)-\~y_{11}(L,z)$ is holomorphic in $\mu_j$'s neighborhood, and $\mu_j$ are the zeros of $\rho(z)-\~y_{11}(L,z)$. Using the same method as $\nu_j=1$, we see that $\Phi_\pm$ are nonsigular at $\mu_j$ when $\nu_j=-1$.

Suppose that $\z_{2j-1}=\z_{2j}$.  In this case, $\rho(\mu_j)=\rho^{-1}(\mu_j)=\pm1$. 
Therefore, $\mu_j$ are the zeros of both $\rho(z)-\~y_{11}(L,z)$ and $\rho^{-1}(z)-\~y_{11}(L,z)$. 
The zeros of $\rho(z)-\~y_{11}(L,z)$ and $\rho^{-1}(z)-\~y_{11}(L,z)$ at $\mu_j$ can cancel the zeros of $\~y_{12}(L,z)$ at $\mu_j$. Moreover, the square root zeros of $\sqrt{f^0}$ at $\mu_j$ cancel the square root of $\sqrt[4]{\D^2-1}$ at $\z_{2j}=\z_{2j-1}$. So  $\Phi_\pm$  are nonsigular at $\mu_j$.

Next we will prove that $\Phi_\pm(x,z)$ have at worst quartic root singularities on $\{E_k\}$.
Suppose that $\nu_{k}=\pm1$, then the only singular contribution to $\Phi_\pm$ at $z\to E_{2k}$ or $E_{2k-1}$ 
is from the boundary values of $1/{\sqrt[4]{\D^2-1}}$ for $z\in\mathbb{R}$. 
Suppose that $\nu_{k}=0$, then $\mu_j=E_{2j}$ or $E_{2j-1}$.  
If $\mu_j=E_{2j}$, then the only singular contribution to $\Phi_\pm$ near $E_{2j-1}$ is from the boundary values 
of $1/{\sqrt[4]{\D^2-1}}$ for $z\in\mathbb{R}$. 
If $\mu_j=E_{2j-1}$, then the only singular contribution to $\Phi_\pm$ near $E_{2j}$ is from the boundary values of ${1}/{\sqrt[4]{\D^2-1}}$ for $z\in\mathbb{R}$.

Suppose that either $\mu_{j}=E_{2j}$ or $\mu_j=E_{2j-1}$.  
By Lemma \ref{y2}, we have $\rho^{\pm1}(\mu_j)-\~y_{11}(L,x)=0$, since $\rho(\mu_j)=\pm1$ at $\mu_j$. 
As the singular behaviors of $\psi^\pm(x,z)$ are from $(\rho^{\pm1}(z)-\~y_{11}(L,z))/\~y_{12}(L,z)$, 
and the numerator  of $\psi^\pm(x,z)$  comprises only holomorphic functions and the square root of a holomorphic function, the zeros of the numerator at $\mu_j$ must have an order of at least $\half$. 
Since $\~y_{22}(L,z)$ has a simple zero at $\mu_j$, $\psi_\pm^\pm(x,z)$ must then have at worst square root singularities at $\mu_j$. 
The boundary values of ${1}/{\sqrt[4]{\D^2-1}}$ for $z\in\mathbb{R}$ have a quartic root singularity at $\mu_j$ and the boundary values of $\sqrt{f^0}$ have a square root zero at $\mu_j$. 
Hence, $\Phi$ has at worst a quartic root singularity at $\mu_j$.

\textbf{Condition 3}. 
Recalling the definition of $\rho(z)$ in section~\ref{s:spectrum},
we have that 
the discrepancy between the boundary values of $\sqrt{\D^2-1}$ on the branch cut~$\Sigma(\L)$ is 
\vspace*{-0.7ex}
\be
\rho_+(z)=\rho_-^{-1}(z),\qquad \rho_+^{-1}(z)=\rho_-(z),\qquad  z\in\Sigma(\L).
\ee
Therefore, 
\vspace*{-0.4ex}
\begin{align}
&\Psi_+(x,z)=\Psi_-(x,z)\,,\quad z\in\Sigma(\L)\,,\\
&\Psi_+(x,z)=\Psi_-(x,z)\sigma_1\,,\quad z\in\mathbb{R}\setminus\Sigma(\L)\,,
\end{align}
where $\sigma_1$ is the 1st Pauli matrix (cf.\ Appendix~\ref{a:notation}).
As a result, we see that the jump matrix $V(x,z)$ takes the form 
\be
V(x,z) = \begin{cases}\e^{-\i zx\sigma_3}(B_-)^{-1}B_+\e^{\i zx\sigma_3}, &z\in\Sigma(\L)\,,
\\
\e^{-\i zx\sigma_3}(B_-)^{-1}\sigma_1B_+\e^{\i zx\sigma_3}, &z\in\mathbb{R}\setminus\Sigma(\L)\,,
\end{cases}
\label{e:Vdef}
\ee
Next, we show that \eqref{e:Vdef} reduces to \eqref{e:Vsimple}. 
To do this, we need to compute the product $(B_-)^{-1}B_+$.
Since the functions $f^\pm(z)$ appearing in~\eqref{B} are holomorphic,
it is sufficient to compute the jump of the scalar factor $\sqrt{f^0(z)}/\sqrt[4]{\Delta^2(z)-1}$.
To this end, recall that one can equivalently express this ratio as~\eqref{e:rootratio}.
It is convenient to consider separately the numerator and denominator of the right hand side of~\eqref{e:rootratio}
and to compute their jumps individually.
We therefore define 
\vspace*{-0.7ex}
\bse
\begin{gather}
p(z) = \left(\prod_{k=\g^-\atop\nu_{k}=0 }^{\g^+}\bigg(1 - \frac{z}{\gamma_k}\bigg)\,\e^{z/\gamma_k}\right)^{\frac{1}{2}}\,,
\label{e:pdef}
\\
h(z) = \left(\prod_{k=\g^- }^{\g^+}{\frac{1}{E_{2k-1}E_{2k}}(E_{2k-1}-z)(E_{2k}-z)\e^{z\big(\frac{1}{E_{2k-1}}+\frac{1}{E_{2k}}\big)}}
 \right)^{\frac{1}{4}}\,,
\end{gather}
\ese
so that $\sqrt{f^0(z)}/\sqrt[4]{\Delta^2(z)-1} = \e^{{\i\pi}/{4}} p(z)/(\sqrt{2}h(z))$ by \eqref{e:rootratio}.
We choose the branch cut for the square root in~\eqref{e:pdef} consistently with the definition of the counting function
$m(z)$ in \eqref{e:mdef}.
Namely, if there are no $\gamma_k$ with $\nu_k=0$, $p(z)\equiv1$ and there is no branch cut.
If there is at least one $\gamma_k$ with $\nu_k=0$, 
recall that $\gamma_*$ is the value of $\gamma_k$ closest to zero with $\nu_k=0$.
Then we take the branch cut for $p(z)$ to be the subset of the real $z$-axis for which $m(z)$ is odd.
With this choice, we have 
\bse
\label{e:phjumps}
\begin{gather}
\frac{p_-(z)}{p_+(z)} = (-1)^{m(z)}, 
\qquad
z\in\Real,
\\
\noalign{\noindent Moreover, the choice of branch cut for the fourth root given after Definition~\ref{e:f+-0def} implies}
\frac{h_-(z)}{h_+(z)} = 
\begin{cases}
     (-1)^{n-1}\i , &z\in [E_{2n-2},E_{2n-1}]\subset\Sigma(\L) \,,\\
    (-1)^n , &z\in [E_{2n-1},E_{2n}]\subset\Real\setminus\Sigma(\L)\,.
   \end{cases}
\end{gather}
\ese
Inserting \eqref{e:phjumps} in \eqref{e:Vdef}, one now obtains~\eqref{e:Vsimple}. 

\textbf{Condition 4}.
Using the asymptotic behavior of $\psi^\pm(x,z)$ in Proposition~\ref{p:asympsi}, we have 
\be
\Psi(x,z)= (U + O(1/z))\,\e^{-\i zx\sigma_3},\qquad z\to\infty,~z\in\Complex.
\ee 
Therefore,
\be
U^{-1}\Psi(x,z)\e^{\i zx\sigma_3}=I+O(1/z),\qquad z\to\infty
\ee
for $z\in\Complex$, implying \eqref{asymPhi}.

\textbf{Condition 5}.
Let $y_{ij}$, $\phi_{ij}$ and $\psi_{ij}$ denote the $(i,j)$-th entries of the matrices $Y$, $\Phi$ and $\Psi$, respectively.  
We will prove that 
\vspace*{-0.4ex}
\be
\phi_{ij}\in\mathcal{A}_2\left(\mathbb{C}\setminus\left(\mathbb{R}\cup\bigcup\overline{D}_{k}\right)\right),\qquad \psi_{ij}\in\mathcal{A}_2\left(\mathbb{C}\setminus\left(\mathbb{R}\cup\bigcup\overline{D}_{k}\right)\right),
\ee
where the function class $\mathcal{A}_p(\Omega)$ is defined in Appendix~\ref{s:phragmenlindelof}.
In Appendix~\ref{ordereigenf}
we show that the entire functions $y_{ij}(x,z)$ have order at most~$1$. 
Thus, by Proposition~\ref{3.2}, we have 
\be
\~y_{ij}(x,z)\in\A_2(\mathbb{C}).
\ee
Next, since 
\be
|\rho^{\pm1}(z)|\leq|\D(z)|+\sqrt{|\D^2(z)-1|}\leq1+|\D(z)|,
\ee
and $\D(z)$ is an entire function of order $1$, we also have 
\be
\rho^{\pm1}(z)\in\A_2(\mathbb{C}\setminus\mathbb{R}).
\ee
Applying Lemma~\ref{PA} to \eqref{y12}, we establish that 
\vspace*{-0.4ex}
\be
\~y_{12}(L,z)^{-1}\in\A_2\left(\mathbb{C}\setminus\bigcup_{j\in\mathbb{Z}}\overline{D}_j\right).
\ee
Therefore, using Proposition \ref{3.1} and \ref{3.2}, we conclude that
\be
\psi_{ij}\in\A_2(\mathcal{D}).
\ee

Next, note that the smallest numbers $a_\pm$ and $a_0$ so that
\be
\sum_{\nu_j=\pm1}\mu_j^{-a_\pm},\hspace{0.2cm}\sum_{\nu_j=0}\mu_j^{-a_0}
\ee
converge are such that $a_\pm$, $a_0\leq 1$. So $f^\pm$ and $f^0$ are entire functions with orders at most $1$. Hence, 
\be
f^\pm,f^0\in\A_2(\mathbb{C}).\label{fA}
\ee
Applying Lemma~\ref{PA} to \eqref{e:d-1} and \eqref{e:d+1}, we conclude that
\be
\frac{1}{\D^2-1}\in\A_2\left(\mathbb{C}\setminus\bigcup_{j}\overline{D}_j\right).\label{DA}
\ee
Combining both \eqref{fA} and \eqref{DA}, we have
\be
g^\pm(z):=\frac{f^0(z)f^\pm(z)}{\D^2(z)-1}\in\A_2\left(\mathbb{C}\setminus\bigcup_{j}\overline{D}_j\right).
\ee
If $\mu_j\notin\{E_k\}$, then we can easily prove that $g^\pm(z)$ is analytic at $\mu_j$, which coincides with the center of $D_j$. Therefore, $g^\pm(z)$ are analytic in such a $D_j$. 
Finally, using Lemma~\ref{PA}, we have 
\be 
g^\pm(z)\in\A_2\left(\mathbb{C}\setminus\bigcup_{k=\g^-}^{\g^+}\overline{D}_{k}\right).
\ee
Recalling the definition \eqref{Phi} of $\Phi(x,z)$, 
this completes the proof of condition~5.
\end{proof}
\begin{lemma}
If $\Phi(x,z)$ solves the RHP~\ref{RHP1}, then $\det\Phi(x,z)\equiv1$.
\end{lemma}

\begin{proof}
We first prove that the particular solution to RHP \ref{RHP1} given by \eqref{Phi} has unit determinant. 
We have
\be
\det\Phi(x,z)=\det B(z)\det\Psi(x,z) = 
  \mp \frac{\~y_{12}(L,z)}{2\sqrt{\D^2-1}}\det\Psi_\pm\,, \qquad z\in\Complex^\pm\,.
\ee
Moreover,
\be
\det\Psi(x,z) = \det\Psi(0,z) = 
   \pm \frac{\rho(z)-\rho^{-1}(z)}{\~y_{12}(L,z)}= \mp \frac{2\sqrt{\D^2-1}}{\~y_{12}(L,z)}\,,\qquad z\in\Complex^\pm\,.
\ee
Thus, we conclude that $\det\Phi(x,z)=1$ for all $x$ and $z$.
Comparing $\det\Phi\equiv1$ with the asymptotics of $\Phi(x,z)$ in \eqref{asymPhi}, we have
\be
\det\left(U\,B(z)\right)=1+O(1/z)\,,
\label{asymphi0}
\ee
as $z\to\infty$ for $z\in\Complex$.

Next we show that \textit{all} solutions of the RHP~\ref{RHP1} have unit determinant.
Let $\tilde{\Phi}(x,z)$ be an arbitrary solution to RHP \ref{RHP1}. 
Define $f(x,z)=\det\tilde{\Phi}(x,z)$, which is holomorphic in $\mathbb{C}\setminus\mathbb{R}$.
As $\det V(x,z)=1$, we can derive that for $z\in\mathbb{R}\setminus\{E_k\}$
\be
f_+(x,z)=\det\tilde{\Phi}_+(x,z)=\det(V(x,z)\tilde{\Phi}_-(x,z))=\det\tilde{\Phi}_-(x,z)=f_-(x,z)\,.
\ee
As $\tilde{\Phi}$ also satisfies condition 2 in RHP \ref{RHP1}, $f(x,z)$ has at worst radical root singularities at $E_k$. Nevertheless, isolated singularities of a holomorphic function must have at least order 1. Thus, $f$ can be extended to an entire function.

To get $f\equiv1$, we still need to prove that $f$ is bounded. First, we prove that $f\in\A_2(\mathbb{C})$. From condition~5 of RHP \ref{RHP1}, we have
\be
\phi_{ij},\tilde{\phi}_{ij}\in\A_2\left(\mathbb{C}\setminus\left(\mathbb{R}\cup\bigcup_{k=\g^-}^{\g^+}\overline{D}_{k}\right)\right)\,.
\ee
Hence, we can conclude that 
\be
f\in\A_2\left(\mathbb{C}\setminus\left(\mathbb{R}\cup\bigcup_{k=\g^-}^{\g^+}\overline{D}_{k}\right)\right)\,.
\ee
Since $f$ is an entire function, by continuity, we have 
\be
f\in\A_2\left(\mathbb{C}\setminus\bigcup_{k=\g^-}^{\g^+}\overline{D}_{k}\right)\,.
\ee
By Proposition \ref{3.4}, we have that $f\in\A_2(\mathbb{C})$\,.
Comparison of \eqref{asymPhi} with \eqref{asymphi0} implies that 
\be
f(x,z)=1+O(1/z),\qquad z\to\infty\,.
\label{asymf}
\ee
But since $f(x,z)$ is entire, the asymptotic behavior in \eqref{asymf} implies that $f(x,z)$ is bounded for all $z\in\Complex$. 
So, by Liouville's theorem, $f(x,z)$ is constant and therefore $f(x,z)\equiv1~\forall z\in\Complex$.
\end{proof}

\begin{theorem}
For all $x\in\mathbb{R}$, the solution of  RHP \ref{RHP1} is unique.
\end{theorem}
\begin{proof}
For fixed $x\in\mathbb{R}$, let $\Phi$ and $\tilde{\Phi}$ be the solutions of RHP \ref{RHP1}. Define
\be
	R(x,z):=\Phi(x,z)\tilde{\Phi}^{-1}(x,z)\,.
\ee
To establish the uniqueness of the solution to RHP \ref{RHP1}, it suffices to prove that $\det R(x,z)\equiv 1$.

\noindent As $\det\tilde{\Phi}=1$, we can express $\tilde{\Phi}^{-1}$ as
\be
\~\Phi^{-1}(x,z)=\begin{pmatrix}\tilde{\phi}_{22}(x,z) &-\tilde{\phi}_{12}(x,z)\\-\tilde{\phi}_{21}(x,z)&\tilde{\phi}_{11}(x,z)\end{pmatrix}\,,
\ee
where $\~\phi_{ij}(x,z)$ is the $ij$-element of $\~\Phi(x,z)$.
It follows from condition 3 of RHP \ref{RHP1} that $R(x,z)$ has no jump, so it can be extended to a holomorphic function on $\mathbb{C}\setminus\{E_k\}$. 
Additionally, since the elements of $\Phi$ and $\tilde{\Phi}$ have at worst quartic root singularities at $E_k$, $R(x,z)$ has at worst square root singularities at $E_k$.  
Nevertheless, isolated singularities of a holomorphic function must have at least order~1. So $R(x,z)$ is an entire function of $z$.
	
Next, we will show that $r_{ij}\in\A_2(\mathbb{C})$. Since $\Phi$ and $\tilde{\Phi}$ solve RHP \ref{RHP1}, we have that
\vspace*{-0.6ex}
\be
\phi_{ij},\tilde{\phi}_{ij}\in\A_2\left(\mathbb{C}\setminus\left(\mathbb{R}\cup\bigcup_{k=\g^-}^{\g^+}\overline{D}_{k}\right)\right)\,.
\ee
Hence, we conclude that 
\be
	r_{ij}\in\A_2\left(\mathbb{C}\setminus\left(\mathbb{R}\cup\bigcup_{k=\g^-}^{\g^+}\overline{D}_{k}\right)\right)\,.
\ee
	Since $R(x,z)$ is an entire function, by continuity, we have 
\be
	r_{ij}\in\A_2\left(\mathbb{C}\setminus\bigcup_{k=\g^-}^{\g^+}\overline{D}_{k}\right)\,.
\ee
By Proposition~\ref{3.4}, we can derive that $r_{ij}\in\A_2(\mathbb{C})$.
	
The asymptotic behavior of both of the matrices $\Phi$ and $\tilde{\Phi}$ as $z\to\infty$ is given by~\eqref{asymPhi}.
Therefore we can calculate the asymptotic behavior of $\~\Phi^{-1}(x,z)$ as
\be
\~\Phi^{-1}(x,z)=B^{-1}(z)(I+O(1/z))U^{-1}\,,\qquad z\to\infty\,.
\ee
We can then easily get the asymptotic behavior of $R(x,z)$ as
\begin{equation}
\label{asymR}
R(x,z) = U\,(I+O(1/z))\,U^{-1} = I+O(1/z)\,,\qquad z\to\infty\,.
\end{equation}
Again, since $R(x,z)$ is entire, the asymptotic behavior \eqref{asymR} indicates that $r_{ij}$s are bounded for $z\in\Complex$. 
Using Liouville's theorem, we have  that $R(x,z)$ is constant.
The asymptotic behavior of $R(x,z)$ in \eqref{asymR} implies that 
$R(x,z)\equiv I~\forall z\in\Complex$.
Hence, the solution of RHP \ref{RHP1} is unique.
\end{proof}

\begin{corollary}
The potential matrix $Q(x)$ of the Dirac equation~\eqref{e:Diraceigenvalueproblem} is obtained from the solution of any solution $\Phi(x,z)$ of RHP~\ref{RHP1} by
	\be
	Q(x) = \lim_{z\to\infty} \i z [\sigma_3,U^{-1}\Phi(x,z)B^{-1}(z)]\,,
	\label{recons}
	\ee
with $U$ as in~\eqref{e:Udef}.
\end{corollary}
\begin{proof}
The reconstruction of $Q(x)$ in \eqref{recons} follows from \eqref{Phi}
and Corollary~\ref{c:reconstruction}.
\end{proof}

\section{Initial value problem for the defocusing NLS equation with periodic BC}
\label{s:time}

We now show how the results of the previous sections can be used to solve the initial value problem for the 
defocusing NLS equation with periodic BC,
\be
q(x+L,t)=q(x,t)\,.
\ee
Specifically, we construct a Riemann-Hilbert characterization of periodic solutions the defocusing NLS equation with infinite-gap initial conditions. 
%
To this end, recall that the defocusing NLS equation~\eqref{e:nls}
is the compatibility condition (or zero curvature condition) of the Lax pair \eqref{e:NLSLP}.
An equivalent formulation of the zero-curvature condition is the Lax equation
\bse
\begin{gather}
\L_t=[\mathcal{A},\L]\,,
\\
\noalign{\noindent with $\L$ as in~\eqref{e:Diracoperator} and}
\mathcal{A}=2\i\sigma_3\partial_x^2+2\i Q\sigma_3\partial_x-\i(Q^2-Q_x)\sigma_3\,.
\end{gather}
\ese
It is well known that, if the potential $q$ of~\eqref{e:Diraceigenvalueproblem} evolves according to the NLS equation~\eqref{e:nls}, the Lax spectrum $\Sigma(\L)$ of the ZS problem is invariant in time
(which is easily proven by showing that the Floquet discriminant $\Delta(z)$ is time-independent).
On the other hand, the movable Dirichlet eigenvalues are time dependent, and so are the Bloch-Floquet eigenfunctions.
Thus, in order to construct an effective time-dependent Riemann-Hilbert problem, 
one must determine the time evolution of the Dirichlet eigenvalues and 
of the Bloch-Floquet eigenfunctions.
We turn to these tasks next.


\begin{proposition}
For all $n\in\Integer$, the Dirichlet eigenvalues $\mu_n(t)$ satisfy the ODE
\bse
\label{e:dmundt}
\be
\partialderiv{\mu_{n}(t)}t=\nu_n(t)c_1(\mu_{n}(t))\frac{\rho(\mu_n(t))-\rho^{-1}(\mu_n(t))}{\~y'_{12}(L,t,\mu_{n}(t))}.
\ee
where 
\be
c_1(z)=-2z^2-\i z(q^*(0,t)-q(0,t))-|q(0,t)|^2-\frac{1}{2}(q_x(0,t)+q_x^*(0,t))\,.
\label{e:c1def}
\ee
\ese
\label{p:dirichlettimeevol}
\end{proposition}

\begin{proof}
Recall that the Dirichlet eigenvalues are poles of the modified Bloch-Floquet solutions,
which, in turn, are defined via~\eqref{e:Ytildedef}.
The transformation~\eqref{e:Ytildedef} yields the modified Lax equation
\bse
\begin{gather}
\tilde{\L}_t=[\tilde{\mathcal{A}},\tilde{\L}]\,,
\label{mlax}
\\
\noalign{\noindent with}
\tilde{\L}= U\L U^{-1}\,,
\qquad
\tilde{\mathcal{A}} = U\mathcal{A}U^{-1}\,.
\end{gather}
\ese
Recall that $\~y_2$ is a Bloch-Floquet solution of $\tilde\L$, 
i.e., $\tilde{\L}\~y_2=z\~y_2$. 
Differentiating both sides with respect to $t$, we have
$\tilde{\L}_t\~y_2+\tilde{\L}\~y_{2,t}=z\~y_{2,t}$,
which, 
combined with the modified Lax equation \eqref{mlax} implies 
\be
\tilde{\L}(\~y_{2,t}-\tilde{\mathcal{A}}\~y_2)=z(\~y_{2,t}-\tilde{\mathcal{A}}\~y_2)\,.
\ee
Therefore, for fixed $t$ and $z$, there exist constants $c_1$ and $c_2$ such that
\be
\~y_{2,t}-\tilde{\mathcal{A}}\~y_2=c_1\~y_1+c_2\~y_2\,.\label{y2t}
\ee
The normalization condition of $\~Y(x,z)$ at $x=0$ gives that $\~y_{12,t}(0,t,z)=\~y_{22,t}(0,t,z)=0$, so we obtain
\be
  (c_1,c_2)^T = -\tilde{\mathcal{A}}\~y_{2}(0,t,z)\,.
\ee
Moreover, the modified equation \eqref{e:modZS} also implies
\be
\~y_{2,xx}(x,t,z)=-\i zU\sigma_3U^{-1}\~y_{2,x}(x,t,z)+UQ(x,t)U^{-1}\~y_{2,x}(x,t,z)+UQ_x(x,t)U^{-1}\~y_2(x,t,z)\,,
\ee
and therefore
\begin{align}
\tilde{\mathcal{A}}\~y_2(x,t,z) 
        = -2\i z^2U\sigma_3U^{-1}\~y_2(x,t,z)+2zUQU^{-1}\~y_2(x,t,z)-\i U(Q^2+Q_x)\sigma_3U^{-1}\~y_2(x,t,z)\,.
\label{Ay2}
\end{align}
Evaluation at $x=0$ yields
\bse
\label{c1c2}
\begin{align}
	&c_1=-2z^2-\i z(q^*(0,t)-q(0,t))-|q(0,t)|^2-\half(q_x(0,t)+q_x^*(0,t))\,,\label{c1}\\
    &c_2=z(q^*(0,t)+q(0,t))-\frac{\i}{2}(q_x(0,t)-q_x^*(0,t))\,.\label{c2}
\end{align}
\ese
Plugging \eqref{Ay2} and~\eqref{c1c2} into \eqref{y2t} and evaluating~\eqref{y2t} at $x=L$ and $z=\mu_n(t)$, we get 
\begin{multline}
\~y_{12,t}(L,t,\mu_n(t)) = \big(2\mu_n(t)^2+\i\mu_n(t)(q^*(0,t)-q(0,t))+|q(0,t)|^2+\half(q_x(0,t)+q_x^*(0,t))\big)\\
\big(\~y_{22}(L,t,\mu_n(t))-\~y_{11}(L,t,\mu_n(t))\big)\,.
\end{multline}
At the same time, differentiating the expansion \eqref{y12} in $t$ we have
\be
\~y_{12,t}(L,t,z)=\e^{A_1z+B_1}\sum_{n\in\mathbb{Z}}\left[\frac{z\mu_{n,t}(t)}{\mu_n^2(t)}\e^{\frac{z}{\mu_n(t)}}\frac{z}{\mu_n(t)}\prod_{j\neq n}\left(1-\frac{z}{\mu_j(t)}\right)\e^{\frac{z}{\mu_j(t)}}\right]\,,
\label{y12t}
\ee
whereas differentiating  \eqref{y12} in $z$ we have
\be
\~y'_{12}(L,t,z)=A_1\e^{A_1z+B_1}\prod_{j\in\mathbb{Z}}\left(1-\frac{z}{\mu_j(t)}\right)\e^{\frac{z}{\mu_j(t)}}+\e^{A_1z+B_1}\sum_{n\in\mathbb{Z}}\left[-\frac{z}{\mu_n^2(t)}\e^{\frac{z}{\mu_n(t)}}\prod_{j\neq n}\left(1-\frac{z}{\mu_j(t)}\right)\e^{\frac{z}{\mu_j(t)}}\right]\,.
\label{y12z}
\ee
Evaluating \eqref{y12t} and \eqref{y12z} at $z=\mu_n(t)$ gives $\~y_{12,t}(L,t,\mu_n(t))=-\mu_{n,t}(t)\~y'_{12}(L,t,\mu_{n}(t))$.
Therefore, we can express $\mu_{n,t}(t)$ as:
\be
	\mu_{n,t}(t) = -\frac{\~y_{12,t}(L,t,\mu_n(t))}{\~y'_{12}(L,t,\mu_{n}(t))}
		 = \frac{c_1(\mu_{n}(t))(\~y_{22}(L,t,\mu_n(t))-\~y_{11}(L,t,\mu_n(t)))}{\~y'_{12}(L,t,\mu_{n}(t))}\,,
\nonumber
\ee
which yields \eqref{e:dmundt} upon substituting $c_1$ via~\eqref{e:c1def}.
\end{proof}

\begin{remark}
As a special case, 
\eqref{e:dmundt} shows that the fixed Dirichlet eigenvalues, for which $\nu_n=0$, are indeed
time-independent, as was mentioned in section~\ref{s:spectrum}.
Note that \eqref{e:dmundt} differs from the Dubrovin equations, which express the temporal derivative of 
the Dirichlet eigenvalues only in terms of known spectral data.
\end{remark}

\begin{proposition}
\label{asymofa}
The time-dependent Bloch-Floquet solutions $\psi^\pm(x,t,z)$ satisfy the ODEs
\be
\psi^\pm_t(x,t,z)+\alpha^\pm(t,z)\psi^\pm(x,t,z)=\tilde{\mathcal{A}}\psi^\pm(x,t,z)
\label{psit}
\ee
where 
\vspace*{-1ex}
\begin{align}
	\begin{split}
\alpha^\pm(t,z)=&z(q^*(0,t)+q(0,t))-\frac{\i}{2}(q_x(0,t)-q^*_x(0,t))\\
	&+\Big(2z^2+\i z\big(q^*(0,t)-q(0,t)\big)+|q(0,t)|^2+\half q_x(0,t)+\half q_x^*(0,t)\Big)
\psi^\pm_2(0,t,z)\,.
\end{split}
    \label{e:alphapm}
\end{align}
As a result, for each $s>0$, as $z\rightarrow\infty$, $z\in\Complex$, we have 
\bse
\label{e:alphapmasymp}
\begin{align}
	&\alpha^\pm=\pm2\i z^2+O(1/z),\qquad z\in\Complex^+,\\
	&\alpha^\pm=\mp2\i z^2+O(1/z),\qquad z\in\Complex^-.
\end{align}
\ese
Furthermore, if $\nu_n(t)=1$ then $\alpha^+(t,z)$ has a simple pole at $\mu_{n}(t)$ with residue $-\mu_{n,t}(t)$ and if $\nu_{n}=-1$ then $\alpha^-(t,z)$ has a simple pole at $\mu_{n}(t)$ with residue $-\mu_{n,t}(t)$.
\end{proposition}

\begin{proof}
For each $t$, we can factor the normalized Bloch-Floquet solutions $\psi^\pm(x,t,z)$  as
\be
\psi^+(x,t,z)=p^+(x,t,z)\rho^{x/L}(z),\qquad 
\psi^-(x,t,z)=p^-(x,t,z)\rho^{-x/L}(z)\,.
\ee
Differentiating $(\tilde{\L}-z)\psi^+=0$ with respect to $t$ and using the modified Lax equation imply that
\be
\tilde{\L}_t\psi^++\tilde{\L}\psi^+_t-z\psi^+_t=[\tilde{\mathcal{A}},\tilde{\L}]\psi^++\tilde{\L}\psi^+_t-z\psi_t^+=(\tilde{\L}-z)(\psi^+_t-\tilde{\mathcal{A}}\psi^+)=0.
\ee
So $\psi_t^+-\tilde{\mathcal{A}}\psi^+$ solves~\eqref{e:modZS} for all $t>0$. 
Therefore, it can be written as a linear combination of $\psi^\pm$:
\be
	p^+_t\rho^{x/L}(z) - \tilde{\mathcal{A}}\,p^+\rho^{x/L}(z) = 
	  - \alpha^+(t,z)p^+\rho^{x/L}(z) - \beta^+(t,z)p^-\rho^{-x/L}(z),
\ee
where $\alpha^+(t,z)$ and $\beta^+(t,z)$ are some undetermined functions independent of $x$. 
Recall that $|\rho(z)|<1~\forall z\in\mathbb{C}\setminus\Sigma(\L)$.
Evaluating the LHS and RHS of the above equation as $x\to\infty$,
every term decays exponentially except the last one.
Therefore, in order for the equation to hold, one needs $\beta^+(t,z)=0$. 
Hence $\psi^+$ solves~\eqref{psit} with the plus sign.
An analogous argument, but now taking the limit as $x\to-\infty$, 
shows that $\psi^-$ solves~\eqref{psit} with the minus sign, for some $\alpha^-(t,z)$. 

Next we derive~\eqref{e:alphapm} and~\eqref{e:alphapmasymp}. 
To do so, we evaluate the first component of~\eqref{psit} at $x=0$.
Since $\psi^\pm_1(0,t,z)=1$ for all time, we have $\psi^\pm_{1,t}(0,t,z)=0$. 
Therefore \eqref{psit} yields
\be
\alpha^\pm(t,z) = (\mathcal{A}\psi^\pm)_1(0,t,z)\,.
\label{e:alphapmeq}
\ee
From \eqref{Ay2} we have that
\be
\tilde{\mathcal{A}}v = -2\i z^2U\sigma_3U^{-1}v + 2zUQU^{-1}v-\i U(Q^2+Q_x)\sigma_3U^{-1}v\,.
\ee
Explicitly, the first component of the above equation is
\be
(\tilde{\mathcal{A}}v)_1 = 2z^2\,v_2 + z\,\big[(q^* + q)\,v_1 + \i(q^*-q)v_2 \big]
\\
- \txtfrac \i2 \big[ (q^*_x - q_x)\,v_1 + \i(2|q|^2 + q^*_x + q_x ) v_2 \big]\,.
\label{e:alphapmeq2}
\ee
Letting $v(x,t,z) = \psi^\pm(x,t,z)$, evaluating \eqref{e:alphapmeq2} at $x=0$ 
recalling that~\eqref{BF} evaluated at $x=0$ yields
$\psi^\pm_1(0,t,z)=1$ and 
$\psi^\pm_2(0,t,z) = (\rho^{\pm1}-\~y_{11}(L,t,z))/\~y_{12}(L,t,z)$,
\eqref{e:alphapmeq} then yields~\eqref{e:alphapm}.
\end{proof}

We are now ready to define the time-dependent Bloch-Floquet solutions.
Let $e^\pm(t,z)$ be solutions of the differential equation  
$e^\pm_t(t,z)=\alpha^\pm(t,z)e^\pm(t,z)$ with the initial condition $e^\pm(0,z)=1$,
i.e.,
\vspace*{-1ex}
\be
e^\pm(t,z)=\exp\left(\int_0^t\alpha^\pm(\tau,z)\d\tau\right)\,.
\ee
Define
\be\label{checkpsi}
 \check{\psi}^\pm(x,t,z)=\psi^\pm(x,t,z)e^\pm(t,z)\,.
 \ee
Then $\check{\psi}^\pm(x,t,z)$ satisfy the system of $\~\L\check\psi^\pm = z\check\psi^\pm$ and 
$\check{\psi}^\pm_t=\tilde{\mathcal{A}}\check{\psi}^\pm$,
for which the modified Lax equation~\eqref{mlax} is the compatibility condition.

\begin{proposition}
\label{p:e}
The functions $e^\pm(t,z)$ satisfy the following properties:
\vspace*{-1ex}
\begin{enumerate}
\advance\itemsep-2pt
\item 
$e^\pm(t,z)$ are meromorphic functions in $\mathbb{C}\setminus\Sigma(\mathcal{L})$.
\item 
$e^+(t,z)$ has simple poles on $\gamma_{k}(0)$ as $\nu_{k}(0)=1$and simple zeros on $\gamma_{k}(t)=1$ as $\nu_{k}(t)=1$;
$e^-(t,z)$ has simple poles on $\gamma_{k}(0)$ as $\nu_{k}(0)=-1$ and simple zeros on $\gamma_{k}(t)$ as $\nu_{k}(t)=-1$. 
\item
$e^\pm(t,z)$ both have square root singularities at $\gamma_{k}(0)$ when $\nu_{k}(0)=0$ and square root zeros at $\gamma_{k}(t)$ when $\nu_{k}(t)=0$.
\item 
The boundary values of $e^\pm$ satisfy $e^\pm_+(t,z)=e^\mp_-(t,z)$ for $z\in\Sigma(\mathcal{L})$.
\item 
For fixed $t$, $e^\pm$ have the following asymptotic behaviors 
as $z\rightarrow\infty$:
\vspace*{-0.6ex}
\be
e^\pm(t,z)=\begin{cases} \e^{\pm2\i z^2t}(1+O(1/z)),& z\in\Complex^+\\
		\e^{\mp2\i z^2t}(1+O(1/z)),& z\in\Complex^-\end{cases}\,.
\ee
\vspace*{-2ex}
\item 
There exist positive constants $c$ and $M$ such that $|e^\pm(t,z)|\leq M\e^{c|z|^2}$.
\end{enumerate}
\end{proposition}
\begin{proof}
Note that since $\alpha^\pm$ are holomorphic for $z\in\C\setminus\Real$ for all $t$, it is obvious that $e^\pm$ are holomorphic for $z\in\C\setminus\Real$. 
In the process of proving Condition~2, we will prove that $e^\pm$ can be extended to meromorphic functions for $z\in\C\setminus\Sigma(\L)$,
thus also proving Condition~1.
	
\textbf{Condition 2}.
Suppose that at $t=0$, $\mu_n\in(\zeta_{2n-1},\zeta_{2n})$ and $\nu_n=1$. 
For $t$ small enough, we have $\mu_n(t)\in(\zeta_{2n-1},\zeta_{2n})$, then from Proposition \ref{asymofa}, $\mu_n$ is a simple pole of $\alpha^+$, so for $z$ near $\mu_n(t)$,  $e^+$ could be written as 
\be
	e^+(t,z)=\exp\left(-\int_{0}^t\frac{\mu_{n\tau}(\tau)\d\tau}{z-\mu_n(\tau)}\right)e^+_{\rm reg}(t,z)\,,
\ee
where $e^+_{\rm reg}(t,z)$ is holomorphic and nonzero in an open set that encompasses the image of $\mu_n(t)$ for some small time window. 
We can rewrite the above function in terms of $\mu_n$
\be
	-\int_{0}^t\frac{\mu_{n\tau}(\tau)\d\tau}{z-\mu_n(\tau)}=\int_{\mu_n(0)}^{\mu_n(t)}\frac{\d\mu_n}{\mu_n-z}=\log(\mu_n(t)-z)-\log(\mu_n(0)-z)\,,
\ee
	which implies
\be
	e^+(t,z)=\frac{z-\mu_n(t)}{z-\mu_n(0)}e^+_{\rm reg}(t,z)\,.
\ee
The proof for $e^-$ is entirely analogous.
While the Dirichlet eigenvalues $\mu_n(t)$ can also be equal to one of the endpoints $\zeta_{2n-1}$ or $\zeta_{2n}$ of the gaps. Here we apply a time translation so that $\mu_n(0)=\zeta_{2n}$.  Without loss of generality, we assume that $\mu_n(t)$ is approaching $\zeta_{2n}$ and that $\mu_n(t)$ is a pole of $\alpha^+$. Then a local coordinate $w$ can be defined by 
	\be
	w(z)=\begin{cases}
		\rho(z)-\rho^{-1}(z)\,,&z\in\C^+,\quad z~\text{near}~\zeta_{2n}\,,
		\\
		\rho^{-1}(z)-\rho(z)\,,&z\in\C^-,\quad z~\text{near}~\zeta_{2n}\,.
		\end{cases}
	\ee
This is an analytic invertible transformation from a neighborhood of $w = 0$ to the $z$ plane.
In the $w$ plane, the single function 
\be
\alpha(t,w) = \alpha^\pm(t,z(w))\,,\qquad z(w)\in\C^\pm\,,
\ee
is meromorphic in a neighborhood of $w=0$, with a single simple pole at $w(\mu_n)$. 
Then 
\be
e(t,w) = e^\pm(t,z(w))\,,\qquad z(w)\in\C^\pm\,,
\ee
is related to $\alpha$ by
	\be
	e(t,w)=\exp\bigg(\int_0^t\alpha(\tau,w)\d \tau\bigg).
	\ee
	The derivatives $w(z)$ is
    \be
	w'(z)=  \pm\frac{\rho'(z)}{\rho(z)}\D(z)\,,\quad z\in\C^\pm,
	\ee
	which implies that
	\be
	\label{e:wmuntdt}
	w_t(\mu_n(t))=\nu_n(t)\frac{\rho'(z)}{\rho(z)}\D(z)|_{z=z(w(\mu_n(t)))}\mu_{n,t}(t).
	\ee
	Then by differentiating \eqref{tracepoly}, we get $\rho'(z)(\rho(z)-\D(z))=\D'(z)\rho(z)$, which gives the logarithmic derivative of $\rho(z)$
	\be
	\label{e:logderirho}
	\frac{\rho'(z)}{\rho(z)}=-\frac{\D'(z)}{\sqrt{\D^2(z)-1}}.
	\ee
	Using \eqref{e:dmundt} and \eqref{e:logderirho} in \eqref{e:wmuntdt} indicates
	\be
	w_t(\mu_n(t))=\frac{\D(z)\D'(z)}{\~y'_{12}(L,t,z)}\bigg|_{z=z(w(\mu_n(t))}c_1(\mu_n(t)),
	\ee
	where $c_1(z)$ is defined in \eqref{e:c1def}. 
	
	Since $\mu_n(t)$ is a pole of $\alpha$, a straightforward computation gives that
	\be
	\alpha(w)=\frac{-w_t(\mu_n(t))}{w-w(\mu_n(t))}+\text{ reg. near }\zeta_{2n}.
	\ee
	Now we compute the local behavior of $e(w)$ for $w$ near $0$ as
	\be
	e(w)=\exp\left(-\int_0^t\frac{-w_t(\mu_n(\tau))}{w-w(\mu_n(\tau))}\d\tau\right)e_{\rm reg}(w)=\frac{w-w(\mu_n(t))}{w-w(\mu_n(0))}e_{\rm reg}(w),
	\ee
where $e_{\rm reg}$ is holomorphic and nonzero near $w=0$. So we complete the proof of Condition 2.

\textbf{Condition 3}. Condition 3 follows from considering the boundary behavior of $\rho(z)$.
	
\textbf{Condition 4}.
By Proposition \ref{asymofa}, we have that as $z\rightarrow\infty$, for $z\in\overline{\Complex^+}$, 
$\pm{\alpha^\pm(t,z)}/{(2\i z^2)}$ are continuous function of $z$ converging to $1$, 
while for $z\in\overline{\Complex^-}$, $\mp{\alpha^\pm(t,z)}/{(2\i z^2)}$ are continuous function of $z$ converging to $1$. 
By the dominated convergence theorem, as $z\to\infty$ we have 
\be
\int_0^t \alpha^+(\tau,z)\,\d\tau=
  \pm 2\i z^2t + O(1/z)\,,\quad z\in\C^\pm\,,
\ee
which yields the asymptotic behavior of $e^+(t,z)$ as $z\rightarrow\infty$. 
The same procedure gives us the asymptotic behavior of $e^-(t,z)$ as $z\rightarrow\infty$.
	
\textbf{Condition 5}.
From the asymptotic behaviors of $\alpha^\pm$ (\ref{e:alphapmasymp}), a subalgebraic bound of the form $|\alpha^\pm(t,z)|\leq\max\{Cz^2,C'\}$ always holds for $z\in\Complex\setminus\mathcal{D}$. 
Here we denote the upper half region of $\mathcal{D}$ by $\mathcal{D}^+$. 
So the domain $\mathcal{D}^+$ consists of the upper half plane with excised half-domes centered on the real line whose heights are bounded above by $R/n$. We then divide $\mathcal{D}^+$ into two parts and denote the first part by $\mathcal{D}^+_U=\{z\in\mathbb{C}^+:\Im{z}>\max\{R,\log2/L\}\}$.  
From \eqref{asymY}, we have that
\be
\label{e:y12Lasymp}
\~y_{12}(L,z)=\sin(zL)+O(1/z)\,.
\ee
Using the lower bound
$|\sin(zL)|\geq\half\left(|\e^{-\i zL}|-\half\right)\geq\frac{|\e^{-\i zL}|}{4}$
we have, for $|z|$ large enough,
\be
		|\~y_{12}(L,z)|\geq C|\e^{-\i zL}|.
\ee
From \eqref{asymprhoinverse} and the fact that $\~y_{11}(L,z)$ has growth order 1, we have the bound 
\be
		|\rho^\pm(z)-\~y_{11}(L,z)|\leq C|\e^{-\i zL}|.
\ee
These together with \eqref{e:alphapm} give us 
\be
		|\alpha^\pm(t,z)|\leq\max\{C|z|^2,C'\}
\ee
for $z\in\mathcal{D}^+_U$.
We use $\mathcal{D}^+_L$ to denote the remaining part of $\mathcal{D}^+$. 
For $z$ large enough, one of the excised discs will overlap. 
At this point the non-straight pieces of boundary are deformed semicircles $C_n$ labeled by $n$ with the closest point to $\mu_n$ a distance away from $\mu_n$ bounded below by $R/n$ for some constant $R$, and furthest point from $\mu_n$ a distance away from $\mu_n$ bounded above by $R'/n$ for some constant $R'$. As $z\rightarrow\infty$, \eqref{e:y12Lasymp} gives us that
\be
	    \~y'_{12}(L,z)=L\cos(zL)+O(1/z).
\ee
The values $z=\mu_n$ are a distance $O(n^{-1})\sim\mu_n^{-1}$ away from the values of $z$ where $\cos(zL)=\pm1$, so for large $n$
\be
	     \~y'_{12}(L,\mu_n)=L+O(\mu_n^{-1}).
\ee
Hence, on $C_n$ the functions $\~y_{12}(L,z)$ are approximated by
\be
	   \~y_{12}(L,z)=(L+O(\mu_n^{-1}))(z-\mu_n)+O(z-\mu_n)\geq\frac{RL}{n}
\ee
for $n$ large enough. 
As $z\sim\frac{n\pi}{L}$ for $z\in C_n$ for $n$ large enough, we have $\~y_{12}(L,z)^{-1}=O(z)$ for $z\in C_n$ for $n$ large enough. 
Finally, we have $|\~y_{12}(L,z)^{-1}|\leq\max\{C|z|, C'\}$ for $z\in\mathcal{D}^+_L$.  
This completes the proof of condition~5.
\end{proof}

\begin{definition}
Let $\check{V}(x,t,z)$ be defined by
	\be
	\check{V}(x,t,z) = \begin{cases}
		(-1)^{n + m(z)} \begin{pmatrix}{f^-(z)}/{f^+(z)} & 0 \\ 0 & {f^+(z)}/{f^-(z)}\end{pmatrix}, &z\in[E_{2n-2},E_{2n-1}],\\
		(-1)^{n + m(z)}\i\e^{-2\i zx\sigma_3-4\i z^2t\sigma_3}\sigma_1, &z\in[E_{2n-1},E_{2n}],
	\end{cases}
	\label{e:Vxtz}
	\ee
with $m(z)$ still defined by \eqref{e:mdef},
and let the matrix-valued time-dependent Bloch-Floquet solution $\check\Psi(x,t,z)$ be defined as in~\eqref{e:defPsi}
but with $\psi_j^\pm$ replaced by $\check\psi_j^\pm$, as defined in \eqref{checkpsi}.
Finally, let $\check{\Phi}(x,t,z)$ be given by
\be
\label{Phicheck}
\check{\Phi}(x,t,z)=\check{\Psi}(x,t,z)B(z)\e^{\i\sigma_3zx+2\i\sigma_3z^2t}\,.
\ee
\end{definition}
\begin{theorem}
\label{qtoPhi}
Let $q(x,t)$ be the solution to the defocusing NLS equation~\eqref{e:nls}
with smooth initial data $q(x,0)=q_0(x)$, 
and let $\Sigma(q_0)$ be the spectral data of the corresponding Dirac operator. 
There exists a solution $\check{\Phi}$ to the following Riemann-Hilbert problem, constructed by \eqref{Phicheck},  
which is uniquely determined by  this following Riemann-Hilbert problem:
\begin{RHP}
\label{RHP2}
Find a $2\times2$ matrix-valued function $\check{\Phi}(x,t,z)$ such that
\vspace*{-1ex}
\begin{enumerate}
\advance\itemsep-4pt
    \item 
    $\check{\Phi}(x,t,z)$ is a holomorphic function of $z$ for $z\in\C\setminus\R$.
    \item 
    $\check{\Phi}^\pm(x,t,z)$ are continuous functions of $z$ for $z\in\R\setminus\{E_k\}$, and have at worst quartic root singularities on $\{E_k\}$.
    \item 
    $\check{\Phi}^\pm(x,t,z)$ satisfy the jump relation $\check{\Phi}^+(x,t,z)=\check{\Phi}^-(x,t,z)\check{V}(x,t,z)$, with $\check{V}(x,t,z)$ given by \eqref{e:Vxtz}.
    \item 
    As $z\to\infty$ with $z\in\Complex$, $\check{\Phi}(x,t,z)$ has the following asymptotic behavior 
    \be
    \check{\Phi}(x,t,z) = U\,(I+O(1/z))B(0,z).\label{asymcheckPhi}
    \ee
    \item 
    There exist positive constants $c$ and $M$ such that $|\check{\phi}_{ij}(x,t,z)|\leq M\e^{c|z|^2}$ for all $z\in \mathcal{D}$.
\end{enumerate}
\end{RHP}
\end{theorem}

For fixed $x,t\in\mathbb{R}$, the solution of  RHP \ref{RHP2} is unique.
Moreover, the potential matrix $Q(x,t)$ of the Dirac operator~\eqref{e:Diracoperator} 
is given in terms any solution $\check{\Phi}(x,t,z)$ of RHP \ref{RHP2} by
\be
Q(x,t)=\lim_{z\to\infty}\i z [\sigma_3,U^\dag\hat{\Phi}(x,t,z)B(0,z)^{-1}],
\label{trecons}
\ee
with $U$ as in \eqref{e:Udef}.

\section{Periodicity conditions in space and time}
\label{s:periodicity}

It is well known that, even in the finite-genus case, and even for Hill's equation,
the potential generated by a generic set of spectral data is not periodic in general, but only quasi-periodic.
A natural question is therefore whether it is possible to identify a subset of the spectral data
which guarantees that the associated potential is in fact periodic.
In this section, assuming existence of the solutions of the RHP~\ref{RHP1},
and following~\cite{McLaughlinNabelek}, 
we prove that it is possible to define a scalar RHP for which the existence of a solution 
implies the potential is periodic. 
Furthermore, we also prove there is an analogous scalar RHP for which the existence of a solution implies temporal periodicity.

To identify the desired RHP, note first that, if the potential is periodic, the Bloch-Floquet theory of the ZS problem
yields the existence of the Floquet multiplier $\rho(z)$.
Next, note that $\rho(z)$ satisfies the scalar RHP defined by Conditions~(i)--(iv) in Theorem~\ref{periodicity} below.
Therefore, the latter is the desired RHP, 
since the existence of solutions for it guarantees the existence of the Floquet multiplier $\rho(z)$.
Next we make these considerations more precise.

Consider a ``candidate spectral data'', namely a sequence $\{E_{2n-1},E_{2n},\gamma_n,\nu_n\}_{n=g_-,\dots,g_+}$ of periodic and antiperiodic eigenvalues $E_n$, 
Dirichlet eigenvalues $\gamma_n$ and signs $\nu_n$,
with $g_-<g_+$ either finite or infinite, satisfying the following conditions:
\vspace*{-1ex}
\begin{itemize}
\advance\itemsep-4pt
	\item $E_{2g_--1}<\dots<E_{n-1}<E_n<E_{n+1}<\dots<E_{2g_+}$;
	\item $\forall n = g_-,\dots,g_+$, $\gamma_n\in[E_{2n-1},E_{2n}]$ and $\nu_n\in\{-1,0,1\}$;
	\item $R=\max_k|E_{2k}-E_{2k-1}|<\infty$;
	\item 
	If $g_-$ and $g_+$ are finite, there exists $N$, $C$ such that $|E_n|>Cn^2$ for all $|n|>N$;
	\item 
	If either $g_-$ or $g_+$ are infinite, $\exists N>0$ such that the discs $D_n$ of radius $R$ centered at $(E_{2n}+E_{2n-1})/2$ are disjoint for $|n|\geq N$.
\end{itemize}

\begin{theorem}
\label{periodicity}
Let $\Phi(x,t,z)$ be the solution of RHP~\ref{RHP2} defined from the candidate spectral data above, 
which determines the potential $q(x,t)$ via \eqref{trecons}.
If there exists a function $r_1(z)$ such that
        \vspace*{-1ex}
		\begin{enumerate}
        \advance\itemsep-4pt
		\item[(i)] 
		$r_1$ is holomorphic in $\C\setminus(\Real\setminus\Sigma(\L))$ with continuous boundary values $r_{1\pm}$ on $\Real\setminus\Sigma(\L)$ from above and below;
		\item[(ii)] 
		$r_1$ satisfies the jump relation $r_{1+}(z)=r_{1-}^{-1}(z)$ for $z\in\Real\setminus\Sigma(\L)$;
        \item[(iii)] 
        $r_1$ satisfies the asymptotic condition $r_1(z)=\e^{\i zL_1}(1+O(1/z))$ for $L_1>0$;
        \item[(iv)] 
        $r_1(z)\in\A_2(\mathcal{D})$,
		\end{enumerate}
		then $q(x+L_1,t)=q(x,t)$.
\end{theorem}

\begin{theorem}
\label{timeperiodicity}
Let $\Phi(x,t,z)$ be defined as in Theorem~\ref{periodicity}.
If there exists a function $r_2(z)$ such that
        \vspace*{-1ex}
		\begin{enumerate}
        \advance\itemsep-4pt
        \item[(i)] 
        $r_2$ is holomorphic in $\C\setminus(\Real\setminus\Sigma(\L))$ with continuous boundary values $r_{2\pm}$ on $\Real\setminus\Sigma(\L)$ from above and below;
        \item[(ii)] 
        $r_2$ satisfies the jump relation $r_{2+}(z)=r_{2-}^{-1}(z)$ for $z\in\Real\setminus\Sigma(\L)$;
        \item[(iii)] 
        $r_2$ satisfies the asymptotic condition $r_2(z)=\e^{2\i z^2L_2}(1+O(1/z))$ for  $L_2>0$;
        \item[(iv)] 
        $r_2(z)\in\A_2(\mathcal{D})$,
		\end{enumerate}
	then $q(x,t+L_2)=q(x,t)$.
\end{theorem}

\begin{proof}[Proof of Theorem~\ref{periodicity}]
Suppose that $r_1(z)$ exists, and define the function
\be
\tilde{\Phi}(x,t,z)=\Phi(x+L_1,t,z)r_1(z)^{\sigma_3}\e^{-\i z\sigma_3 L_1}\,.
\ee
We will prove that $\tilde{\Phi}$ also solves RHP \ref{RHP2}.
Condition~1 and condition~2 are obvious. 
Regarding condition~3 (the jump), 
note that on the spectral band $[E_{2n-2},E_{2n-1}]$, the jump relation for $\tilde{\Phi}$ is
\begin{multline}
\tilde{\Phi}_+(x,t,z) = \Phi_+(x+L_1,t,z)r_1(z)^{\sigma_3}\e^{-\i z\sigma_3 L_1}
\\
  = \Phi_-(x+L_1,t,z)(-1)^{n+m(z)}\begin{pmatrix} f^-/f^+ & 0 \\ 0 & f^+/f^- \end{pmatrix}r_1^{\sigma_3}\e^{-\i z\sigma_3 L_1}
\\
    = \tilde{\Phi}_-(x+L_1,t,z)(-1)^{n+m(z)}\begin{pmatrix} f^-/f^+ & 0 \\ 0 & f^+/f^- \end{pmatrix}.
\end{multline}
On the other hand, on $[E_{2n-1},E_{2k}]$ the jump can be derived by 
\begin{multline}
\tilde{\Phi}_+(x,t,z) = \Phi_+(x+L_1,t,z)r_{1+}(z)^{\sigma_3}\e^{-\i z\sigma_3 L_1}
\\
    = \Phi_-(x+L_1,t,z)(-1)^{n+m(z)}\i\e^{-2\i zx\sigma_3}\sigma_1r_{1-}^{-\sigma_3}\e^{-\i z\sigma_3 L_1}
\\
    =c\tilde{\Phi}_-(x+L_1,t,z)(-1)^{n+m(z)}\i\e^{-2\i zx\sigma_3}\sigma_1.
\end{multline}
Therefore, $\tilde{\Phi}$ satisfies condition~3. 
Condition~4 follows from the asymptotic behavior of $r_1(z)$. 
Condition~5 comes from the fact that $\A_2(\mathcal{D})$ is an algebra.  
Theorem~\ref{qtoPhi} and the reconstruction formula \eqref{trecons} imply that 
\be
\Phi(x,t,z)=\Phi(x+L_1,t,z)r_1(z)^{\sigma_3}\e^{-\i zL_1},
\ee
which indicates $q(x+L_1,t)=q(x,t)$.
\end{proof}

\begin{proof}[Proof of Theorem~\ref{timeperiodicity}]
To prove the statement about temporal periodicity, suppose that $r_2$ exists, and introduce $\tilde{\Phi}$ as 
$\tilde{\Phi}(x,t,z)=\Phi(x,t+L_2,z)r_2(z)^{\sigma_3}\e^{-2\i z^2L_2}$.
Then the proof proceeds in a manner entirely analogous to the above.
\end{proof}

\section{Baker-Akhiezer functions and finite-gap potentials}
\label{s:finitegap}

Up to this point we have made no use of the spectral curve underlying the periodic potentials of the Dirac operator.
In this section we show how the formalism presented above can also be interpreted in terms of Baker-Akhiezer functions.
Then we show how, 
when $\g_-$ and $\g_+$ are both finite, the formalism can be used to recover the finite-genus potentials of the Dirac operator 
and therefore the corresponding finite-gap solutions of the defocusing NLS equation.
Specifically, we provide a corollary to Theorems~\ref{t:direct} and~\ref{qtoPhi}
that allows us to interpret the method we have presented in terms of 
Baker-Akhiezer functions on a Riemann surface of possibly infinite genus.

Let $\mathcal{X}\subset\mathbb{C}\times\mathbb{C}$ be the curve defined by
\be
w^2=P(z)=-L^2(E_{-1}-z)(E_0-z)\prod_{k=g_-}^{g_+}(z-E_{2k-1})(z-E_{2k})\label{w1}
\ee
Similarly to \cite{McLaughlinNabelek},
this curve is diffeomorphic via a holomorphic map to the desingularization of the curve defined by 
$w^2=\Delta^2(z)-1$
by two-point blowups at the degenerate Dirichlet eigenvalues.
As in \cite{McLaughlinNabelek}, if $\g=\infty$ we choose not to compactify~$\mathcal{X}$ as a topological space
as the compactification would not be smooth at infinity 
(because of the accumulation of ``holes'' at infinity).
The projection $\pi((z,w))\rightarrow z$ onto the $z$ plane has two inverses under composition:
\begin{align}
\pi_\pm^{-1}(z)=(z,\pm\sqrt{P(z)}),
\end{align}
which hold only at branch points. 

\begin{corollary}
For every $x,t$ there is a unique meromorphic function $\check{\psi}(x,t,p)$ on $\mathcal{X}$ 
with only simple poles $p_k=(\mu_{n_k},\nu_{n_k},P(\mu_{n_k}))$ for $k=g_-,g_-+1,...,g_+-1,g_+$ such that,
as $z\rightarrow\infty$ with $z\in\C\setminus\R$,
\begin{align}
\check{\psi}(x,t,\pi_\pm ^{-1}(z))=\e^{\pm\i (zx+2z^2t)}\left(\begin{pmatrix}1\\\mp\i\end{pmatrix}+O(1/z)\right).
\end{align}
Moreover,
$\check{\psi}(x,t,\pi_\pm^{-1}(\,\cdot\,))\in\mathcal{A}_2(\mathcal{D})$.
\end{corollary}    
The function $\check\psi$ is known as the Baker-Akhiezer function for the NLS equation
(e.g., see \cite{BBEIM,KotlyarovShepelsky} for the finite-genus case).


In the finite gap case, let $\mathcal{X}$ be the Riemann surface of genus $\g = \g_+-\g_-$ 
defined by the equation $w^2=P(z)$, where
\be
P(z)=\prod_{j=2g_--1}^{2g_+}(z-E_{j})
\ee
with cuts along the arcs $\Gamma_j=(E_{2j},E_{2j+1})$. 
The standard projection $\pi:\mathcal{X}\to \mathbb{CP}^1$ is defined by:
\be
\pi(P)=z,\quad P=(w,z).
\ee
The projection $\pi$ defines $\mathcal{X}$ as a two-sheeted covering of $\mathbb{CP}^1$. There are two points $\infty^\pm\in\mathcal{X}$, with the property $\pi(\infty^\pm)=\infty\in\mathbb{CP}^1$.

Let us introduce a canonical basis $a_i$, $b_j$ of cycles satisfying $a_i\circ b_j=\delta_{ij}$, $a_i\circ a_j=0$ and $b_i\circ b_j=0$, where $\circ$ denotes minimal crossing number in the homology class of $a_i$ and $b_j$. Then we introduce the basis of Abelian differentials $\omega_i$ of the first kind on the Riemann surface $\mathcal{X}$, normalized such that 
\be
\int_{a_j}\omega_i=2\pi\i\,\delta_{ij}.
\ee
The Abel map $\mathbf{A}: \mathcal{X}\rightarrow\mathbb{C}^{\g}$ is defined as follows:
\be
A_j(P)=\int_{P_{2g_-}}^P\omega_j.
\ee
The normalized holomorphic differentials define the $b$-period matrix as:
\be
B_{jk}=\int_{b_k}\omega_j.
\ee
Associated with the matrix $B$ there is the Riemann theta function defined for $\textbf{u}\in\C^{\g}$ by the Fourier series:
\be
\Theta(\mathbf{p})=\sum_{\mathbf{m}\in\mathbb{Z}^{\g}}\exp \{\half\langle B\mathbf{m},\mathbf{m}\rangle+\langle\mathbf{p},\mathbf{m}\rangle\},
\ee
where $\langle\mathbf{p},\mathbf{q}\rangle=p_1q_1+\dots+p_nq_n$ for $\mathbf{p}$, $\mathbf{q}\in\C^n$.
The Abelian integrals $\Omega_1(P)$, $\Omega_2(P)$ and $\Omega_3(P)$, $P\in\mathcal{X}$ are fixed by the following conditions:
\begin{itemize}
	\item $ \int_{a_i}\d\Omega_j=0.$
	\item 	 
	$\Omega_1(P)=\pm(z+O(1)),\quad P\to\infty^\pm$,\\
	$\Omega_2(P)=\pm(2z^2+O(1)),\quad P\to\infty^{\pm},$\\
	$\Omega_3(P)=\pm(\log z+O(1)),\quad P\to\infty^{\pm},\quad z=\pi(P).$
	\item $\Omega_j(P)$ have no singularities at any points other than $\infty^{\pm}$.
\end{itemize}
Next we introduce an arbitrary divisor $\mathcal{D}$ with $\deg\mathcal{D}=g$ in general position, namely
$\mathcal{D}=\sum\nolimits_{j=1}^\g P_j$, with $\pi(P_j)\neq E_j$.
Then following \cite{BBEIM}, for all $\alpha\in\Complex$ we consider the vector-valued Baker-Akhiezer function $\psi(P,x,t)=(\psi_1,\psi_2)^T$ defined as follows:
\bse
\label{e:BakerAkhiezer}
\begin{align}
&\psi_1(P)=\frac{\Theta(\int_{\infty^-}^P \boldsymbol{\omega}+\i \mathbf{V}x+\i\mathbf{W}t-\mathbf{D})\,\Theta(\mathbf{D})}
{\Theta(\int_{\infty^-}^P\boldsymbol{\omega}-\mathbf{D})\,\Theta(\i\mathbf{V} x+\i\mathbf{W}t-\mathbf{D})}
  \exp\big\{\i x\Omega_1(P)+\i t\Omega_2(P)-\txtfrac{\i}{2}E x+\txtfrac{\i}{2}Nt\big\},
\\
&\psi_2(P)=\alpha\sqrt{\omega_0}\frac{\Theta(\int_{\infty^-}^P \boldsymbol{\omega}+\i \mathbf{V}x+\i\mathbf{W}t-\mathbf{D}-\mathbf{r})\,\Theta(\mathbf{D}-\mathbf{r})}
{\Theta(\int_{\infty^-}^P\boldsymbol{\omega}-\mathbf{D})\,\Theta(\i\mathbf{V} x+\i\mathbf{W}t-\mathbf{D})}
  \exp\big\{\i x\Omega_1(P)+\i t\Omega_2(P)+\txtfrac{\i}{2}E x-\txtfrac{\i}{2}Nt+\Omega_3\big\}.
\end{align}
\ese
The vector-valued parameters appearing in \eqref{e:BakerAkhiezer} are as follows:
\begin{align*}
&\boldsymbol{\omega}=(\omega_1,\dots,\omega_{\g}),\quad
    \mathbf{V}=(V_1,\dots,V_{\g}),\quad
    \mathbf{W}=(W_1,\dots,W_{\g}),\\
&\mathbf{r}=\int_{\infty^-}^{\infty^+}\boldsymbol{\omega},\quad\mathbf{D}=\sum_{j=1}^{\g}\int_{\infty^-}^{P_j}\boldsymbol{\omega}+\mathbf{K},\\
&V_j=\int_{b_j}\d\Omega_1,\quad
    W_j=\int_{b_j}\d\Omega_j,\quad
    K_j=\pi\i+\half B_{jj}-\frac{1}{2\pi\i}\sum_{k\neq j}\int_{a_k}(\int_{\infty^-}^P\omega_j)\omega_k(P).
\end{align*}
The quantities $E$, $N$ and $\omega_0$ are determined by the second terms of the asymptotic expansions of the integrals $\Omega_j(P)$ at the points $\infty^{\pm}$ \cite{BBEIM}.

Finally, we fix some simple connected neighborhood $U$ of the point $z=\infty$ which has no branch points. Then for each $z\in U$, $\pi^{-1}(z)$ identifies exactly two points denoted by $P^{\pm}\in\mathcal{X}$ so that $p^\pm\to\infty^\pm$ when $z\to\infty$. 
For all $z\in U$, we the define the matrix-valued function
\be\label{BA}
\Psi(z,x,t)=(\psi(P^+),\psi(P^-)).
\ee
It is relatively straightforward to show that the above function satisfies all the propeties of the finite-genus Baker-Akhiezer function \cite{BBEIM}. 
As a result, the solution to the NLS equation recovered from the Baker-Akhiezer function is given by
\be
q(x,t)= 2\e^{-\i Ex+\i Nt}
  \frac{2\,\Theta(\i\mathbf{V}x+\i\mathbf{W}t-\mathbf{D}+\mathbf{r})\,\Theta(\mathbf{D})}
    {\alpha\,\Theta(\i\mathbf{V}x+\i\mathbf{W}t-\mathbf{D})\,\Theta(\mathbf{D}-\mathbf{r})}.
\ee

\section{Concluding remarks}
\label{s:conclusions}

In summary, we formulated the direct and inverse transform for a one-dimensional self-adjoint Dirac operator 
with periodic boundary conditions via Riemann-Hilbert techniques, and we used the formalism to 
solve the initial value problem for the defocusing NLS equation with 
periodic boundary conditions.

The formalism of this work generalizes the one developed in \cite{McLaughlinNabelek} for Hill's operator.
Compared to \cite{McLaughlinNabelek}, one of the main complications in the theory for Dirac operator is that,
the spectrum for Hill's operator is bounded from below, which, even in the case of infinite genus, 
provides a natural starting point for the count of all sequences of eigenvalues 
and also for the definition of the branch points of all square roots and fourth roots.
In contrast, the spectral bands of the Dirac operator extend to infinity in both directions along the real $z$~axis, 
which complicates the analysis.
Another significant difference from the analysis of Hill's equation in \cite{McLaughlinNabelek} is that,
in that case the spectral problem is an eigenvalue problem for a scalar second-order differential operator.
As a result, the asymptotic behavior of various quantities contains square roots of the spectral parameter,
which complicates the analysis.
In this respect, the Zakharov-Shabat problem is simpler, since no square roots of $z$ appear in the analysis.
On the other hand, the Zakharov-Shabat problem is complicated significantly by the need to introduce the
modified Bloch-Floquet solutions.

Recall that, in the finite-genus formalism, 
the Dirichlet eigenvalues associated with all base points in the spatial domain
are needed in order to reconstruct the potential,
and the dependence of the Dirichlet eigenvalues on the base point (as governed by the Dubrovin equations) 
is highly nontrivial,
and is linearized by the Abel map.
With the Riemann-Hilbert formalism, in contrast, only one base point is all that is needed.
Indeed, the correct spatial and temporal dependence of the solution is simply a consequence of the explicit, 
parametric dependence of the jump condition of the RHP on~$x$ and~$t$,
precisely like in the IST formalism on the line.

The present work should also be compared with that in \cite{DeconinckFokasLenells}, 
where the IVP for the focusing and defocusing NLS was 
studied using the unified transform method.  The latter is based on simultaneous analysis of both parts of the 
Lax pair, but where the problem is posed on the finite interval.  
As a result, the formalism requires the use of unknown boundary values, which must then be eliminated.
In constrast, the present work uses the spectral theory of the ZS problem, and is therefore much more analogous
to the IVP on the line, and does not use the values of the potential at the boundary of the domain.

A result of this work is that only one set of Dirichlet eigenvalues is needed for the inverse problem,
just like in the inverse spectral theory for the KdV equation. 
This is in contrast with the  trace formulae obtained from the  finite-gap formalism, 
in which two sets of Dirichlet eigenvalues are used to reconstruct the potential.
The result was illustrated by computing explicitly the solution of the NLS in the case of genus zero in Appendix~\ref{s:genus0}.
In Appendix~\ref{a:alternative} we also showed that either of the two sets of Dirichlet eigenvalues used in the finite-genus formalism
works equally well for the purposes of the present inverse spectral theory.
The fact that only one set of Dirichlet eigenvalues is needed is consistent with the fact that,
for the direct and inverse spectral theory on the line, each norming constant in the defocusing case 
is uniquely determined by one real degree of freedom.
Similarly, in the theta function representation of the solution of the NLS equation, the real divisor also 
contains only one set of real degrees of freedom (cf.\ section~\ref{s:finitegap}).
Note also that, since the RHP~\ref{RHP1} already admits a unique solution, one does not have the option of prescribing
additional data (e.g., such as a second set of Dirichlet eigenvalues) while preserving the solvability of the RHP.
This means that a bijective map must exist between the two sets of Dirichlet eigenvalues.

The results of this work also open a number of interesting problems for future study, 
which can be divided along three main classes.
The first class of questions concerns the NLS equation, and 
a first question in this regard is whether the present results can be used to establish the well-posedness of the 
initial-value problem in some appropriate functional classes.
In this regard, note that (as is usual) 
the IST was formulated under the assumption of existence and uniqueness of solutions of the IVP.  
On the other hand, one can now turn the perspective around and use the results of the present work to
prove the well-posedness of the IVP in appropriate functional classes, similarly to what was done in 
\cite{JenkinsLiuPerrySulem,KappelerTopalov,PelinovskySaalmannShimabukuro,PelinovskyShimabukuro,VillarroelPrinariAblowitz,XinZhou1998}.

Another question is related to the fact that the NLS equation is an infinite-dimensional Hamiltonian system. It is well-known, when the IVP is posed on the line, the system is completely integrable, and the IST can be viewed as a canonical transformation to action-angle variables.
For the KdV equation, action-angle variables with periodic BC are also known \cite{KappelerPoschel2010}.  
An obvious question is therefore what are the action-angle variables for the defocusing NLS equation with periodic BC
in the infinite-genus case.

Yet another interesting question relates to possible existence of a suitable infinite period limit
of the present formalism.  
In other words, the question is whether there is any way in which the formalism can reduce to the IST on the line
with either zero or non-zero BC.  
The question is nontrivial, since the periodic problem
and the problem on the line are fundamentally different.
In particular, for the IST on the line, the value of the potential as $x\to\pm\infty$ must be prescribed (whether it be zero or nonzero), whereas in our case the value of $q(x)$ at $x= \pm L/2$ is not given.
For the NLS equation, this difficulty is also compounded by the fact that the IST on the line differs quite a bit
depending on whether zero or non-zero BC are given.
(For the KdV equation, in contrast, the value of the potential at infinity can always be set without loss of generality 
thanks to the Galilean invariance of the equation.)

A last question in the first class relates to the ``linearization'' of the direct and inverse transform.  
It is well known that, 
for potentials on the line, the inverse scattering transform reduces exactly to the direct and inverse Fourier transform 
in the ``linear limit'', i.e., the limit of vanishingly small potentials~\cite{AKNS1974}. 
A natural question is therefore whether there is any way in which the present formalism reduces to the solution
of the periodic problem by Fourier series in the linear limit.

A second class of open questions concern whether the analysis of the present work can be generalized to other systems.
For example, an obvious interesting question is whether the fact that a single set of Dirichlet eigenvalues suffices 
for the purposes of inverse spectral theory also carries through to the focusing Zakharov-Shabat problem.
In this respect, note that the analysis of the focusing Zakharov-Shabat problem is significantly complicated by
the fact that the corresponding Dirac operator is not self-adjoint, and as a result the eigenvalues
are not relegated to the real $z$-axis \cite{MaAblowitz,McLaughlinOverman}.
This difficulty is also compounded by the fact that,
while the notion of spectral bands and gaps can still be introduced in a relatively straightforward way
\cite{BiondiniOregero2020,BOT_JST2023}, 
the movable Dirichlet eigenvalues are not confined to curves in the complex plane \cite{McKean1981}.

Another possible future extension of the present work concerns the 
periodic problem for the Manakov system, i.e., the integrable two-component coupled nonlinear Schr\"odinger equation.
Realizing such an extension will require overcoming several significant difficulties,
which stem from the fact that the associated scattering problem is a three-component system.
Therefore, even though in the defocusing case the spectral problem is still self-adjoint, 
the characterization of the spectrum is much more complicated and cannot be reduced to the study of a single
entire function like with the KdV and NLS equations.
At the same time, 
we remark that the periodic problem for the Manakov system is still completely open even with the finite-genus formalism.
The reason for this state of affairs is similar: 
since the scattering problem is effectively third-order,
the spectral curve associated with the finite-genus solutions is not hyperelliptic as in the case of the KdV and NLS equation.
For this reason, the analysis of the periodic Manakov system remains as an oustanding open problem.

Finally, a third class of questions relates to the application of the present formalism to study concrete problems.
One such possible application concerns the study of semiclassical limits and small dispersion problems~\cite{bronski1996,DVZ,jenkins2013,KMM2003,LaxLevermore,TV2012},
which continue to receive considerable attention 
(e.g., see \cite{ElHoefer,BiondiniOregero2020,BOT_JST2023,physd2016deng,pre2017deng,prl2016deng} and references therein).
Another possible application concerns the emerging topics of soliton gases and the statistics of integrable systems 
\cite{Congy2021,Congy2023,El2021,ElTovbisPRE2020,Gelash_PRL2019,Gelash_PRE2018,GirottiGravaJenkinsMcLaughlin,GirottiGravaJenkinsMcLMinakov}.
We hope that the results of this work and the present discussion will stimulate further work on all of these topics.

\subsection*{Declarations}

\noindent\textbf{Acknowledgment.}
We thank Marco Bertola, Dionyssis Mantzavinos, Ken McLaughlin, Jeffrey Ore\-gero and Barbara Prinari for many insightful conversations on topics related to the present work.
We especially thank Alexander Tovbis and Zachery Wolski for their careful review of the manuscript and their constructive feedback.
%
This work was partially supported by the National Science Foundation under grant number DMS-2009487.



\section*{Appendix}
\setcounter{section}1
\setcounter{subsection}0
\setcounter{equation}0
\def\thesection{\Alph{section}}
\def\theequation{\Alph{section}.\arabic{equation}}
\def\thetheorem{\Alph{section}.\arabic{theorem}}
\def\thefigure{\Alph{section}.\arabic{figure}}
\addcontentsline{toc}{section}{Appendix}

\subsection{Notation and standard definitions}
\label{a:notation}

We use the Pauli matrices defined as
\be
\sigma_1 = \begin{pmatrix} 0 &1 \\ 1&0 \end{pmatrix}\,,
\qquad
\sigma_2 = \begin{pmatrix} 0 & -i \\ i&0 \end{pmatrix}\,,
\qquad
\sigma_3 = \begin{pmatrix} 1 & 0 \\ 0 &-1 \end{pmatrix}\,.
\label{e:Paulidef}
\ee
Throughout this work, 
subscripts $x$ and $t$ denote partial differentiation,
the superscript $T$ denotes matrix transpose, the asterisk complex conjugation,
and 
the dagger conjugate transpose.
Also, we denote by $f_\re = \half(f + f^*)$ and $f_\im = -\half\i(f - f^*)$ the real and imaginary parts of a complex quantity.
For any $2\times2$ matrix $A = (a_{ij})$, we use the notation $a_1 = (a_{11},a_{21})^T$ and $a_2 = (a_{12},a_{22})^T$ 
to denote its first and second columns, respectively.
The ``half-plane residues'' $\Res\nolimits^{\pm}$ used in section~\ref{s:genus0} are defined as 
\be
\Res_{z=z_o}{\!}^\pm f(z) = \lim_{{z\to z_o^{\phantom I}}\atop{\Im(z-z_o)\gl0}}\!\![(z-z_o)f(z)]\,.
\ee
Finally, we take the signum function $\sgn(z)$ to take zero value when its argument is zero. 
%

\subsection{Basic properties of the spectrum}
\label{a:Dirichlet}

In this appendix we prove some basic properties of the Dirichlet eigenvalues, largely following~\cite{MaAblowitz}.

\begin{proposition}
	For all $z\in\Complex$, the Floquet discriminant $\Delta(z)$ satisfies the Schwarz symmetry: $\Delta^*(z^*)=\Delta(z)$. Subsequently, for all $z\in\Complex/\Real$, the Floquet multiplier $\rho(z)$ satisfies $\rho^*(z^*)=\rho(z)$.
\end{proposition}
\begin{proof}
	From \eqref{e:symY~}, we have that $\Delta^*(z^*)=\half \tr\~Y^*(z^*)=\half \tr\~Y(z)=\Delta(z)$, which directly gives the symmetry for $\rho(z)$.
\end{proof}

\begin{proposition}
All the Dirichlet eigenvalues $\mu_j$  are real.
\end{proposition}
\begin{proof}
Since $\~Y$ a the fundamental solution of \eqref{e:modZS}, 
its second column $\tilde{y_2}$ also satisfies \eqref{e:modZS}. 
Thus,
\be
\~y_{12,x} - z\~y_{22}=q_\re\~y_{12}+q_\im\~y_{22},
\qquad
    \~y_{22,x} + z\~y_{12}=-q_\re\~y_{22}+q_\im\~y_{12}.
\label{zs2}
\ee
Using \eqref{zs2}, it is easy to show that 
\be
	(\~y_{12}\~y_{22}^*-\~y_{22}^*\~y_{12})_x
	  + (z^*-z)(|\~y_{12}|^2+|\~y_{22}|^2)=0;
\ee
hence
\be
	(z^*-z)\int_0^L(|\~y_{12}|^2+|\~y_{22}|^2)\,\d x
	  + \big[ \~y_{12}\~y_{22}^*-\~y_{22}^*\~y_{12}) \big]_0^L = 0\,.
\label{e:tempint}
\ee
Also, by definition we have $\~y_{12}(0,\mu_j)=\~y_{12}(L,\mu_j)=0$.
Evaluating~\eqref{e:tempint} at $z=\mu_j$, the second term in the above equation vanishes. 
Since the integral is strictly positive, we conclude that $\mu_j=\mu_j^*$.
\end{proof}

\noindent
The following result will be useful
in the proof of Proposition~\ref{p:oneDirichlet} later in this section: 

\begin{proposition}
For all $z\in\Real$, the Floquet discriminant $\D(z)$ satisfies
\begin{multline}
\D'(z) = \half\sgn[\~y_{21}(L,z)]\,\left\{\int_0^L\big[(a\~y_{21} + b \~y_{21})^2 + 
  (a\~y_{11}+b\~y_{12})^2\big]\,\d x
		- \frac{\Delta^2-1}{|\~y_{21}(L,z)|}\int_0^L(\~y_{11}^2+\~y_{21}^2)\,\d x \right\},
\label{dd}
\end{multline}
with 
\be
a = \frac{\~y_{22}(L,z)-\~y_{11}(L,z)}{2\sqrt{|\~y_{21}(L,z)|}}\,,
\qquad 
b = -\sgn\~y_{21}(L,z)\sqrt{|\~y_{21}(L,z)|}\,.
\ee
As a result, $\sgn(\D'(z)) = \sgn\~y_{21}(L,z)$ for all $z\in\Real$.
\end{proposition}

\begin{proof}
Let $W(x,z) = \~Y'(x,z) = (w_{ij})$.
Differentiating \eqref{zs2} with respect to $z$ yields
\be
\label{zs3}
	w_{1j,x} - zw_{2j}-q_\re w_{1j}-q_\im w_{2j}=\~y_{2j}\,,
\qquad
    w_{2j,x} + zw_{1j}+q_\re w_{2j}-q_\im w_{1j}=-\~y_{1j}\,.
\ee
Comparing~\eqref{zs3} with~\eqref{zs2}
shows that 
$(w_{1j},w_{2j})^T$ satisfies the same ODE as $(\~y_{1j},\~y_{2j})^T$, 
except for the presence of a ``forcing'' term on the right-hand side of~\eqref{zs3}.
Taking $j=1$ in \eqref{zs3} and expanding $(w_{11},w_{21})^T$ in terms of 
$(\~y_{11},\~y_{21})^T$ and 
$(\~y_{12},\~y_{22})^T$, 
we can write
\be
w_{11}=c_{11}\~y_{11}+c_{12}\~y_{12}\,,\qquad
w_{21}=c_{11}\~y_{21}+c_{12}\~y_{22}\,,
\ee
where $c_{11}$ and $c_{12}$ are functions of $x$,
and using variation of parameters in the ODE we have 
\be
c_{11}=\int_{0}^x(\~y_{21}\~y_{22}+\~y_{11}\~y_{12})(s)\,\d s\,,\qquad
c_{12}=-\int_{0}^x(\~y_{21}^2+\~y_{11}^2)(s)\,\d s\,.
\ee
Then 
\bse
\label{e:mij}
\begin{gather}
	w_{11}(x) = \bigg( \int_0^x(\~y_{21}\~y_{22}+\~y_{11}\~y_{12})(s)\d s \bigg)\,\~y_{11}(x)
	  - \bigg( \int_0^x(\~y_{21}^2+\~y_{11}^2)(s)\d s \bigg)\,\~y_{12}(x)\,,\\
	w_{21}(x) = \bigg(\int_0^x(\~y_{21}\~y_{22}+\~y_{11}\~y_{12})(s)\d s \bigg)\,\~y_{21}(x)
	  - \bigg(\int_0^x(\~y_{21}^2+\~y_{11}^2)(s)\d s)\~y_{22}(x)\,,\\
	w_{12}(x) = \bigg(\int_0^x(\~y_{22}^2+\~y_{12}^2)(s)\d s \bigg)\,\~y_{11}(x)
	  - \bigg(\int_0^x(\~y_{12}\~y_{11}+\~y_{22}\~y_{21})(s)\d s \bigg)\,\~y_{12}(x)\,,\\
	w_{22}(x) = \bigg(\int_0^x(\~y_{22}^2+\~y_{12}^2)(s)\d s \bigg)\,\~y_{21}(x)
	  - \bigg(\int_0^x(\~y_{12}\~y_{11}+\~y_{22}\~y_{21})(s)\d s \bigg)\,\~y_{22}(x)\,.
\end{gather}
\ese
Next, differentiating $\Delta(z)$ with respect to $z$, we have
\vspace*{-1ex}
\begin{multline}
2\Delta'(z) = w_{11}(L,z)+w_{22}(L,z)
		=\left(\int_0^L(\~y_{21}\~y_{22}+\~y_{11}\~y_{22})\d x\right)\~y_{11}(L,z)-\left(\int_0^L(\~y_{21}^2+\~y_{11}^2)\d x\right)\~y_{12}(L,z)
\\
		+ \left(\int_0^L(\~y_{22}^2+\~y_{12}^2)\d x\right)\~y_{21}(L,z)-\left(\int_0^L(\~y_{12}\~y_{11}+\~y_{22}\~y_{21})\d x\right)\~y_{22}(L,z)\,.
 \label{dD}
\end{multline}
Assuming for the moment that $\~y_{21}(L,z)\neq0$, we have
\be
\~y_{12}(L,z)=\sgn\~y_{21}(L,z)\frac{4(\Delta^2-1)-[\~y_{11}(L,z)-\~y_{22}(L,z)]^2}{4|\~y_{21}(L,z)|}\,.
\ee
Next, to evaluate $\Delta'(z)$ we write
\be
-\~y_{12}(L)\int_0^L\~y_{21}\d x
  = -\sgn\~y_{21}(L,z)\frac{\Delta^2-1}{\~y_{21}(L,z)}\int_0^L\~y_{21}^2\d x
	+ \sgn\~y_{21}(L,z)\frac{[\~y_{11}(L,z)-\~y_{22}(L,z)]^2}{4|\~y_{21}(L,z)|}\int_0^L\~y_{21}^2\d x,
\nonumber
\ee
and
\be
\~y_{21}(L)\int_0^L\~y_{22}^2\d x=\sgn\~y_{21}(L,z)|\~y_{21}(L,z)|\int_0^L\~y_{22}^2\d x[\sgn \~y_{21}(L,z)]^2.
\nonumber
\ee
Then, after some straightforward calculations, we finally obtain~\eqref{dd}.
\end{proof}

\begin{proposition}
All the Dirichlet eigenvalues $\mu_j$s are simple roots of $\~y_{12}(L,z)$.
\end{proposition}
\begin{proof}
Since $\mu_j$ is a root of $\~y_{12}(L,z)$, we have $\~y_{11}(L,z)\~y_{22}(L,z)=1$, which implies that $\~y_{11}(L,z)\neq 0$. Therefore, from \eqref{e:mij}, we have $w_{12}(L,z)\neq 0$. 
Thus, we conclude that $\mu_j$ is a simple root of $\~y_{12}(L,z)$.
\end{proof}

\begin{proposition}
All the main eigenvalues are simple or double roots of $\Delta^2(z)-1$.
\end{proposition}
\begin{proof}
Here we only prove the result for the periodic band edges $\hat{\zeta}_j$. 
The proof for the antiperiodic band edges follows similarly. 
First we claim that $\~y_{21}(L,\hat{\zeta}_j)=0$. 
At $z=\hat{\zeta}_j$, we also have
\bse
\begin{gather}
    \~y_{11}(L,\hat{\zeta}_j)+\~y_{22}(L,\hat{\zeta}_j)=2,\qquad
    \~y_{12}(L,\hat{\zeta}_j)=0.
\end{gather}
\ese
Hence, we have
\be
	\~y_{11}(L,\hat{\zeta}_j)=\~y_{22}(L,\hat{\zeta}_j)=1.
\ee
So from \eqref{e:mij}, we obtain
\begin{gather*}
    w_{11}(L,\hat{\zeta}_j)=\int_0^L(\~y_{21}\~y_{22}+\~y_{11}\~y_{12})(x)\d x,\qquad
    w_{21}(L,\hat{\zeta}_j)=-\int_0^L(\~y_{21}^2+\~y_{11}^2)(x)\d x,\\
    w_{12}(L,\hat{\zeta}_j)=\int_0^L(\~y_{22}^2+\~y_{12}^2)(x)\d x,\qquad
    w_{22}(L,\hat{\zeta}_j)=-\int_0^L(\~y_{12}\~y_{11}+\~y_{22}\~y_{21})(x)\d x.
\end{gather*}
From \eqref{dD}, we have 
\be
	2\frac{d^2\Delta}{d\zeta^2}=2\left\{\left(\int_0^L(\~y_{21}\~y_{22}+\~y_{11}\~y_{22})(x)\d x\right)^2-\left(\int_0^L(\~y_{21}^2+\~y_{11}^2)\d x\right)\left(\int_0^L(\~y_{22}^2+\~y_{21}^2)\d x\right)\right\}\leq 0\,,
\ee
where the last inequality follows from the Cauchy-Schwarz inequality. 
The equality sign occurs when $\~y_1$ and $\~y_2$ are linearly dependent. 
This can never happen, however, since $\~y_1$ and $\~y_2$ are linearly independent. 
Therefore, we have $d^2\Delta/d\zeta^2<0$.
\end{proof}

\begin{proposition}
There is exactly one Dirichlet eigenvalue $\mu_j$ in each degenerate or non-degenerate spectral gap.  
\label{p:oneDirichlet}
\end{proposition}

\begin{proof}
We first prove that there is at least one Dirichlet eigenvalue in each spectral gap.
At two adjacent band edges $\zeta_j$ and $\zeta_{j+1}$, $\D'(\zeta_j)$ and $\D'(\zeta_{j+1})$ have different signs. 
From \eqref{dd} this means that $\~y_{21}(L,\zeta_j)$ and $\~y_{21}(L,\zeta_{j+1})$ have different signs. 
By continuity, there must be $\mu_j$ between $\zeta_j$ and $\zeta_{j+1}$, such that $\~y_{21}(L,\mu_j)=0$. 
We have thus proved the first part of the proposition.

Next we prove that there is at most one Dirichlet eigenvalue in each spectral gap.  Since $\det\~Y(x;z)=1$, we have 
$\~y_{11}(L,\mu_j)\~y_{22}(L,\mu_j)-\~y_{12}(L,\mu_j)\~y_{21}(L,\mu_j)=1$,
implying 
\begin{equation*}
	\~y_{22}(L,\mu_j)=\frac{1}{\~y_{11}(L,\mu_j)},
\qquad
    2\Delta(\mu_j)=\~y_{11}(L,\mu_j)+\frac{1}{\~y_{11}(L,\mu_j)}.
\end{equation*}
Therefore,
\be
\~y_{11}(L,\mu_j)=\Delta(\mu_j)+\sqrt{\Delta^2(\mu_j)-1}.
\ee
It is clear that $\~y_{11}(L,\mu_j)>0$ if $\Delta(\mu_j)>1$, and $\~y_{11}(L,\mu_j)<0$ if $\Delta(\mu_j)<-1$.
Then we have 
\be
	w_{21}(L,\mu_j)=-\~y_{22}(L,\mu_j)\int_0^L(\~y_{21}^2+\~y_{11}^2)\d x=-\frac{1}{\~y_{11}(L,\mu_j)}\int_0^L(\~y_{21}^2+\~y_{11}^2)\d x\,.
\ee
This implies that $w_{21}(L,\mu_j)$ can only have a definite sign in each unstable band, so that $\~y_{21}(L,\mu_j)$ is monotonic in one band. 
This means that there is at most one $\mu_j$ in each unstable band.
\end{proof}

\subsection{Asymptotics of the modified fundamental matrix as $z\to\infty$}
\label{a:Ytildeasymp}

\begin{proof}[Proof of Lemma~\ref{l:asymY}]
We want to derive the asymptotic expansion of $Y(x,z)$ as $z\to\infty$ up to order $1/z^2$.
Let
\be
\label{e:defmuxz}
\mu(x,z)=Y(x,z)\,\e^{\i zx\sigma_3}\,,
\ee
and expand $\mu(x,z)$ as  
\be
\mu(x,z)= \sum_{n=0}^\infty \frac{\mu^{(n)}(x,z)}{z^n}\,.
\label{e:muexpansion}
\ee 
If $q\in L^2([0,L],\Complex)$, 
integration by parts on the integral equation~\eqref{phi} 
then yields the asymptotics of $\mu$ as $z\to\infty$ along the real axis:
\bse
\label{e:asymmu}
\begin{gather}
\begin{align}
&\mu^{(0)}(x,z) = \I\,,\qquad 
\mu^{(1)}(x,z) = \frac{1}{2\i z}\sigma_3\left(\bigg(Q(x)-\e^{-2\i zx\sigma_3}Q(0)\bigg)-\int_0^xQ^2(s)\d s\right)\,,\\
&\begin{aligned}
\mu^{(2)}(x,z) = \frac{1}{4z^2}\Bigg(&Q_x(x)-\e^{-2\i zx\sigma_3}Q_x(0)-\bigg(\e^{2 \i z x\sigma_3} Q(x) - Q(0)\bigg)Q(0)+\int_0^x Q(s)  Q_x(s) \d s
	\\&\kern8em
	-Q(x)\int_0^x Q^2(s)\d s -\half\bigg(\int_0^x Q^2(s)\d s\bigg)^2-\e^{-2\i zx\sigma_3}\bigg(Q^2(s)\d s \bigg)Q(0)
	\Bigg)\,.
	\end{aligned}
\end{align}
\end{gather}
\ese
Combining these expressions then yields~\eqref{asymY} for $z\in\Real$.

When $z\not\in\Real$, it is necessary to consider separately the first and second column of $\mu(x,z)$.
Nonetheless, one can rigorously prove \cite{BOT_JST2023} that, 
if $q\in C^2([0,L])$, the asymptotic behavior of each of the columns of $\mu(x,z)$ also holds as $z\to\infty$ 
for $z\in\overline{\Complex^+}$
(e.g., see Lemmas~6.1 and~Lemma~6.2 in \cite{BOT_JST2023} for the first and second column of $\mu$, respectively).
For brevity we will not repeat the proof here.

Finally, the asymptotic behavior as $z\to\infty$ in the lower-half plane is obtained by noting that, 
if $A(x,z)$ is any matrix solution of~\eqref{e:zs}, so is $\sigma_2 A^*(x,z^*)$.
Then, taking into account that $Y(0,z) = \I$ for all $z\in\Complex$, we have
\be
Y(x,z^*) = \sigma_2 Y^*(x,z) \sigma_2\,,
\ee
which yields the desired behavior as $z\to\infty$ in $\overline{\Complex^-}$.
\end{proof}

Next, combining \eqref{e:defmuxz}, \eqref{e:muexpansion} and \eqref{e:asymmu}, and 
recalling the definition~\eqref{e:Ytildedef} of $\~Y(x,z)$, 
tedious but straightforward calculations then yield the following asymptotic behavior for 
the entries of $\~Y(x,z)$ as $z\to \infty$ along the real axis:
\bse
\begin{align}
\label{asymy11tilde}
&\begin{aligned}
	\~y_{11}(x,z)=\frac{1}{2}\e^{-\i zx}\Bigg(1-\frac{1}{2\i z}K(x)-\frac{1}{2\i z}&(q(0)+q^*(x))+\frac{1}{4z^2}\bigg(|q(0)|^2+\Lambda(x)-\half K^2(x)
	\\
	&-q(0)q^*(x)+q_x^*(x)-q^*(x)K(x)-q(0)K(x)-q_x(0)\bigg)+O\left(\frac{1}{z^3}\right)\Bigg)
	\\
	+\frac{1}{2}\e^{\i zx}\Bigg(1+\frac{1}{2\i z}K(x)+\frac{1}{2\i z}&(q(x)+q^*(0))+\frac{1}{4z^2}\bigg(|q(0)|^2+\Lambda^*(x)-\half K^2(x)
	\\
	&-q^*(0)q(x)+q_x(x)-q(x)K(x)-q^*(0)K(x)-q_x^*(0)\bigg)+O\left(\frac{1}{z^3}\right)\Bigg)\,,
\end{aligned}
\\
\label{asymy12tilde}
&\begin{aligned}
		\~y_{12}(x,z)=\frac{\i}{2}\e^{-\i zx}\Bigg(1-\frac{1}{2\i z}K(x)+\frac{1}{2\i z}&(q(0)-q^*(x))+\frac{1}{4z^2}\bigg(|q(0)|^2+\Lambda(x)-\half K^2(x)
	\\
	&+q(0)q^*(x)+q_x^*(x)-q^*(x)K(x)+q(0)K(x)+q_x(0)\bigg)+O\left(\frac{1}{z^3}\right)\Bigg)
	\\
	-\frac{\i}{2}\e^{\i zx}\Bigg(1+\frac{1}{2\i z}K(x)+\frac{1}{2\i z}&(q(x)-q^*(0))+\frac{1}{4z^2}\bigg(|q(0)|^2+\Lambda^*(x)-\half K^2(x)
	\\
	&+q^*(0)q(x)+q_x(x)-q(x)K(x)+q^*(0)K(x)+q_x^*(0)\bigg)+O\left(\frac{1}{z^3}\right)\Bigg)\,,
   \end{aligned}
	\\
&\begin{aligned}
\label{asymy21tilde}
	\~y_{21}(x,z)=-\frac{\i}{2}\e^{-\i zx}\Bigg(1-\frac{1}{2\i z}K(x)+\frac{1}{2\i z}&(q^*(x)-q(0))+\frac{1}{4z^2}\bigg(|q(0)|^2+\Lambda(x)-\half K^2(x)
	\\
	&+q(0)q^*(x)-q_x^*(x)+q^*(x)K(x)-q(0)K(x)-q_x(0)\bigg)+O\left(\frac{1}{z^3}\right)\Bigg)
	\\
	+\frac{\i}{2}\e^{\i zx}\Bigg(1+\frac{1}{2\i z}K(x)+\frac{1}{2\i z}&(q^*(0)-q(x))+\frac{1}{4z^2}\bigg(|q(0)|^2+\Lambda^*(x)-\half K^2(x)
	\\
	&+q^*(0)q(x)-q_x(x)+q(x)K(x)+q^*(0)K(x)-q_x^*(0)\bigg)+O\left(\frac{1}{z^3}\right)\Bigg)\,,
\end{aligned}
\\
	&\begin{aligned}\label{asymy22tilde}
		\~y_{22}(x,z)=\frac{1}{2}\e^{-\i zx}\Bigg(1-\frac{1}{2\i z}K(x)+\frac{1}{2\i z}&(q(0)+q^*(x))+\frac{1}{4z^2}\bigg(|q(0)|^2 + \Lambda(x)-\half K^2(x)
		\\
		&-q(0)q^*(x)-q_x^*(x)+q^*(x)K(x)+q(0)K(x)+q_x(0)\bigg)+O\left(\frac{1}{z^3}\right)\Bigg)
		\\
		+\frac{1}{2}\e^{\i zx}\Bigg(1+\frac{1}{2\i z}K(x)-\frac{1}{2\i z}&(q(x)+q^*(0))+\frac{1}{4z^2}\bigg(|q(0)|^2+\Lambda^*(x)-\half K^2(x)
		\\
		&-q^*(0)q(x)-q_x(x)+q(x)K(x)+q^*(0)K(x)+q_x^*(0)\bigg)+O\left(\frac{1}{z^3}\right)\Bigg)\,,
	\end{aligned}
\end{align}
\ese
where 
\be
K(x) = \int_0^x \ |q(s)|^2\d s\,,\qquad
\Lambda(x) = \int_{0}^x q(s)q_x^*(s)\d s\,.
\label{e:Kdef}
\ee
Importantly, similarly to~\eqref{asymY},
these expressions also hold off the real $z$-axis, for the same reasons as in the proof of Lemma~\ref{l:asymY}.
The above expressions are used in the proof of Lemma~\ref{asympsi0}.

\subsection{Asymptotics of the modified Bloch-Floquet solutions as $z\to\infty$}
\label{asymofpsipm}

Here we derive the expressions in Proposition~\ref{p:BFyexpansion} for $\psi^\pm(x,z)$,
as well as their asymptotic behavior as $z\to\infty$ given in Propositions~\ref{p:asympsi}
and~\ref{asympsi0}.

\begin{proof}[Proof of Proposition~\ref{p:BFyexpansion}]
Since $\~y_1$ and $\~y_2$ are independent solutions, we can expand each 
Bloch eigenfunction $\psi = (\psi_1,\psi_2)^T$ in terms of them:
\be
\psi(x,z) = c_1 \~y_1(x,z) + c_2 \~y_2(x,z)
\label{e:BF_yexpansion}
\ee
We can also express $\~y_1(x+L,z) = (\~y_{11}(x+L,z),\~y_{21}(x+L,z))^T$ in a similar way:
\bse
\label{e:ytildex+L}
\begin{gather}
\~y_1(x+L,z) = \~y_{11}(L,z) \~y_1(x,z) + \~y_{21}(L,z) \~y_2(x,z)\,,
\\
\~y_2(x+L,z) = \~y_{11}(L,z) \~y_1(x,z) + \~y_{22}(L,z) \~y_2(x,z)\,,
\end{gather}
\ese
where the coefficients of the expansion are easily obtained by evaluating the above expressions at $x = 0$.
Evaluating~\eqref{e:BF_yexpansion} at $x+L$ and using~\eqref{e:ytildex+L}, we then have
\be
\psi(x+L,z) = c_1 (\~y_{11}(L,z) \~y_1(x,z) + \~y_{21}(L,z) \~y_2(x,z) ) + c_2 (\~y_{12}(L,z) \~y_1(x,z) + \~y_{22}(L,z) \~y_2(x,z) )\,.
%
\ee
Hence,
\be
	c_1(\~y_{11}(L)-\rho)+c_2\~y_{12}(L)=0,
\qquad
	c_1\~y_{21}(L)+c_2(\~y_{22}(L)-\rho)=0.
\ee
Recalling that the Floquet multipliers $\rho^{\pm1}$ are precisely the values for which the determinant of the above system is zero,
solving for $c_1$ and $c_2$, 
and imposing the normalization $\psi_1(0) = 1$,
one then obtains~\eqref{BF}.
\end{proof}

\begin{proof}[Proof of Proposition~\ref{p:asympsi}]
From \eqref{rho} and \eqref{asymY}, as $z\to\infty$, we have
\be
\label{asymprhoinverse}
\rho^{-1}(z) = 
\e^{\mp\i zL}\bigg(1 \mp \frac{1}{2\i z}K_o+O(1/z^2)\bigg), \qquad z\in\Complex^\pm\,,
\ee
with $K_o = K(L)$ and $K(x)$ is as in~\eqref{e:Kdef}, implying 
\begin{align}
\log(\rho(z)) = \pm \i zL \pm \frac{1}{2\i z}K_o+O(1/z^2),\qquad z\in\Complex^\pm\,.
\label{e:logrhoasymp}
\end{align}
We thus have the following asymptotic behavior for $\rho^{-x/L}(z)$:
\begin{equation}
\rho^{-x/L}(z) = 
			\e^{\mp\i zx}\bigg(1\mp\frac{K_o x}{2\i L z}+O(1/z^2)\bigg),\qquad z\in\Complex^\pm\,.
\label{e:rhoexpxoverL}
\end{equation}
\bm
Then Proposition~\ref{p:BFyexpansion} follows from Proposition~\ref{p:BFyexpansion} and the asymptotic behavior of $\~Y(z)$.  
On the other hand, we next provide an alternative proof.
\em
By Floquet's theorem, we can write 
\be
	(\psi^-(x,z),\psi^+(x,z)) =: P(x,z)\,\rho^{-x/L\sigma_3}(z)\,,
\label{Prho}
\ee
where $P(x,z) = (p^-(x,z),p^+(x,z))$ 
and where $p^\pm(x,z)$ are periodic in~$x$ with period $L$. 
Inserting~\eqref{Prho} into \eqref{e:modZS} we have
\be
	P_x(x,z)=L^{-1}\log(\rho(z))P(x,z)\sigma_3+\i z\sigma_2P(x,z)+\tilde{Q}(x)P(x,z)\,,
\label{e:P_ODE}
\ee
where $\tilde{Q}(x)=UQU^{-1}$.
To compute the asymptotic behavior of $P(x,z)$, we
convert the above first-order system of ODEs into a scalar second-order ODE.
Consider $p^-(x,z)$ first for simplicity. 
Equation~\eqref{e:P_ODE} implies
\be
	p_x^-=L^{-1}\log(\rho(z))P(x,z)+\i z\sigma_2P(x,z)+\tilde{Q}(x)P(x,z)\,,
\ee
which can be converted to integral form as
\vspace*{-0.6ex}
\begin{multline}
	p^-(x,z)=p^-(0,z)-\frac{L}{2\log\rho}p_x^-(0,z)(1-\rho^{2x/L})
\\
-\frac{L}{2\log\rho}\int_0^x(1-\rho^{-2(s-x)/L})(-\frac{1}{L^2}(\log\rho)^2-z^2+|q|^2+\tilde{Q}_x)p^-(s,z)\d s\,.
\end{multline}
Now we consider the Neumann series expansion
\bse
\label{defp}
\begin{gather}
p^-(x,z) = \sum_{n\in\Natural} p^{(n)}(x,z),
\qquad
p^{(0)}(x,z) = p^-(0,z)-\frac{L}{2\log\rho}p_x^-(0,z)(1-\rho^{2x/L}).
\\
p^{(n+1)}(x,z) = -\frac{L}{2\log\rho}\int_0^x(1-\rho^{-2(s-x)/L})
  \bigg(-\frac{1}{L^2}(\log\rho)^2-z^2+|q|^2+\tilde{Q}_x\bigg)\,p^{(n)}(s,z)\,\d s\,,
\qquad n\in\Natural\,.
\end{gather}
\ese
The asymptotic behavior of $\log\rho$ in~\eqref{e:logrhoasymp} implies
\begin{align}
    &-\frac{1}{L^2}(\log\rho)^2-z^2=-K_o/L+O(1/z),\label{asymln}\\
\noalign{\noindent and, in turn,}
&\frac{L}{2\log\rho}= \pm \frac{1}{2\i z} \pm \frac{K_o}{4\i z^3}+O(z^{-4})\,,\qquad z\in\Complex^\pm\,.
\label{asymL/ln}
\end{align}
Equation~\eqref{asymln} allows one to prove that the Neumann series~\eqref{defp} is convergent. 
Next we compute the asymptotic behavior of $p^{(0)}$.
From~\eqref{BF} we have
\vspace*{-0.8ex}
\be
	p^-(0,z)=\psi^-(0,z) = \begin{pmatrix} 1 \\ \frac{\rho^{-1}-\~y_{11}(L,z)}{\~y_{12}(L,z)} \end{pmatrix}\,.
\ee
Moreover, from \eqref{e:modZS}, we have $\psi^\pm_x=(-\i z\sigma_2+\tilde{Q}(x))\psi^\pm$, which yields
	\be
	\psi_x^-(0,z)=\~y_{1,x}(0,z)+\frac{\rho^{-1}-\~y_{11}(L,z)}{\~y_{12}(L,z)}\~y_{2,x}(0,z)=\left(\begin{array}{c}q_\re(0)+(z+q_\im(z))\psi_2^-(0,z)\\-z+q_\im(0)-q_\re(0)\psi^-_2(0,z)\end{array}\right).
	\ee
Thus we have 
\be
	\label{p-x0}
	p_x^-(0,z)=\frac{1}{L}\log\rho(z)\,\psi^-(0,z)+\psi_x^-(0,z)\,.
\ee
We compute the asymptotic behavior of the two components of $p^{(0)}$ separately.
Equation~\eqref{e:rhoexpxoverL} implies $\rho^{2x/L}(z) = 1+o(1)$ as $z\to\infty$ for $z\in\overline{\Complex^+}$ and $x\in(0,L)$ and for $z\in\overline{\Complex^-}$ and $x\in(-L,0)$.
Combining \eqref{defp}, \eqref{asymL/ln} and \eqref{p-x0}, the asymptotic behavior of the first component of $p^{(0)}$ can be expressed as
\begin{align}
p^{(0)}_1(x,z) =1-\half-\frac{L}{2\log\rho}q_\re(0)-\frac{L}{2\log\rho}(z+q_\im(0))\psi^-_2(0,z)
    = 1+O(1/z^2)
\end{align}
as $z\in\overline{\Complex^+}$ and $x\in(0,L)$, 
where the subscripts 1 and 2 denote the components of $p^{(0)}$. 
Similarly, as $z\in\overline{\Complex^-}$ and $x\in(-L,0)$, we have 
\be
	p^{(0)}_1(x,z)=1+O(1/z^2)\,.
\ee
For the second component of $p^{(0)}$, as $z\in\overline{\Complex^+}$ and $x\in(0,L)$ we have
\begin{align}
p^{(0)}_2(x,z) = \psi^-_2(0,z)-\half\psi_2^-(0,z)-\frac{L}{2\log\rho}(z-q_\im(0)+q_\re(0)\psi_2^-(0,z))
        = -\i-\frac{q^*(0)}{z}+O(1/z^2)\,.
\end{align}
Similarly, as $z\in\overline{\Complex^-}$ and $x\in(-L,0)$, we have 
\be
p^{(0)}_2(x,z)=\i-\frac{q(0)}{z}+O(1/z^2)\,.
\ee
Next we consider the asymptotic behavior of $p^{(1)}$.
It is obvious that $ 1-\rho^{-2(s-x)/L} = 1 + O(1/z)$ for $z\in\overline{\Complex^+}$ and $0 < s\le x< L$ and for $z\in\overline{\Complex^-}$ and 
$- L < x \le s < 0$, 
so that 
\be
    p^{(1)}(x,z)=-\frac{L}{2\log\rho}\int_0^x(-K_o/L+|q|^2+\tilde{Q}_x)\,p^{(0)}\d s\,.
\ee
Therefore, we can write the first element of $p^{(1)}$ as
	\be
	p^{(1)}_1(x,z)=-\frac{L}{2\log\rho}\left[\int_0^x(-K_o/L+|q|^2+q_{\re,x})p^{(0)}_1\d s+\int_0^x(q_{\im,x})p^{(0)}_2\d s\right]\,.
	\ee
Plugging the asymptotic behaviors of $p^{(0)}_1$, $p^{(0)}_2$ and ${L}/({2\log\rho})$ into the above equation, 
we obtain the asymptotic expression for $p^{(1)}_1$:
\bse
\be
	p^{(1)}_1=\frac{K_ox}{2\i zL} - \frac{1}{2\i z}\int_0^x|q|^2\d s-\frac{1}{2\i z}(q^*(x)-q^*(0))+O(1/z^2)\,.
\ee
Moreover, following the same method, we have the following asymptotic behavior of $p^{(1)}_2$:
\be
	p^{(1)}_2=-\frac{K_ox}{2zL} + \frac{1}{2z}\int_0^x|q|^2\d s-\frac{1}{2z}(q^*(0)-q^*(x))+O(1/z^2)\,.
\ee
\ese
Therefore, the Neumann series for $p^-$ yields
\bse
\begin{gather}
p^-(x,z)= \begin{pmatrix} 1 \\ -\i \end{pmatrix}
  + \begin{pmatrix} -\i \\ -1 \end{pmatrix} \frac{K_o x}{2z L}
  + \begin{pmatrix} \i \\ 1 \end{pmatrix} \frac1{2z} \int_0^x|q(s)|^2\,\d s
  + \begin{pmatrix} \i \\ -1 \end{pmatrix} \frac 1{2 z}q^*(x) +  \begin{pmatrix} -\i \\ -1 \end{pmatrix} \frac 1{2 z} q^*(0)+O(1/z^2)\,,\quad z\in\Complex^+\,,\\
  p^-(x,z)= \begin{pmatrix} 1 \\ \i \end{pmatrix}
  + \begin{pmatrix} \i \\ -1 \end{pmatrix} \frac{K_o x}{2z L}
  + \begin{pmatrix} -\i \\ 1 \end{pmatrix} \frac1{2z} \int_0^x|q(s)|^2\,\d s
  + \begin{pmatrix} -\i \\ -1 \end{pmatrix} \frac 1{2 z}q(x) +  \begin{pmatrix} \i \\ -1 \end{pmatrix} \frac 1{2 z} q(0)+O(1/z^2)\,,\quad z\in\Complex^-\,.
\end{gather}
\ese
Combining~\eqref{asymprhoinverse}, \eqref{Prho} and the above expression, we obtain~\eqref{asympsi-}.
Similar calculations yield~\eqref{asympsi+}.  
The details are omitted for brevity.
Importantly, note that $K_o$ does not appear there.
\end{proof}

\begin{proof}[Proof of Proposition \ref{asympsi0}]
From (\ref{rho}), it is obvious that, as $z\to\infty$ in $\Complex$,  
	\be
	\rho(z)=O(1/\Delta(z)),\qquad
	\rho^{-1}(z)=2\Delta(z)+O(1/\Delta(z)).
	\ee
From Proposition \ref{asymD}, $O(\Delta^{-1}(z))$ decays exponentially as $z\to\infty$ and $z\in\Complex$. Therefore, 
we obtain the asymptotic behavior for $\psi_2^-(x,z)$ as:
\be
\label{psi-20}
    \psi_2^-(0,z)=\frac{2\Delta(z)-\~y_{11}(L,z)}{\~y_{12}(L,z)}=\frac{\~y_{22}(L,z)}{\~y_{12}(L,z)}+O(1/z)
    = \mp\i + O(1/z)\,\qquad z\in\Complex^\pm\,,
\ee
where we have used the asymptotics of $\~y_{12}$ in \eqref{asymy12tilde}  and $\~y_{22}$  in \eqref{asymy22tilde}
(cf.\ Appendix~\ref{a:Ytildeasymp}).
Similarly, using the asymptotics of $\~y_{11}$ in~\eqref{asymy11tilde},
we find the asymptotic behavior of $\psi_2^+(0,z)$ as
\be
\label{psi+20}
    \psi_2^+(0,z) = -\frac{\~y_{11}(L,z)}{\~y_{12}(L,z)}+O(1/z^3)
	= \pm\i + O(1/z)\,,\qquad z\in\Complex^\pm\,.
\ee
Moreover, taking into account the higher order terms in the asymptotic expansion of $\~Y = (\~y_{ij})$, 
\eqref{psi-20} and \eqref{psi+20} yield \eqref{asymppsi-20} and~\eqref{asymppsi+20}.
\end{proof}

\subsection{Order of growth of eigenfunctions and monodromy matrix at infinity}
\label{ordereigenf}
\begin{proposition}
The order of growth of $y_{ij}$ and $\tilde{M}_{12}$ is~1.
\end{proposition}

\begin{proof}
We prove in detail the claim for $\~M_{12}$;
the proof for $\~y_{ij}$ follows identical arguments.
Recall that an entire function $f$ has order of growth $p$ as $z\to\infty$ if $p$ is the infimum of all positive numbers such that
$|f(z)|\le A\,\e^{B|z|^p}$ in a neighborhood of $\infty$ for some positive constants $A$ and~$B$.
Accordingly, we need to show there exists a constant $C>0$ such that 
\vspace*{-0.4ex}
\be
|\tilde{M}_{12}(z)|\exp(-C|z|)
\label{o1}
\ee
is bounded for all $z$ and a constant $c>0$ such that, for $z=-\i k$ and $k\in\mathbb{R}^+$,
as $k\to\infty$ one has
\be
|\tilde{M}_{12}(z)|\,\exp(-c|z|)\to\infty\,.
\label{o2}
\ee 
We apply Picard's iteration to the differential equation~\eqref{e:zs}. 
Letting $Y=(y_1,y_2)$ as before, 
we expand $y_1$ as 
\bse
\be
y_1(x,z)=\sum_{n=0}^\infty y_1^{(n)}(x,z),
\ee 
where 
\vspace*{-1ex}
\begin{gather}
y_1^{(0)}(x,z)=\begin{pmatrix}\e^{-izx}\\0\end{pmatrix},
\\
y_1^{(n)}(x,z)=\int_0^x\e^{-\i z\sigma_3(x-s)}Q(s)y_1^{(n-1)}(s,z)\,\d s\,,
\quad n>0\,.
\label{p1n}
\end{gather}
\ese
Using the $L^1$ vector norm $\|y_1\|=|y_{11}|+|y_{21}|$ and the corresponding subordinate matrix norm,
from the above definition it is obvious that, for all $x\geq0$,
\be
\|y_1^{(0)}\|\leq \e^{|z|x}.
\ee
Now let $K^*>0$ be a constant such that for all real values of $x$, $\|Q(x)\|\leq K^*$.  Then by induction from \eqref{p1n}, we find that 
\be
\|y_1^{(n)}\|\leq2\frac{(K^*x)^n}{n!}\e^{|z|x}.
\ee
By a similar method, we can also find
\begin{equation*}
\|y_2^{(n)}\|\leq2\frac{(K^*x)^n}{n!}\e^{|z|x}.
\end{equation*}
Using the definition of $\tilde{M}$, we then have
\be
|\tilde{M}_{12}|\leq\half(|y_1|+|y_2|)\leq2\sum_{n=0}^{\infty}\frac{(K^*L)^n}{n!}\e^{|z|L}=2e^{|z|L}\e^{K^*L}.
\ee
This proves that the expression \eqref{o1} is bounded if we choose $C=L$. 
In order to complete the proof, we will use \eqref{asymY}. 
Taking $z=-\i k$, 
it is easy to see that as $k\to\infty$,
\be 
|\sin(zL)| = \half\e^{kL}(1+O(1/k)),\quad 
|\cos(zL)| =  \half\e^{kL}(1+O(1/k)).
\nonumber
\ee
This proves that \eqref{o2} is true for any $c$ between $0$ and $L$. 
So the proof is complete.
\end{proof}

\subsection{Phragmen-Lindel\"{o}f theorem}
\label{s:phragmenlindelof}

Here we obtain some technical results that are needed in order to prove that the matrix function $\Phi(x,z)$ satisfies the RHP~\ref{RHP1}.
We begin by stating without proof a few results from~\cite{McLaughlinNabelek}.

\begin{theorem}[Phragmen-Lindel\"{o}f]
\label{t:PhragmenLindelof}
Let $\alpha$, $\beta\in\mathbb{R}$ with $\alpha<\beta$ be such that $\beta-\alpha<2\pi$ and let $p>0$ satisfy $p<{\pi}/{|\beta-\alpha|}$. 
Choose a branch of the complex argument function so that the interval $(\alpha,\beta)$ is a subset of the image of $\arg z$. 
Let $f(\lambda)$ be a holomorphic function for 
\be
 	z\in\Omega_{\alpha,\beta}:=\{z\in\mathbb{C}:\alpha<\arg z<\beta\}
\ee
that extends to a continuous function on $\overline\Omega_{\alpha,\beta}$. If $f(z)$ is bounded for $z\in\partial\Omega_{\alpha,\beta}$ and $|f(z)|\leq M\e^{c|z|^p}$ for $z\in\Omega_{\alpha,\beta}$, then $f(z)$ is bounded for $z\in\Omega_{\alpha,\beta}$.
\end{theorem}
 
\begin{definition}
Given a domain $\Omega\subset\mathbb{C}$, let $\mathcal{A}(\Omega)$ be the algebra of holomophic functions on $\Omega$. 
We define the Phragmen-Lindel\"{o}f subalgebras $\A_p(\Omega)$ as follows:
\be
\A_p(\Omega)=\{f\in\A(\Omega): \exists c,M>0~ \text{s.t.}~|f(z)|\leq M\e^{c|z|^p}\}.
\ee
\end{definition}
\begin{proposition}
\label{3.1}
 	If $\Omega_1\subset\Omega$ and $p_1\geq p$, then $\A_p(\Omega)\subset\A_{p_1}(\mathbb{\Omega}_1)$.
 \end{proposition}
 \begin{proposition}
\label{3.2}
If $f(z)$ is an entire function of order $q$, and $q<p$, then $f\in\A_p(\C)$.
\end{proposition}

\begin{proposition}
$\A_p(\Omega)$ is a subalgebra of $\A(\Omega)$. 
Moreover, if $f\in\A_p(\Omega)$ and $0<q<1$, then $f^q\in\A_p(\Omega')$ where $\Omega'\subset\Omega$ is some domain on which $f^q$ can be defined as a single-valued function.
\end{proposition}

\begin{proposition}
Let $\{\overline{D}_n\}_{n\in\Integer}$ be are a family of closed discs of radius $R>0$ for which there exists some $N$ such that $\overline{D}_n$ are disjoint for $|n|>N$. Then the domain
\be
 \Omega=\C\setminus\bigcup_{n\in\Integer}\overline{D}_n
\ee
is an open set.
\end{proposition}
 
\begin{proposition}
\label{3.4}
Let $\Omega\subset\C$ be a domain, 
and let $\overline{D}_n\subset\Omega$ for $|n|=1,2,...,\mathcal{N}$ 
(with $\mathcal{N}$ either finite or infinite) be a family of discs with radius $R>0$.
When $\mathcal{N}$ is infinite, assume further that there exists $N\in\Natural$ such that the $\overline{D}_n$ are disjoint for $|n|>N$. 
Then 
\be
\A(\Omega)\cap\A_p\left(\Omega\setminus\bigcup_{|n|=1}^\mathcal{N}\overline{D}_n\right)\subset\A_p(\Omega).
\ee
\end{proposition}

Next we prove some additional results which are generalizations of the above ones, 
and which are needed in our case because of the different form of the infinite product expansions~\eqref{e:Hadamardproducts}
compared to those in \cite{McLaughlinNabelek}.

\begin{lemma}
Let $\{\eta_n\}_{n\in\Integer}$ be a sequence of nonzero complex numbers for which
there exist positive constants $N$ and $C$ such that $|\eta_n|\geq C|n|$ for $|n|\geq N$. 
For all $z\in\Complex$, let $\mathfrak{N}(z)\geq0$ be the smallest integer such that $|\eta_n|>2|z|$ for $|n|>\mathfrak{N}(z)$. 
Then $\mathfrak{N}(z)\leq D|z|$, where
\be
	D=\max\left\{\frac{2}{C},\frac{2N}m\right\},
\ee
\label{l:PA0}
with $m = \min\nolimits_{n\in\Integer}|\eta_n|\ne0$.
\end{lemma}

\begin{proof}
We separate this proof into three parts. 
Firstly, we consider the case when $2|z|<m$. 
It is obvious that $|\eta_n|>2|z|$ for all $n$. 
Thus, $\mathfrak{N}(z)=0\leq D|z|$. 
Secondly, we consider the case when $m\leq2|z|\leq M$, where $M=\min\{|\eta_N|, |\eta_{-N}|\}$ 
Then we have $\mathfrak{N}(z)\leq N$. Hence,
 \be
 \mathfrak{N}(z)
 \leq \frac{2N}{m}|z|.
 \ee
Finally, we consider the case when $M\leq2|z|$. 
It is easy to see that $\mathfrak{N}(z)\geq N$, so that according to the definition of sequence $\{\eta_n\}$, we have
 \be
 \min\{|\eta_{\mathfrak{N}(z)}|,|\eta_{-\mathfrak{N}(z)}|\} \geq C|\mathfrak{N}(z)|.
 \ee
Now we need to prove $\min\{|\eta_{\mathfrak{N}(z)}|,|\eta_{-\mathfrak{N}(z)}|\}\leq 2|z|$. If it is not true, then there will exists an integer $M=\mathfrak{N}(z)-1$ such that $|\eta_n|>2|z|$ for $|n|>M$. This contradicts the fact that $\mathfrak{N}(z)$ is the smallest one. These two inequalities give us $\mathfrak{N}(z)\leq {2}|z|/C$.
\end{proof}

\begin{remark}
If for the complex sequence $\{\eta_n\}_{n\in\Integer}$ there exists an $\=n$ such that $\eta_{\=n}=0$, 
one can consider the sequence $\{\eta_n\}_{n\in\Integer,n\neq\=n}$. 
Following the above methods, the same conclusions can be obtained.
\end{remark}

\begin{lemma}
\label{PA}
Let $\{\eta_n\}_{n\in\Integer}$ 
satisfy the same hypotheses as in Lemma~\ref{l:PA0},
and let $D_n$ with $n\in\Z$ be a family of open discs of radius $R>0$ centered at $c_n$ with $c_0=0$, such that $\eta_n\in\{z\in\C:|z-c_n|\leq R/2\}$. 
Let $N'\in\Natural$ be such that the discs $\overline{D}_n$ are disjoint for $n>N'$. 
If the function $P(z)$ is defined by the canonical product 
\be
	P(z)=e^{Az+B}\prod_{n\in\Z}\bigg(1-\frac{z}{\eta_n}\bigg)\,\e^{\frac{z}{\eta_n}}
\ee
where $A$ and $B$ are constants,
then 
\be
	P^{-1}\in\A_2\left(\C\setminus\bigcup_n\overline{D}_n\right).
\label{e:canonicalproduct}
\ee
\end{lemma}
 
\begin{proof}
Let $z\in\C\setminus\bigcup D_n$. 
Decompose the canonical product \eqref{e:canonicalproduct} into three parts as follows:
\begin{equation*}
 	P(z):=\mathrm{e}^{Az+B}P_1(z)P_2(z),
\end{equation*}
where
\be
    P_1(z)=\prod_{|n|\leq\mathfrak{N}(z)}\left(1-\frac{z}{\eta_n}\right)\mathrm{e}^{\frac{z}{\eta_n}},\qquad
 	P_2(z)=\prod_{|n|>\mathfrak{N}(z)}\left(1-\frac{z}{\eta_n}\right)\mathrm{e}^{\frac{z}{\eta_n}}.
\ee
For the first term $\e^{-(Az+B)}$, we can always find a constant $C_1$ such that $|\e^{-(Az+B)}|\leq C_1\e^{|z|}$. 
So it is clear that $\e^{-(Az+B)}\in\A_1(\C\setminus\bigcup D_n).$
By the reverse triangle inequality,
\be
 	|z-\eta_n|=|z-c_n-\eta_n+c_n|\geq||z-c_n|-|\eta_n-c_n||.
\ee
Together with $|z-c_n|\geq R$ and $|\eta_n-c_n|\leq R/2$, we have $|z-c_n|-|\eta_n-c_n|>0$ and $|z-\eta_n|\geq R/2$,
so that 
\be
 \prod_{|n|\leq\mathfrak{N}(z)}\left|1-\frac{z}{\eta_n}\right|=\prod_{|n|\leq\mathfrak{N}(z)}\frac{1}{|\eta_n|}|\eta_n-z|\geq 	|\prod_{|n|\leq\mathfrak{N}(z)}\frac{R}{2|\eta_n|}\geq R^{2\mathfrak{N}(z)}(2|\eta_{\mathfrak{N}(z)}|)^{2\mathfrak{N}(z)}.
\ee
We can always find positive constants $C_2$ and $M_1$ so that
\be
 	\bigg|\e^{-\sum_{n\leq\mathfrak{N}(z)}\frac{z}{\eta_n}}\bigg|
 	=
 	\bigg|\e^{z(-\sum_{n\leq\mathfrak{N}(z)}\frac{1}{\eta_n})}\bigg|
 	\leq M_1\e^{C_2|z|}.
\ee
So we can bound $|P_1|^{-1}$ as
\vspace*{-0.6ex}
\begin{multline}
	|P_1(z)|^{-1}\leq M_1(2\mathfrak{N}(z)R^{-1})^{2\mathfrak{N}(z)}\e^{C_2|z|}\leq M_1(4R^{-1}|z|)^{2\mathfrak{N}(z)}\e^{C_2|z|}\\
\leq M_1\e^{2D|z|[\log(4R^{-1})+\log|z|]}e^{C_2|z|}\leq M_2\e^{C_3|z|^2},
\end{multline}
 where $C_3$ and $M_2$ are some positive contants.
 
Define a set $S=\{z\in\C: 2|z|\leq\min\{|\eta_N|,|\eta_{-N}|\}\}\setminus\bigcup D_n$,
 so that $S$ is a compact set on which $P_2^{-1}(z)$ is piecewise polynomial with a finite number of domains of definition. Hence, there must exists $M_3\geq 1$ such that $|P_2^{-1}(z)|\leq M_3$ for $z\in S$. Now let $z\in\C\setminus\bigcup D_n$ such that $2|z|>\min\{|\eta_N|,|\eta_{-N}|\}$, which also implies that $\mathfrak{N}(z)\geq N$. 
Taking the principal branch of the logarithm, we then obtain
\be
 \log \,P_2(z) = \sum_{|n|>\mathfrak{N}(z)}\left(\frac{z}{\eta_n}+\log\bigg(1-\frac{z}{\eta_n}\bigg)\right)=\sum_{|n|>\mathfrak{N}(z)}\frac{z}{\eta_n}+\sum_{|n|>\mathfrak{N}(z)}\sum_{m=1}^{\infty}\left(-\frac{z^m}{m\eta_n^m}\right).
\ee
Therefore,
\begin{multline}
|\log\,P_2(z)| = \Bigg| \sum_{|n|>\mathfrak{N}(z)}\sum_{m=2}^\infty \frac{z^m}{m\eta_n^m} \Bigg|
 	\le \sum_{|n|>\mathfrak{N}(z)}\sum_{m=2}^{\infty}\frac{|z|^m}{|\eta_n|^m}\\
 	= \sum_{|n|>\mathfrak{N}(z)}\frac{|z|^2}{|\eta_n|^2}\left(1-\frac{|z|}{|\eta_n|}\right)^{-1}
 	\le 2|z|^2\sum_{|n|>\mathfrak{N}(z)}\frac{1}{|\eta_n|^2}
 	\le \frac{2}{C^2}|z|^2\sum_{|n|>\mathfrak{N}(z)}\frac{1}{|n|^2}\leq\frac{2\pi^2}{3C^2}|z|^2,
\end{multline}
which implies 
\be
 |P_2(z)|=\e^{\Re(\log P_2(z))}\geq\e^{-\frac{2\pi^2}{3C^2}|z|^2}.
\ee
Thus,
\be
 |P_2(z)|^{-1}|\leq\e^{\frac{2\pi^2}{3C^2}|z|^2}.
\ee
Therefore, $|P_2(z)|^{-1}|\leq M\e^{\frac{2\pi^2}{3C^2}|z|^2}$ holds for all $z\in\C\setminus\bigcup D_n$. 
Thus we have proved that 
$P^{-1}\in\A_2(\C\setminus\bigcup\overline{D}_n)$, 
QED.
\end{proof}

\subsection{Alternative sets of Dirichlet eigenvalues}
\label{a:alternative}

One of the main results of this work is that it is sufficient to use only one set of Dirichlet eigenfunctions
in order to reconstruct the potential $q(x)$.
This is in contrast with the finite-genus formalism, in which two sets of Dirichlet eigenvalues are introduced 
and used \cite{MaAblowitz,McLaughlinOverman}.
Here we show that the second set of Dirichlet eigenfunctions would work equally well for the purposes of inverse spectral theory.
In a nutshell, the result follows because both choices have the correct asymptotics as $z\to\infty$.
This is in contrast with the Dirichlet eigenvalues of the original problem. 

Recall that~\cite{MaAblowitz,McLaughlinOverman} introduce a second set of Dirichlet eigenvalues, 
which \cite{MaAblowitz} refer to as the ``second auxiliary spectrum''.  
This auxiliary Dirichlet spectrum is still defined by \eqref{e:dirichlet}, 
except that ``BC$_{{\rm Dir},x_o}$'' now denotes the following modified Dirichlet boundary conditions with base point $x_0$:
\be
v_1(x_0)+\i v_2(x_0)=v_1(x_0+L)+\i v_2(x_0+L)=0.
\label{e:DirichletBC2}
\ee
Let $\{\mu_j(x_o)\}_{j\in\mathbb{Z}}$ and $\{\check\mu_j(x_o)\}_{j\in\mathbb{Z}}$ 
be the set of Dirichlet eigenvalues with base point $x_o$
associated with the BC~\eqref{e:Dirbcs} and~\eqref{e:DirichletBC2},
respectively. 
Then (e.g., see Theorem~2.2 in \cite{McLaughlinOverman})
\bse
\label{e:trace}
\begin{gather}
q(x) + q^*(x) = 2 \sum_{j\in\mathbb{Z}} (\z_{2j} + \z_{2j+1} - 2\mu_j(x))\,,
\label{e:trace1}
\\
q(x) - q^*(x) = -2{\rm i} \sum_{j\in\mathbb{Z}} (\z_{2j} + \z_{2j+1} - 2\check\mu_j(x))\,.
\label{e:trace2}
\end{gather}
\ese

Next we show that this second set of Dirichlet eigenvalues could also be used instead of the original one
to formulate a suitable Riemann-Hilbert problem.
To this end, we introduce the following modified similarity transformation:
\be
\check{Y}=\check{U}Y\check{U}^{-1}\,,\qquad
\check{U}=U\e^{-\frac{\i\pi}{4}\sigma_3}\,.
\ee
The modified solutions $\check{Y}$ satisfy the ODE
$\check{Y}_x = \check{U}(-\i z\sigma_3 + Q)\check{U}^{-1}\,\tilde{v}$.
It is easy to see, however, that, 
similarly to what happens with the original similarity transformation~\eqref{e:Ytildedef},
under this similarity transformation the new Floquet discriminant $\check{\D}$ and Floquet multiplier $\check{\rho}$ 
coincide with the original ones, $\D$ and $\rho$, respectively. 
As a result, the main spectrum $\{\zeta_j\}$ remains unchanged.
On the other hand, it is also straightforward to see that the auxiliary Dirichlet spectrum defined above 
consists of the eigenvalues $\check{\mu}$ satisfying
\be
\check{y}_{12}(L,\check{\mu})=0.
\ee
Moreover, all the properties stated in Theorem~\ref{spectraprop} 
can also be proved to hold for the auxiliary Dirichlet spectrum. 
Therefore, we can define the auxiliary spectral data as the set 
\be
\check{S}(q)=\{E_{2k},E_{2k+1},\check{\gamma}_k,\check{\sigma}_k\},
\ee
where $\check{\gamma}_k$ and $\check{\sigma}_k$ are defined similarly to those in Definition \ref{defgamma} and \ref{specdata}.  Standard Bloch-Floquet theory enables us to define an alternative set of Bloch-Floquet solutions $\check{\psi}^\pm$, 
similarly to \eqref{BF}, except with $\~y_{ij}$ replaced by $\check{y}_{ij}$.  
Importantly, the new RHP $\check{\Phi}$ has  the following asymptotic behavior 
\be
\check\Phi(x,z) = \check{U}\,(I+O(1/z))\,\check{B}(z)\,,
\ee
where $\check{B}(z)$ is defined in \eqref{B} with $\mu_n$ replaced by $\check{\mu}_n$.
Consequently, by following the same procedure described in section~\ref{s:asymptotics}, 
Section~\ref{s:inverse} and section~\ref{s:time}, 
it is clear that using the alternative definition of the Dirichlet eigenfunction would also allow one to construct a Riemann-Hilbert problem with a unique solution, from which we could obtain the reconstruction formula.

In contrast, if one replaces the BC~\eqref{e:Dirbcs} or~\eqref{e:DirichletBC2} with 
a more standard set of Dirichlet BC, e.g., 
\be
v_1(x_o) = v_1(x_o+L) = 0\,,
\label{e:BC_Dir_no}
\ee
the resulting Bloch-Floquet solutions would not allow one to define an effective RHP.
To see this, note first that the Dirichlet spectrum associated with~\eqref{e:BC_Dir_no} corresponds to the zeros of $y_{12}(L,z)$.
Now suppose that one replaces $\~y_{ij}(x,z)$ with~$y_{ij}(x,z)$ in the definition of the Bloch-Floquet solutions
in~\eqref{BF}. 
Let us denote the corresponding functions by by $\mathring{\psi}^\pm(x,z)$.
Following the same procedure used in Appendix \ref{asymofpsipm}, 
and omitting the details for brevity,
one then obtains the asymptotic behavior
of $\mathring\psi^-(x,z)$ and $\mathring\psi^+(x,z)$ as $z\to\infty$ as
\bse
\begin{align}
&\mathring\psi^-(x,z) = \begin{cases}
  \e^{-izx}\left(\begin{array}{c}1-\frac{1}{2iz}\int_0^x|q(t)|^2dt+O(1/z^2)\\-\frac{1}{2iz}q^*(x)+O(1/z^2)\end{array}\right),
&z\in\C^+,
\\
  \e^{izx}\left(\begin{array}{c}\frac{q(x)}{q(0)}+O(1/z)\\\frac{2iz}{q(0)}+O(1)\end{array}\right),
&z\in\C^-\,,
\end{cases}
\\
&\mathring\psi^+(x,z) = \begin{cases}
e^{izx}\left(\begin{array}{c}\frac{q(x)}{q(0)}+O(1/z)\\\frac{2iz}{q(0)}+O(1)\end{array}\right),
&z\in\C^+,
\\
e^{-izx}\left(\begin{array}{c}1+O(1/z)\\-\frac{1}{2iz}q^*(x)+O(1/z^2)\end{array}\right),
&z\in\C^-.
\end{cases}
\end{align}
\ese
If we then follow the same steps as in sections~\ref{s:asymptotics} and~\ref{s:inverse}, 
i.e., we define a new matrix Bloch-Floquet solution $\mathring\Psi(x,z)$ by replacing $\psi^\pm(x,z)$ with $\mathring\psi^\pm(x,z)$ in~\eqref{e:defPsi} 
and we define a new matrix $\mathring\Phi(x,z)$ by replacing $\Psi(x,z)$ with $\mathring\Psi(x,z)$ in~\eqref{Phi}, 
we finally get a new RHP with the asymptotic condition
\be
\label{asymhatPhi}
\mathring\Phi(x,z) = \begin{pmatrix} 1 & {q(x)}/{q(0)}\\0 &{2iz}/{q(0)} \end{pmatrix} (I+O(z^{-1}))\mathring B(z),
\ee
where $\mathring B(z)$ is defined in \eqref{B} with all the $\mu_j$s replaced by the new Dirichlet spectrum. 
But since the leading-order term in \eqref{asymhatPhi} contains the potential $q(x)$, 
which is precisely the unknown for the inverse problem,
the matrix $\mathring\Phi(x,z)$ obtained from the standard Dirichlet BC~\eqref{e:BC_Dir_no} 
cannot be used to construct an effective RHP.

\subsection{Explicit solution of the inverse problem for genus-zero potentials}
\label{s:genus0}

As a simple illustration of the formalism presented in sections~\ref{s:asymptotics} and~\ref{s:inverse}, 
here we present in detail the direct and inverse spectral theory for genus-zero potentials,
which also serves as an example of the fact that only one set of Dirichlet eigenvalues is needed for the inverse problem.

Let $q(x) = A\,\e^{\i\alpha}$, with $\alpha\in\Real$ and where we can take $A>0$ without loss of generality, 
owing to the scale and phase invariances of the NLS equation.
Since $q(x)$ is independent of $x$, the scattering problem can be solved exactly.
Let us choose the eigenvector matrix of $X = - \i z\sigma_3 + Q$ as
\be
W(z) = I - \i\sigma_3Q/(z+\lambda),
\ee
so that $X\,W = W\,(-\i\lambda)\sigma_3$, 
where $\lambda = \lambda(z) = (z^2-A^2)^{1/2}$.
Explicitly, we define the complex square root so that:
(i) its branch cut is $[-A,A]$;
(ii) $(z^2-A^2)^{1/2}>0$ for all $z\in\Real\setminus[-A,A]$;
(iii) on the cut, $(z^2-A^2)^{1/2}$ is continuous from above.
We then have
\bse
\begin{gather}
Y(x,z) = W(z)\,\e^{-\i\lambda x \sigma_3}W^{-1}(z)\,,
\label{e:Ygenus0}
\\
M(z) = Y(L,z) = 
\bm
\begin{pmatrix}
  \frac{\e^{-\i\l L}(\l+z)\left(A^2-\e^{2\i\l L}(z-\l)^2\right)}{2 A^2 \l}  & \frac{\e^{\i\alpha}A\sin(\l L)}{\l} \\ 
  \frac{\e^{-\i\alpha}A\sin(\l L)}{\l} & -\frac{\e^{-\i\l L}(\l+z)\left((z-\l)^2-A^2\e^{2\i\l L}\right)}{2 A^2 \l}
\end{pmatrix},
\em
\end{gather}
\ese
implying
$\D(z) = \cos(\lambda L)$.
Therefore, the Lax spectrum is the set of values of $z\in\Complex$ such that $\lambda\in\Real$, i.e.,
$\Sigma(\L) = (-\infty,-A)\cup(A,\infty)$, which implies that there is only one spectral gap.
Moreover, 
\bse
\begin{gather}
	\~y_{11}(L,z) = \cos(\l L)+\frac{\sin(\l L)}{\l}A\cos\alpha,
	\quad
	\~y_{12}(L,z) = \frac{\sin(\lambda L)}\lambda\,(z+ A\sin\alpha)\,,
	\\
	\~y_{21}(L,z) = \frac{\sin(\lambda L)}\lambda\,( A\sin\alpha-z)\,,
	\quad
	\~y_{22}(L,z)=\cos(\l L) - \frac{\sin(\l L)}\l A\,\cos\alpha\,.
\end{gather}
\ese
\bm It is then clear that the Dirichlet eigenvalues are the points $z_n$ such that $\l L=n\pi$ for $n\in\Integer\setminus\{0\}$. 
In addition, there is \em
also a single Dirichlet eigenvalue at $\mu_0 = - A\,\sin\alpha$.
For simplicity we will drop the subscript ``0'' below.
Moreover, 
$\rho(z) = \e^{\i\lambda L}$.
Finally, it is straightforward to verify that 
$\~y_{22}(L,\mu) = \e^{-AL\cos\alpha}$,
and that $\rho(\mu) = \e^{-AL|\cos\alpha|}$.
Therefore we have that, for all $n\in\Integer$,
$\nu_o = 1$ when $\alpha \in ((2n-\half)\pi,(2n+\half)\pi)$,
$\nu_o = -1$ when $\alpha \in ((2n+\half)\pi,(2n+\frac32)\pi)$
and $\nu_o = 0$ when $\alpha = (n+\frac12)\,\pi$.

Suppose $\nu_o = -1$, which implies $\alpha\in(\frac{\pi}{2},\frac{3\pi}{2})$.  
The cases $\nu_o = 0$ and $\nu_o=1$ are entirely analogous.
\bm 
Note that $-2\i\l(\mu)\,\e^{-\i\l(\mu)x} = 2 A \cos\alpha \big(\sinh(A x \cos\alpha)-\cosh(A x \cos\alpha)\big)$.
We also have 
$f^-(z) = \mu-z$, $f^+(z) = 1$ and $f^0(z) = \~y_{12}(L,z)/(f^-(z)f^+(z))=-\sin(\l L)/\l$,
\em
which imply
\bse
\begin{gather}
B(z) = \begin{cases}
	\displaystyle \bm \i b(z) \em \diag(\mu-z,1),&z\in\Complex^+,
	\\[0ex]
	\displaystyle \bm b(z) \em \diag(1,\mu-z),&z\in\Complex^-,
\end{cases}
\label{e:Bgenus0}
\end{gather}
\bm where
\begin{gather}
b(z)=\frac{\e^{\i\pi/4}}{(z^2-A^2)^{1/4}}\,.
\end{gather}
\em
\ese
\bm
For all $z\in \Sigma(\L)$, we have $\rho_+(z)=\rho_-^{-1}(z)$. 
Recall that for $z\in(-A,A)$, we have $\l_+(z)=-\l_-(z)$, where as usual the subscripts $\pm$ denote the non-tangential limits.
For $z\in\R$, from \eqref{e:defPsi} and \eqref{BF}, there is a switch between the first and second columns of $\Psi$ in different half planes, 
which cancels the jump of $\rho$ as $z\in\Sigma(\L)$, implying that $\Psi_+(x,z) = \Psi_-(x,z)$ for all $z\in\Sigma(\L)$, and $\Psi_+(x,z) = \Psi_-(x,z) \sigma_1$ for all $z\in(-A,A)$. 
Thus, we obtain the jump matrix as
\be\label{VR}
V(x,z) = \Phi^{-1}_-(x,z)\Phi_+(x,z) = 
\begin{cases}
    \e^{-\i zx\hat{\sigma}_3}(B_-^{-1}B_+)
  =  \i  \frac{ b_+(z) }{ b_-(z) } 
      \begin{pmatrix} {\mu-z} & 0\\ 0 & \frac{1}{\mu-z} \end{pmatrix}\,,
      \qquad z\in(-\infty,-A)\cup(A,\infty)\,,\\
       \e^{-\i zx\hat{\sigma}_3}(B_-^{-1}\sigma_1B_+)
  =  \i \frac{ b_+(z) }{ b_-(z) }\e^{-\i zx\hat{\sigma}_3}  \sigma_1\,,
      \qquad z\in(-A,A)\,.\\
    \end{cases}
\ee
From the above definition of $\lambda(z)$, we also have
\vspace*{-1ex}
\be\label{4throotjump}
\frac{ b_+(z) }{ b_-(z) } = \begin{cases}
    -\i, & z\in(-\infty,-A)\,,\\
    -1, & z\in (-A,A)\,,\\
    \i, & z\in(A,\infty)\,.
\end{cases}
\ee
\em
Hence, the $2\times2$ matrix-valued function $\Phi(x,z)$ defined by~\eqref{Phi} solves the following RHP:
\begin{RHP}
\label{RHPgenus0}
Find a $2\times2$ matrix-valued function $\Phi(x,z)$ such that
\vspace*{-1ex}
\begin{enumerate}
\advance\itemsep-4pt
\item 
		$\Phi(x,z)$ is a holomorphic function of $z$ for $z\in\C\setminus\R$.
\item 
		$\Phi^\pm(x,z)$ are continuous functions of $z$ for $z\in\R\setminus\{A, -A\}$, and have at worst quartic root singularities on $\{A,-A\}$.
		\item 
		$\Phi^\pm(x,z)$ satisfies the jump relation 
		\be
		\Phi^+(x,z)=\Phi^-(x,z)V(x,z),\qquad z\in\Real\,,
		\ee
        with 
    \be
    \label{e:Vg0}
    V(z)=\begin{cases}
	\displaystyle\begin{pmatrix} \mu-z & 0 \\ 0 &{1}/({\mu-z})\end{pmatrix}, &z\in(-\infty,-A),
	\\[2ex]
	\begin{pmatrix} 0 & -\i\e^{-2\i zx} \\ -\i\e^{2\i zx} & 0 \end{pmatrix},&z\in(-A,A),
	\\[2ex]
	\displaystyle\begin{pmatrix}z-\mu & 0 \\ 0 &1/(z-\mu)\end{pmatrix}, &z\in(A,\infty).
    \end{cases}
    \ee
\item 
    $\Phi(x,z)$ has the asymptoic behavior $\Phi(x,z) = U\,(I+O(1/z))\,B(z)$ as $z\to\infty$,
        with $U$ as in~\eqref{e:Udef} and $B(z)$ as in~\eqref{e:Bgenus0}.
\item 
		There exist positive constants $c$ and $M$ such that $|\phi_{ij}(x,z)|\leq M\e^{c|z|^2}$ for all $z\in \mathcal{D}$.
	\end{enumerate}
\end{RHP}

\bm
In the special case $\alpha= s\pi/2$, with $s=\pm1$,
the potential $q(x) = s\i A$ is purely imaginary, and the above formalism simplifies considerably.
Indeed, in this case, the movable Dirichlet eigenvalue is $\mu_0 = \mp A$.
Thus, $f^0(z) = \~y_{12}(L,z)$ and
$f^\pm(z)\equiv1$, 
implying $b(z) = \e^{i\pi/4} [(z + s A)/(z - s A)]^{1/4}$.
The jump matrix in the RHP~\ref{RHPgenus0} then simplifies to
\be
V(z) = \begin{pmatrix} 0 & -s\i\,\e^{-2\i zx} \\ -s\i\,\e^{2\i zx}& 0 \end{pmatrix}, \qquad z\in(- A, A)\,,
\ee
while $V(z) \equiv I$ for $z\in\Real$\,.
In this case, it is easy to solve the RHP explicitly, by employing the transformation
\be
\Phi^{(2)}(x,z) = \e^{-\i\pi/4} U^{-1}\,\Phi(x,z) \,\e^{\i (\l-z)x\sigma_3}\,.
\ee
motivated by the fact that $\lim_{z\to\infty} b(z) =\e^{\i\pi/4}$.
Then, $\Phi^{(2)}(x,z)$ satisfies a RHP with jump matrix $-\i\sigma_1$ on $(-A,A)$,
quartic root singularities at $z=\pm A$ and normalization $\lim_{z\to\infty}\Phi^{(2)}(x,z) = I$. 
And the solution of this RHP can be immediately written as (e.g., see~\cite{Deift1998})
\be
\Phi^{(2)}(z)=\frac1{2}\begin{pmatrix}
    \gamma(z) + \gamma^{-1}(z) & \gamma^{-1}(z)-\gamma(z)\\
    \gamma^{-1}(z)-\gamma(z) & \gamma(z) + \gamma^{-1}(z)
\end{pmatrix},
\ee
where $\gamma(z) = [(z-sA)/(z+sA)]^{1/4}$.
Then the reconstruction formula yields 
\be
q(x) = \lim_{z\to\infty}\i z [\sigma_3,\Phi^{(2)}(x,z)]=s\i A\,.
\ee
\em

\input references

\end{document}

%% file: references.tex
\addcontentsline{toc}{section}{References}
\makeatletter
\def\@biblabel#1{#1.}
\def\doibase{http://dx.doi.org/}
\def\reftitle#1{``#1''}
\def\booktitle#1{\textit{#1}}
\def\href#1#2{#2}
\makeatother